\documentclass[11pt,reqno]{amsart}
\usepackage[T1]{fontenc}
\usepackage{lmodern}
\usepackage[expansion=false]{microtype}
\usepackage[margin=1in]{geometry}
\usepackage{amsmath,amssymb,amsthm,mathtools,mathrsfs,booktabs,enumitem}
\usepackage[hidelinks]{hyperref}
\usepackage{zref-clever}
\zcsetup{cap,abbrev=false,nameinlink=false,rangesep={\textendash},rangetopair=false}
\zcRefTypeSetup{equation}{reffont=\upshape}
\NewDocumentCommand{\zeqref}{m}{\textup{\hbox{\normalfont\zcref[noname]{#1}}}}
\AddToHook{env/proposition/begin}{\zcsetup{countertype={theorem=proposition}}}
\AddToHook{env/lemma/begin}{\zcsetup{countertype={theorem=lemma}}}
\AddToHook{env/corollary/begin}{\zcsetup{countertype={theorem=corollary}}}
\AddToHook{env/definition/begin}{\zcsetup{countertype={theorem=definition}}}
\AddToHook{env/remark/begin}{\zcsetup{countertype={theorem=remark}}}
\usepackage{cite}
\usepackage{comment}
\numberwithin{equation}{section}
\newtheorem{theorem}{Theorem}[section]
\newtheorem{proposition}[theorem]{Proposition}
\newtheorem{lemma}[theorem]{Lemma}
\newtheorem{corollary}[theorem]{Corollary}
\theoremstyle{definition}

\theoremstyle{remark}

\newtheorem{example}[theorem]{Example}
\newcommand{\E}{\mathbb E}
\newcommand{\Prob}{\mathbb P}
\newcommand{\R}{\mathbb R}
\newcommand{\one}{\mathbf1}
\newcommand{\dd}{\,\mathrm d}
\newcommand{\cP}{\mathcal P}

\DeclareMathOperator{\Var}{Var}
\DeclareMathOperator{\supp}{supp}
\DeclareMathOperator{\sech}{sech}
\DeclareMathOperator{\arctanh}{arctanh}
\DeclareMathOperator{\Law}{Law}
\DeclareMathOperator{\Cov}{Cov}
\DeclareMathOperator{\tr}{tr}

\DeclareMathOperator{\PD}{PD}

\allowdisplaybreaks[2]
\title[Overlap structure of mixed even models]
{Overlap structure of mixed even $p$-spin models}
\author{P. M. Aronow}
\author{Patrick Lopatto}
\date{\today}
\hypersetup{pdftitle={Overlap structure of mixed even p-spin models},bookmarksdepth=2}
\begin{document}

\begin{abstract}
We identify the limiting overlap structure of mixed even $p$-spin
models, including the Sherrington--Kirkpatrick model, at every fixed
deterministic external field.  In particular, we show that the disorder-averaged distribution of
the overlap, or its absolute value at zero field, converges to the
Parisi measure. We identify the full limiting array using the overlap
array of the associated Ruelle probability cascade.  We also establish the corresponding
Ghirlanda--Guerra identities and
determine the limiting law of the quenched overlap distribution.
For the zero-field Sherrington--Kirkpatrick model, we additionally prove
temperature chaos at every pair of distinct nonnegative inverse
temperatures.
\end{abstract}
\maketitle
\setcounter{tocdepth}{1}
\tableofcontents

\section{Introduction}\label{sec:introduction}
The Sherrington--Kirkpatrick (SK) model is central to the theory of mean-field spin glasses and, more broadly, to our modern understanding of disordered systems \cite{mezard1987spin,charbonneau2023spin}. 
It describes a system of spins
taking values $\pm1$ with independent Gaussian interactions between
every pair \cite{SherringtonKirkpatrick1975}. These interactions favor
either alignment or opposition, creating competing constraints whose
collective effect is particularly pronounced at low temperature.

Parisi proposed that the equilibrium properties of the SK model are 
described by an order parameter that takes its value in the space of probability measures 
\cite{Parisi1979}. His replica-symmetry-breaking ansatz leads to a variational formula for the
limiting free energy, which has now been put on rigorous footing \cite{Talagrand,chatterjee2026michel}. The minimizer of the variational problem is unique and yields the relevant order parameter, which is known as the \emph{Parisi measure} \cite{ACunique}. Parisi's prediction about the order
parameter consists of two distinct parts; for simplicity, we state them in the case where there is no external field. The first states that below the critical
temperature, the Parisi measure should exhibit full replica symmetry
breaking, meaning that it has support equal to a nontrivial interval. The second concerns an observable called the \emph{overlap}, defined as the normalized inner product of two configurations sampled
independently from the Gibbs measure with the disorder fixed. 
It states that the 
disorder-averaged distribution of the absolute overlap should converge
to the Parisi measure as the number of spins tends to infinity. Together,
these assertions describe equilibrium configurations with a continuum
of possible overlap values.

The first assertion was recently established through the entire low-temperature phase \cite{LopattoFRSB, Lopatto2026ExternalField}. However, the second has remained open throughout this regime, despite significant effort. The strongest previous results in this direction are Panchenko's theorems on overlap identities and
ultrametricity \cite{PanGG,PanUM}, which identify the overlap structure
for generic mixed $p$-spin models, a class which does not include the SK model. The contribution of this paper is to
establish the second part of Parisi's prediction for the SK model and, more
generally, for mixed even $p$-spin models with exponentially
summable coefficients at every fixed external field. We proceed by first proving the
full Ghirlanda--Guerra identities, then use them to identify the joint limiting
overlaps of arbitrarily many replicas with the hierarchy prescribed
by the Parisi measure. We also determine the limiting law of the
random overlap distribution obtained when the disorder is held fixed.

\subsection{Model and limiting objects}
For each even integer $p\ge2$, let
$(g^{(p)}_{i_1\cdots i_p})$ be independent standard Gaussian variables,
which are further independent  across the interaction orders $p$. Throughout, we consider sequences $(b_p)$ and functions $\xi(q)$ satisfying 
\begin{equation}\label{eq:mixture}
 \xi(q)=\sum_{p\in2\mathbb N}b_p^2q^p\not\equiv0,
 \qquad b_p\ge0,\qquad \sum_{p\in2\mathbb N}2^pb_p^2<\infty.
\end{equation}
For $m \in [-1,1]^N$, define
\begin{equation}\label{eq:hamiltonian}
 H_N(m)=\sum_{p\in2\mathbb N}\frac{b_p}{N^{(p-1)/2}}
  \sum_{i_1,\ldots,i_p\le N}g^{(p)}_{i_1\cdots i_p}
  m_{i_1}\cdots m_{i_p},\quad
 R(m,n)=\frac{m\cdot n}{N},\quad q(m)=R(m,m).
\end{equation}
The process $H_N$ has a continuous version and covariance
$\E H_N(m)H_N(n)=N\xi(R(m,n))$, even for infinite mixtures, as described in Appendix~\zcref[noname]{app:source}. On
$\{-1,1\}^N$ equipped with counting measure, set
\[
 G_{N,h}(\sigma)=Z_{N,h}^{-1}
       \exp\left(H_N(\sigma)+h\sum_{i=1}^N\sigma_i\right),
 \qquad F_{N,h}=N^{-1}\log Z_{N,h},\qquad p_N=\E F_{N,h},
\]
with the constant $Z_{N,h}$ chosen to make $G_{N,h}$ a probability measure. We write $\langle\cdot\rangle$ for expectation with respect to an
i.i.d.\ sequence $\sigma^1,\sigma^2,\ldots$ sampled from the Gibbs
measure $G_{N,h}$ and call each sample $\sigma^\ell$ a
\emph{replica}. For example,
\[
 \bigl\langle f(R_{12})\bigr\rangle
 =
 \sum_{\sigma^1,\sigma^2}
 G_{N,h}(\sigma^1)G_{N,h}(\sigma^2)
 f\bigl(R(\sigma^1,\sigma^2)\bigr).
\]
We also write $R_{ij}=R(\sigma^i,\sigma^j)$ for the overlap between
replicas $i$ and $j$, and set
\begin{equation}\label{eq:unsigned-convention}
 S_{ij}=\begin{cases}R_{ij},&h\ne0,\\ |R_{ij}|,&h=0.\end{cases}
\end{equation}

Let $\mu=\mu_{\xi,h}$ be the unique minimizer of the Parisi
functional and let $Q=\max\supp\mu$. We recall the functional and uniqueness property in
Section~\zcref[noname]{sec:scalar}. We define  the off-diagonal array space by
\[
 \Omega=[-1,1]^{\mathcal I},\qquad
 \mathcal I=\{(i,j):1\le i<j<\infty\},
\]
and equip it with its compact product topology. By convergence of array laws, we mean weak convergence of the
probability measures on $\Omega$. Since $\mathcal I$ is countable, this is equivalent to
convergence of expectations for every bounded continuous function
depending on only finitely many array entries.  For each $N$, let $\mathcal L_N$ denote the probability law on
$\Omega$ defined by
\begin{equation}\label{e:annealed}
 \mathcal L_N(A)
 =
 \E\Bigl[
 G_{N,h}^{\otimes\infty}
 \bigl((R_{ij})_{i<j}\in A\bigr)
 \Bigr],
 \qquad A\subseteq\Omega\ \text{Borel},
\end{equation}
where the expectation is over the Gaussian disorder. We call any weak
subsequential limit of $(\mathcal L_N)$ a \emph{physical overlap law}.

We first record two properties of the finite-volume overlap arrays.
Complete each array to a symmetric array with $R_{ii}=1$. For every $m\ge1$,
the matrix $(R_{ij})_{i,j\le m}$ is positive semidefinite. Indeed,
for any $c_1,\ldots,c_m\in\mathbb R$,
\[
 \sum_{i,j=1}^m c_i c_j R_{ij}
 =
 \frac1N\left\|\sum_{i=1}^m c_i\sigma^i\right\|^2
 \ge0.
\]
This finite matrix is also invariant in law under any relabeling of the indices, because 
$\sigma^1,\sigma^2,\ldots$ are i.i.d.\ conditional on the disorder.
Both properties pass to a weak subsequential limit. Positive
semidefiniteness of each finite principal submatrix is a closed
condition on finitely many overlap coordinates, and invariance under
a finite permutation of the replica labels is preserved by weak
convergence.

It follows that the completed limiting array satisfies the sufficient
conditions for the Dovbysh--Sudakov representation
\cite{dovbysh1984gram,DS}, which gives a random probability measure
$\eta$ on $\ell^2\times\mathbb R_+$ such that, conditional on $\eta$,
\[
 R_{ij}
 =v^i\cdot v^j+a_i\delta_{ij},
 \qquad a_i\ge0,
\]
where the pairs $
 (v^1,a_1),(v^2,a_2),\ldots$ 
are i.i.d.\ samples from $\eta$. Since $R_{ii}=1$, we have
$a_i=1-\|v^i\|^2$, so the $\ell^2$-marginal $G$ of $\eta$ is supported
on the unit ball of $\ell^2$. We call $G$ a \emph{physical directing
measure}.

Write $\mathcal L_\mu$ for the law of the nonnegative Ruelle
probability cascade (RPC) overlap array
$\mathcal U=(\mathcal U_{ij})_{i<j}$ associated with $\mu$ \cite{Ruelle1987}.
This means that $\mathcal L_\mu$ is the joint law of all overlaps among an
infinite sequence of replicas sampled from the cascade, and its
two-replica marginal is $
 \mathcal U_{12}\sim\mu$. 
For finitely supported $\mu$, this is the usual finite-level RPC, and RPCs for 
general $\mu$ are obtained as limits of such finite-level cascades. We recall the construction of RPCs in 
Section~\zcref[noname]{sec:rpc}.
Given independent signs
$(\varepsilon_i)_{i\ge1}$ uniform on $\{-1, 1\}$ and independent of $\mathcal U$, define
\[
 \mathcal R^{\mu,h}_{ij}=
 \begin{cases}
 \mathcal U_{ij},&h\ne0,\\
 \varepsilon_i\varepsilon_j\mathcal U_{ij},&h=0,
 \end{cases}\qquad i<j.
\]

\subsection{Identities, full array, and quenched law}
We begin with the identities we use to identify the full overlap array.
For a limiting signed array, let $\mathcal R_n$ be the completion of
$\sigma(R_{ij}:1\le i<j\le n)$ with respect to the law of
the limiting array. 
\begin{theorem}\label{thm:gg}
Assume \zeqref{eq:mixture} and fix $h\in\R$. Every physical
overlap law satisfies, for every $n\ge2$,
\begin{equation}\label{eq:full-gg-conditional}
 \Law(S_{1,n+1}\mid\mathcal R_n)
 =\frac1n\mu+\frac1n\sum_{\ell=2}^n\delta_{S_{1\ell}}
 \quad\text{almost surely}.
\end{equation}
Equivalently, for every bounded $\mathcal R_n$-measurable $F$ and
bounded Borel $\psi:[0,1]\to\R$,
\begin{equation}\label{eq:full-gg}
 \E\!\left[F\,\psi(S_{1,n+1})\right]
 =
 \frac1n\,\E[F]\int \psi(q)\,\mu(dq)
 +
 \frac1n\sum_{\ell=2}^n
 \E\!\left[F\,\psi(S_{1\ell})\right].
\end{equation}
\end{theorem}

The next theorem identifies the limiting overlap distribution.
\begin{theorem}\label{thm:array}
Assume \zeqref{eq:mixture} and fix $h\in\R$.  With $\mathcal L_N$
defined by \zeqref{e:annealed},
\[
 \mathcal L_N
 \Longrightarrow
 \Law\bigl((\mathcal R^{\mu_{\xi,h},h}_{ij})_{i<j}\bigr).
\]
In particular, $(S_{ij})_{i<j}$ converges to $\mathcal L_{\mu_{\xi,h}}$.
Every physical directing measure is almost surely supported on the
sphere of squared radius $Q=\max\supp\mu_{\xi,h}$.
\end{theorem}

In particular, the ordinary one-overlap marginal converges to $\mu$ when
$h\ne0$. At zero field the absolute marginal converges to $\mu$, and
the signed marginal converges to the symmetrized measure 
$\tfrac12(\mathrm{id})_\#\mu+\tfrac12(-\mathrm{id})_\#\mu$.

To state the quenched result, write $J$ for the disorder and let
\begin{equation}\label{eq:quenched-array-definition}
 \mathscr Q_N^J
 =G_{N,h}^{\otimes\infty}\circ(R_{ij})_{i<j}^{-1}
 \in\cP(\Omega),\qquad
 P_N^J=G_{N,h}^{\otimes2}\circ S_{12}^{-1}.
\end{equation}
Here, $\mathscr Q_N^J$ is the conditional law of the entire overlap
array when the disorder is held fixed, $P_N^J$ is the
conditional law of the single overlap $S_{12}$, and $\cP(\Omega)$ denotes the space of Borel probability measures on $\Omega$. Both are random
probability measures because they depend on the random disorder $J$.

For a general Parisi measure $\mu$, the associated RPC is a random
probability measure $G_\mu^{\rm RPC}$ on the sphere
\[
 \mathbb S_Q=\{v\in\ell^2:\|v\|^2=Q\}, \qquad Q = \max \supp \mu,
\]
whose i.i.d.\ samples have  
overlap-array law  $\mathcal L_\mu$.  Conditional on $G_\mu^{\rm RPC}$, sample
\[
 v^1,v^2,\ldots\overset{\mathrm{i.i.d.}}{\sim}G_\mu^{\rm RPC}
\]
and form their overlap array.  We denote its conditional law by
\[
 \mathscr Q_\mu^{\rm RPC}
 =
 (G_\mu^{\rm RPC})^{\otimes\infty}
 \circ
 \bigl((v^i)_{i\ge1}\mapsto(v^i\cdot v^j)_{i<j}\bigr)^{-1}.
\]
Then $\mathscr Q_\mu^{\rm RPC}$ is a random probability measure on
overlap arrays, while
\[
 \E\bigl[\mathscr Q_\mu^{\rm RPC}\bigr]
 =\mathcal L_\mu
\]
is the deterministic annealed RPC array law, where the expectation is taken in the randomness from $G_\mu^{\rm RPC}$. 

We now specialize this construction to $
 \mu=\mu_{\xi,h}$. 
If $h\ne0$, we set $
 \mathscr Q_{\xi,h}^{\rm RPC}=\mathscr Q_\mu^{\rm RPC}$. 
If $h=0$, conditional on $G_\mu^{\rm RPC}$, we sample independent
fair signs $\varepsilon_i\in\{-1,1\}$, which are also independent of the
vectors $(v^i)_{i\ge1}$, and replace the nonnegative overlap
$v^i\cdot v^j$ by $\varepsilon_i\varepsilon_j\,v^i\cdot v^j$ in the definition of $\mathscr Q_\mu^{\rm RPC}$. 
In this case, we let $\mathscr Q_{\xi,0}^{\rm RPC}$ denote the
conditional law of the resulting signed overlap array given
$G_\mu^{\rm RPC}$.  In both cases,
$\mathscr Q_{\xi,h}^{\rm RPC}$ is random because it depends on
$G_\mu^{\rm RPC}$. The conditional law averages over the sampled
vectors and, at zero field, the independent signs. Averaging this conditional
law over the random measure $G_\mu^{\rm RPC}$ gives the annealed RPC
array law, 
\[
 \E\bigl[\mathscr Q_{\xi,h}^{\rm RPC}\bigr]
 =
 \Law\bigl((\mathcal R^{\mu_{\xi,h},h}_{ij})_{i<j}\bigr).
\]
In the following theorem, we write $\langle \cdot \rangle_J$ for expectation with respect to the Gibbs measure with fixed disorder $J$. 

\begin{theorem}\label{thm:quenched}
Under \zeqref{eq:mixture}, for every fixed $h$,
\begin{equation}\label{eq:quenched-convergence}
 \mathscr Q_N^J\ \Longrightarrow\ \mathscr Q_{\xi,h}^{\rm RPC}
 \quad\text{in distribution as random elements of }\cP(\Omega).
\end{equation}
For every bounded Borel $\psi:[-1,1]\to\R$ whose discontinuity set
has $\mu$-measure zero,
\begin{equation}\label{eq:one-third}
 \lim_{N\to\infty}\Var_J\langle\psi(S_{12})\rangle_J
 =\frac13\left[\int\psi^2\,d\mu-
                     \left(\int\psi\,d\mu\right)^2\right].
\end{equation}
If $\mu=\delta_q$, then $P_N^J$ converges weakly in probability to
$\delta_q$. If $\mu$ is not a point mass, $P_N^J$ does not converge
weakly in probability to any deterministic probability measure.
\end{theorem}

\subsection{Hierarchical geometry and cavity invariance}

We now state a few corollaries of the preceding theorems. 
The first concerns the ultrametric structure of the overlaps. 
\begin{corollary}\label{cor:geometry}
Assume \zeqref{eq:mixture}, and let $\mathcal L$ be any physical overlap
law. Then, under $\mathcal L$,
\begin{equation}\label{eq:geometry}
S_{12}\ge\min(S_{13},S_{23})\quad\text{almost surely}.
\end{equation}
\end{corollary}
Let $\mathcal L$ be a physical overlap law and let
$G$ be a directing measure in its Dovbysh--Sudakov representation.
By Theorem~\zcref[noname]{thm:array}, $G$ is supported on
\[
 \mathbb S_Q=\{v\in\ell^2:\|v\|^2=Q\},
 \qquad Q=\max\supp\mu.
\]
At zero field, after identifying each vector $v$ with $-v$, the
preceding corollary shows that three independent samples from $G$
satisfy the ultrametric inequality for the quotient distance
\[
 d_\pm([v],[w])
 =
 \min\{\|v-w\|,\|v+w\|\}
\]
almost surely.

The next corollary describes how replicas are partitioned by an overlap threshold and identifies the distribution of the resulting block frequencies.

\begin{corollary}
\label{cor:pd}
Fix $0\le q<Q$ and set $
 a=\mu([0,q])$. 
For a limiting overlap array, define a relation on the replica labels by
\[
 i\sim_q j
 \quad\Longleftrightarrow\quad
 i=j\ \text{or}\ S_{ij}>q.
\]
By Corollary~\zcref[noname]{cor:geometry}, this is almost surely an equivalence
relation.  Let $\Pi_q$ denote the resulting random partition of
$\mathbb N$.

Suppose first that $0<a<1$.  If the first $n$ replicas occupy
$k$ distinct blocks of $\Pi_q$, with block sizes
$n_1,\ldots,n_k$, then, conditional on the overlaps among the first
$n$ replicas,
\[
 \Prob\bigl(n+1\text{ joins block }j\mid\mathcal R_n\bigr)
 =
 \frac{n_j-a}{n},
 \qquad 1\le j\le k,
\]
and
\[
 \Prob\bigl(n+1\text{ starts a new block}\mid\mathcal R_n\bigr)
 =
 \frac{ka}{n}.
\]
Consequently, almost surely, the limit
\[
 p_B=\lim_{n\to\infty}\frac{|B\cap\{1,\ldots,n\}|}{n}
\]
exists for every block $B$ of $\Pi_q$. If $(p_k)_{k\ge1}$ is the
sequence of these frequencies arranged in decreasing order, then $
 (p_k)_{k\ge1}\sim\PD(a,0)$, 
where $
 \PD(a,0)$ denotes the Poisson--Dirichlet distribution.
If $a=0$, then $\Pi_q$ consists almost surely of a single block.
\end{corollary}
At zero field, $S_{ij}=|R_{ij}|$, so two states belong to the same
$q$-cluster when their absolute overlap exceeds $q$. In particular,
states related by the global sign change $v\mapsto-v$ belong to the
same cluster.

In physics, the cavity method studies the effect of adding a single
spin. Summing over that spin reweights the thermodynamic states by
$\cosh(h+Y(v))$. Although this can change their relative weights by
order-one factors, the following corollary shows that the joint law
of the weights and overlap geometry is preserved. This gives a
consistency relation for the limiting Gibbs measure.

\begin{corollary}\label{cor:cavity-invariance}
Let $G$ be a random directing measure for the limiting overlap array
in Theorem~\zcref[noname]{thm:array}.  We identify two directing measures $G_1$ and $G_2$ if there exists
a surjective linear isometry
\[
 U:\overline{\operatorname{span}}(\supp G_1)
   \longrightarrow
   \overline{\operatorname{span}}(\supp G_2)
\]
such that $G_2=U_\#G_1$.  
Conditional on $G$, let $Y$ be centered Gaussian with
covariance $\E[Y(v)Y(w)\mid G]=\kappa(v\cdot w)$, where
$\kappa=\xi'$. Then
\begin{equation}\label{eq:cavity-invariance}
 \mathcal C_{\xi,h}(G)(dv)
 =\frac{\cosh(h+Y(v))}{G\cosh(h+Y)}G(dv)
\end{equation}
has the same law as $G$ on the space of isometry classes defined above.
\end{corollary}

\subsection{Temperature chaos} 

The methods we develop for the previous results also allow us to establish temperature chaos for the SK model, as well as a temperature-chaos criterion for mixed even $p$-spin models. Since this result is a conceptual detour, we defer a fuller discussion to Section~\zcref[noname]{sec:temperature-chaos}.

\subsection{Proof ideas}\label{subsec:proof-ideas}

\subsubsection{Previous approaches}
Panchenko has used differentiability of the Parisi formula to identify
overlap moments corresponding to the interaction terms present in
the Hamiltonian \cite{PanDiff}. Further concentration and Gaussian integration
by parts give the corresponding Ghirlanda--Guerra identities at each order with a nonzero coefficient
\cite{GhirlandaGuerra1998,PanGG}. In suitably generic models, these powers form
a determining family, yielding the overlap marginal and the full
identities \cite{PanchenkoBook}. 

Further, Panchenko's ultrametricity theorem converts the full identities into
a geometric description \cite{PanUM}. His proof first derives invariance
under reweightings determined by sampled overlaps, then uses it to
duplicate vectors while controlling their overlaps. He concludes by showing that repeated
duplication of a configuration violating ultrametricity would
contradict Hilbert space geometry. The theorem applies to overlap
arrays independently of their origin, leaving the derivation of the
full identities for a particular Hamiltonian as a separate problem.
Given the conclusion of Panchenko's ultrametricity result, the Baffioni--Rosati reconstruction theorem then identifies the
resulting hierarchy as an RPC determined by its one-overlap marginal
\cite{BaffioniRosati}.

Related results describe the analogous hierarchy in finite systems.
Talagrand has constructed pure states under certain assumptions on the overlap
identities and marginal \cite{TalagrandPureStates}, and Jagannath showed that,
when the averaged overlap distribution converges and the approximate
Ghirlanda--Guerra identities hold, the Gibbs measures admit nested clusters
at suitable overlap levels, with weights converging in distribution to
RPC weights \cite{Jagannath2017}.

Cavity methods have also been used to describe the organization of spins within states. 
For suitably perturbed mixed $p$-spin models, Panchenko derived
cavity consistency equations under which the full joint overlap law
determines the limiting joint distribution of spins across sites
and replicas \cite{PanchenkoSpinDistributions}.
Auffinger and Jagannath represented spin distributions of generic
models through branching Parisi diffusions and derived multiscale
TAP equations \cite{AuffingerJagannathSpinDistributions}.
They also proved TAP equations within pure states of generic
models when the Parisi measure has an atom at the maximum of its support 
\cite{AuffingerJagannathTAP}.
Chen, Panchenko, and Subag developed a generalized TAP free energy
for magnetization vectors \cite{CPSGeneralizedTAP,CPSII}.
Here finite-replica comparisons yield conditional magnetization
laws for the unperturbed model before its overlap structure
has been identified. We then use these laws to establish the
full Ghirlanda--Guerra identities and determine the overlap array.

\subsubsection{Ideas of the present proof}
Panchenko's approach obtains overlap information by differentiating
the free energy with respect to the interaction coefficients. For
generic mixtures, the available interaction orders determine the
overlap distribution. For SK, this argument gives only its second
moment and the corresponding Ghirlanda--Guerra identities. We obtain
the missing information by studying spins in several replicas
jointly. We add auxiliary terms involving products of these spins
and compare the resulting free energies with scalar Parisi
expressions. These comparisons determine spin moments under the
original Gibbs measure while keeping its interaction coefficients
fixed.

We apply these moment estimates to the effect of removing and
restoring one spin. Deleting the spin and every interaction
containing it leaves a Gibbs measure on the remaining spins.
Restoring the deleted interactions and summing over the spin
reweights the remaining configurations by its partition function.
In the limiting representation, the restored spin in a state
represented by a vector $v$ has conditional mean
$M(v)=\tanh(h+Y(v))$. Conditional on the directing measure $G$
of the reduced systems, $Y$ is a centered Gaussian field with
covariance
\[
 \E[Y(v)Y(w)\mid G]=\xi'(v\cdot w).
\]
For vectors sampled from the reweighted measure, the moment
estimates identify the law of one magnetization and the joint law
of two magnetizations at vectors of equal norm as the corresponding
scalar Parisi laws, conditional on all sampled overlaps.

We next use these conditional magnetization laws to determine
the distribution of a new overlap given the earlier overlaps.
The two-magnetization law applies to vectors of equal norm,
so we first reduce the problem to identifying the overlap law
after restricting and normalizing the directing measure on
each sphere of positive mass. For replicas
$v^1,\ldots,v^{n+1}$ sampled from the resulting reweighted
measure, set $Y_j=Y(v^j)$ and $M_j=\tanh(h+Y_j)$.
For a bounded function $F$ of the signed overlaps among
the first $n$ replicas, Gaussian integration by parts gives
\[
 \begin{aligned}
 &n\E[F\,\xi'(R_{1,n+1})f(Y_1)M_{n+1}]\\
 &\qquad=\E\Bigl[F\Bigl(
   \xi'(\|v^1\|^2)f'(Y_1)-Y_1f(Y_1)
   +f(Y_1)\sum_{\ell=1}^n\xi'(R_{1\ell})M_\ell
   \Bigr)\Bigr],
 \end{aligned}
\]
where $R_{11}=\|v^1\|^2$ in this display and
$f\in C^1(\R)$, with $f,f'$ growing at most exponentially.  The conditional magnetization
laws express the field and magnetization terms as scalar
Parisi expectations. Scalar integration by parts then turns
the display into an identity involving only overlaps.
Because $F$ is arbitrary, this identity determines a
conditional expectation given the earlier overlaps.
The function of the new overlap appearing in that expectation
depends on our choice of $f$.

We choose $f$ to obtain enough information to determine
conditional probabilities at every overlap threshold.
Write $S_{ij}=R_{ij}$ at nonzero field and
$S_{ij}=|R_{ij}|$ at zero field. Taking $f(x)=e^{tx}$
at nonzero field and $f(x)=\sinh(tx)$ at zero field gives
scalar expressions whose asymptotics as $t\to\infty$
distinguish different overlap levels. Comparing identities
for three replicas first proves
$S_{23}\ge\min(S_{12},S_{13})$ and, at zero field, shows
that three nonzero signed overlaps have positive product.
We then use exponentials of linear combinations of two
fields to identify the conditional tail probabilities
\[
 \Prob(S_{1,n+1}>q\mid\mathcal R_n)
 =\frac1n\mu((q,1])
  +\frac1n\sum_{\ell=2}^n\one_{\{S_{1\ell}>q\}},
\]
where $\mathcal R_n$ is generated by the signed overlaps
among the first $n$ replicas. We first establish this
formula when $0<S_{1j}<q$ for some $2\le j\le n$.
An additional replica, weighted toward positive overlaps
below $q$, supplies this condition in the enlarged sample.
Exchangeability relates the error in the formula for the
original sample to its error for the enlarged sample.
The choice of weights makes this error vanish.
Together with the treatment of the remaining thresholds
and endpoint atoms, this identifies the overlap law on
each spherical restriction.

This identification also shows that the original directing
measure is supported on a single sphere, giving the full
Ghirlanda--Guerra identities for the original overlap array.
We then use Panchenko's ultrametricity theorem and the
Baffioni--Rosati reconstruction theorem to identify its
full law. This also determines the distribution of the
random quenched sampling measure.

\subsubsection{Outline of the proof}
Section~\zcref[noname]{sec:scalar} describes the scalar Parisi
magnetizations used in the cavity comparisons. Writing
$u(q,x)=\partial_x\Phi(q,x)$, the process $u(q,X_q)$ along
the Parisi diffusion is a martingale. Two diffusions coupled
to coincide up to time $r$ and evolve independently thereafter
describe magnetizations with overlap $r$. We establish
integration-by-parts identities for these laws, the variational
identity $\E[u(q,X_q)^2]=q$ on $\supp\mu$, and the integrability
estimates needed in the subsequent comparisons.

Section~\zcref[noname]{sec:sources} considers replicas in one
group or two equal-sized groups, with prescribed overlaps
within and between groups. We add to their joint Hamiltonian
a parameter times a sum over sites of a chosen product of
replica spins. Gaussian interpolation bounds the resulting
constrained pressure by scalar Parisi expressions. Varying
the parameter controls the site average of this spin product,
whose predicted value is obtained by replacing each selected
spin by the scalar magnetization for its group. For any fixed
error tolerance, sufficiently narrow overlap constraints make
the disorder-averaged Gibbs probability of violating this
prediction exponentially small in $N$.

Section~\zcref[noname]{sec:cavity} applies these estimates
after removing and restoring one spin. Exchangeability of
sites converts the site averages into moments at the restored
site. Auxiliary replicas selected to lie near a directing
vector then determine its cavity magnetization law. The
argument retains arbitrary additional overlaps in the
conditioning. Conditional on the full overlap array, the law
of one magnetization depends only on its vector's squared
norm, and the joint law at two vectors of equal norm depends
only on that norm and their mutual overlap.
Appendix~\zcref[noname]{app:cavity} controls the coefficient
changes caused by deleting a site, so these conclusions
apply at the original mixture and field.

Section~\zcref[noname]{sec:shells} reduces the remaining
argument to normalized restrictions of the directing measure
to spheres. If $Q=0$, a vanishing even overlap moment forces
every off-diagonal overlap to vanish. Otherwise, the scalar
and cavity identities exclude squared norms below $Q$ and
require every remaining squared norm $s$ to satisfy
$\E[u(s,X_s)^2]=s$. We use this to show that the possible squared norms form a deterministic countable set, and that the conditional cavity laws remain valid after
restricting and normalizing on each sphere of positive mass. 

Section~\zcref[noname]{sec:regression} identifies the overlap
law on each such sphere. We compare Gaussian integration
by parts for its cavity fields with the scalar Parisi
identities. Differentiating the normalizing constant
introduces an additional replica, whose overlaps can be
included in the conditional magnetization laws. Exponential
functions of the terminal scalar field produce increasing
functions of overlap whose differences from their terminal
values decay at strictly ordered rates. These rates rule
out violations of ultrametricity for absolute overlaps and,
at zero field, exclude $R_{12}R_{13}R_{23}<0$.

The next step determines conditional overlap tail probabilities
on a sphere of squared radius $S$. For $0<q<S$ carrying no
mass under either $\mu$ or the unknown marginal of $S_{12}$,
we prove 
\[
 \Prob\!\left(
   S_{1,n+1}>q\mid(R_{ij})_{i<j\le n}
 \right)
 =
 \frac{\mu((q,1])
   +\sum_{j=2}^n\one_{\{S_{1j}>q\}}}{n}.
\]
Initially, the argument requires $0<S_{1j}<q$ for some
$2\le j\le n$. We apply the cavity and scalar identities
with exponential functions of the fields at replicas $1$
and $j$. Their coefficients are chosen so that the limiting
magnetizations distinguish overlaps above and below $q$.
Subtracting the identities and using the sign and ultrametricity
relations gives the displayed conditional probability.

To remove the assumption that some $S_{1j}$ lies in $(0,q)$,
we introduce an additional replica and weight the identities
toward overlaps in that interval. Exchangeability relates
the error in the desired formula to a weighted error for
the enlarged sample. The preceding argument makes this
error zero when the added overlap lies in $(0,q)$, and the
weights suppress the remaining overlaps. This proves the
threshold identity whenever $\mu((0,q))>0$, including
the case $n=1$ needed to identify the marginal. If the
positive support of $\mu$ starts above zero, the one-replica
identities exclude overlap mass in the intervening interval.
Conditional first moments of $\xi'(S_{1,n+1})$ and total
mass then determine the unresolved atoms at zero and the
bottom of the positive support. Thus the full conditional
Ghirlanda--Guerra identities hold on each sphere.

Returning to the original overlap coordinate, two replicas
sampled from any normalized restriction to a sphere satisfy
$S_{12}\sim\mu$. If such a sphere had squared radius $S>Q$,
two sufficiently nearby vectors would have overlap greater
than $Q$ with positive probability, contradicting
$Q=\max\supp\mu$. Hence the directing measure is supported on
the sphere of squared radius $Q$, and the end of
Section~\zcref[noname]{sec:regression} proves
Theorem~\zcref[noname]{thm:gg}.

Section~\zcref[noname]{sec:rpc} identifies the overlap-array
law using Panchenko's ultrametricity theorem and the
Baffioni--Rosati reconstruction theorem. At zero field,
these results are applied to the Gram array $(R_{ij}^2)$,
which determines the law of the absolute overlaps. The sign
relations proved above and invariance under independently
reversing each replica recover the signed array. This proves
Theorem~\zcref[noname]{thm:array}.

Section~\zcref[noname]{sec:quenched} proves convergence in
distribution of the random quenched sampling measure.
Products of conditional Gibbs expectations can be written
using disjoint groups of replicas, so the annealed array
limit determines all joint moments needed to identify this
distribution. The Ghirlanda--Guerra identities with three
and four replicas give the variance identity in
Theorem~\zcref[noname]{thm:quenched}, and the RPC identification
gives the Poisson--Dirichlet cluster frequencies.

Finally, joint convergence of overlaps and the restored spin
represents the limiting measure after restoration as the
cavity reweighting of the measure before restoration. Their
common overlap-array law identifies their distributions up
to isometry, proving invariance under the one-site cavity map.

\subsection*{Acknowledgments} 
This paper was written
by the authors with the assistance of large language models, which were used to
suggest mathematical arguments, help with drafting and revision, and write code
for computational verification. P.\ L.\ was partially supported by NSF grant DMS-2450004.

\section{Scalar Parisi identities}\label{sec:scalar}
We collect the scalar diffusion identities, magnetization bounds, and
Gaussian integration-by-parts formula used in the cavity comparisons.
For the representation of spin distributions through branching
Parisi diffusions in generic models; see
\cite{AuffingerJagannathSpinDistributions}.
Throughout Sections~\zcref[noname]{sec:scalar}--\zcref[noname]{sec:regression},
$\xi$ and $h$ are fixed, and
\[
 \kappa(q)=\xi'(q),\qquad \mu=\mu_{\xi,h},\qquad
 \alpha(q)=\mu([0,q]),\qquad Q=\max\supp\mu.
\]

\subsection{The Parisi formula}
Given a probability measure $\nu$ on $[0,1]$, set
$\alpha_\nu(q)=\nu([0,q])$. The Parisi PDE is 
\begin{equation}\label{e:parisipde}
 \partial_q\Phi_\nu
 =-\frac{\kappa'(q)}2\bigl((\Phi_\nu)_{xx}
                  +\alpha_\nu(q)(\Phi_\nu)_x^2\bigr),
 \qquad \Phi_\nu(1,x)=\log(2\cosh x).
\end{equation}
For a distribution function with jumps, the PDE is understood in
integrated form and its time derivative exists almost everywhere.
Also, solutions with step distribution functions can be obtained using the Cole--Hopf transform, with approximation in $L^1([0,1])$ yielding the solution for general $\nu$ 
\cite[Proposition~2]{ACunique}.

In the next lemma, we record some results on the regularity of solutions to this PDE from \cite{ACunique} and two  immediate consequences. 
\begin{lemma}\label{lem:scalar-spatial-regularity}
Let $\Phi$ be the solution to \zeqref{e:parisipde} associated with $\nu$, and set
\[
 u(s,x)=\partial_x\Phi(s,x),
 \qquad
 V(s,x)=\partial_{xx}\Phi(s,x).
\]
The spatial derivatives
\[
 \partial_x^j\Phi,\qquad 0\le j\le4,
\]
exist and are continuous on $[0,1]\times\mathbb R$.  Further, there exists a constant $C_\xi>0$ such that
\begin{equation}\label{e:pdebounds}
 |u(s,x)|\le1,
 \qquad
 \frac{C_\xi}{\cosh^2 x}\le V(s,x)\le1,
 \qquad
 |\partial_x^3\Phi(s,x)|\le4
\end{equation}
for every $(s,x)\in[0,1]\times\mathbb R$, and
\begin{equation}\label{e:pdebounds2}
 V(s,x)>0
 \qquad\text{and}\qquad
 |u(s,x)|<1
\end{equation}
for every finite $x$.
\end{lemma}

\begin{proof}
The existence and continuity of the $\partial_x^j\Phi$, as well as the bounds in \zeqref{e:pdebounds}, are given in
\cite[Proposition~2]{ACunique}.  For \zeqref{e:pdebounds2}, the lower bound on $V$ gives
$V(s,x)>0$ for every finite $x$.  Hence, for fixed $s$, the function
$x\mapsto u(s,x)$ is strictly increasing.  Further, $u$  cannot attain either $1$ or $-1$ at any $x\in \R$, since
if $u(s,x_0)=1$, then $V(s,x)>0$  would give $u(s,x)>1$ for
$x>x_0$, contradicting $|u| \le 1$, and similarly if $u(s,x_0)=-1$. Hence
$|u(s,x)|<1$, as desired.
\end{proof}

The Parisi functional and its minimum are
\begin{equation}\label{eq:parisi-functional}
 \mathcal P_{\xi,h}(\nu)=\Phi_\nu(0,h)
       -\frac12\int_0^1q\kappa'(q)\alpha_\nu(q)\,dq,
 \qquad P=\min_{\nu\in\cP([0,1])}\mathcal P_{\xi,h}(\nu),
\end{equation}
where $\cP([0,1])$ denotes the set of Borel probability measures on $[0,1]$. It is known that $\lim_{N \rightarrow \infty} p_N = P$ \cite{Talagrand, PanchenkoParisi2014} and that the 
minimizer $\mu$ in the second expression of \zeqref{eq:parisi-functional} is unique \cite{ACunique}. We suppress its subscript
in $\Phi=\Phi_\mu$ and $\alpha = \alpha_\mu$, and we write
\begin{equation}\label{eq:parisi-pde}
 u=\Phi_x,\qquad V=\Phi_{xx},\qquad
 \Phi_q=-\frac{\kappa'}2(V+\alpha u^2).
\end{equation}
Let $B$ be a standard real Brownian motion. Expectations involving
only the scalar diffusion below are taken over its Brownian law.
The associated diffusion and its second magnetization moment are
\begin{equation}\label{eq:parisi-sde}
 dX_q=\kappa'(q)\alpha(q)u(q,X_q)\,dq
          +\sqrt{\kappa'(q)}\,dB_q,\qquad X_0=h,
 \qquad \Gamma(q)=\E u(q,X_q)^2.
\end{equation}
For $q\in[0,1]$ and $x\in\mathbb R$, we write $\E_{q,x}$ for
expectation with respect to the same diffusion started from $X_q=x$. To explain the terminology, we remark that $u(q,x)=\partial_x\Phi(q,x)$ is the effective
single-spin magnetization when the Parisi diffusion is at $x$ at time $q$, since $u(q,x)
 =
 \E_{q,x}[\tanh(X_1)]$.

By \cite[Proposition~1]{ChenGT}, we have 
\begin{equation}\label{eq:contact}
 \Gamma(q)=q,\qquad \Gamma'(q)\le1
       \quad(q\in\supp\mu),
\end{equation}
interpreted using one-sided derivatives at $0$ and $1$. 
Since $u(1,x)=\tanh x$ and $X_1$ is finite almost surely,
$\Gamma(1)<1$, implying $Q<1$. When $h=0$, we have 
\begin{equation}\label{eq:zero-in-support}
 0\in\supp\mu
\end{equation}
by \cite[Theorem~1(i)]{ACproperties}. 

For $q\ge Q$, we have  $\alpha(q)=1$, and direct computation gives
\begin{equation}\label{eq:scalar-terminal}
 \Phi(q,x)=\log(2\cosh x)+\frac{\kappa(1)-\kappa(q)}2,
 \qquad u(q,x)=\tanh x,\qquad q\ge Q.
\end{equation}
For every $q\in[0,1]$, let $\nu_{q,h}$ be the law of $u(q,X_q)$. 
We also require the joint law of two magnetizations whose diffusions
coincide up to a prescribed time and then evolve independently afterward.
For every $r, q \in [0,1]$ such that $0\le r\le q$, let
\[
 (X_s^1,X_s^2)_{0\le s\le q}
\]
be two copies of the Parisi diffusion that follow the same path up to
time $r$ and, conditional on their common value at time $r$, use
independent Brownian increments on $[r,q]$.  We define
$\nu_r^{q,h}$ to be the law of
\begin{equation}\label{e:nurqh}
 \bigl(u(q,X_q^1),u(q,X_q^2)\bigr).
\end{equation}
Each coordinate has marginal law $\nu_{q,h}$.

When $h=0$, we also need negative values of the overlap parameter.
For $-q\le r<0$, define $\nu_r^{q,0}$ as the image of
$\nu_{|r|}^{q,0}$ under 
$(m_1,m_2)\mapsto(m_1,-m_2)$.
The definitions from the positive and negative sides agree
continuously at $r=0$, because if $h=0$, the diffusion starts from $X_0=0$,
and the solution $\Phi$ of the Parisi PDE is even in the spatial variable, so $u(q,\cdot)$
is odd.  

The laws $\nu_{q,h}$ and $\nu_r^{q,h}$ are weakly continuous in
their parameters, including the reflected extension at zero field.
See the coupling argument in the proof of
Lemma~\zcref[noname]{src:replica-coefficient}.

\subsection{Martingales and bounds on the magnetization}
For a continuous distribution function, differentiating
\zeqref{eq:parisi-pde} and applying It\^o's formula gives the following
identities. For general distribution functions they hold in integrated
form by approximation \cite[Lemma~2, Eqs.~(31)--(32)]{ACunique}.
We have 
\begin{align}
 d\,u(q,X_q)&=\sqrt{\kappa'(q)}\,V(q,X_q)\,dB_q,
       \label{eq:martingale-u}\\
 d\,V(q,X_q)&=-\kappa'(q)\alpha(q)V(q,X_q)^2\,dq
       +\sqrt{\kappa'(q)}\,\Phi_{xxx}(q,X_q)\,dB_q.
       \label{eq:martingale-v}
\end{align}
In particular,
\begin{equation}\label{eq:gamma-derivative}
 \Gamma'(q)=\kappa'(q)\E \big[V(q,X_q)^2\big].
\end{equation}

The next lemma quantifies the gap between $V$ and $1-u^2$.
\begin{lemma}\label{lem:slack}
For $q<Q$ and $x\in\R$,
\begin{equation}\label{eq:slack}
 1-u(q,x)^2-V(q,x)
 =\E_{q,x}\int_q^Q\kappa'(s)(1-\alpha(s))V(s,X_s)^2\,ds>0.
\end{equation}
For $q\ge Q$, $V(q,x)=1-u(q,x)^2$.
\end{lemma}
\begin{proof}
By \zeqref{eq:martingale-u}--\zeqref{eq:martingale-v} and It\^o's formula, the drift term of
$1-u^2-V$  is
$-\kappa'(s)(1-\alpha(s))V^2$. The value of $1-u^2-V$ at $Q$ is zero
by \zeqref{eq:scalar-terminal}. Taking expectations proves the claimed identity.
For $s<Q$, $\alpha(s)<1$ by the definition of $Q$. Also $V>0$ and
$\kappa'(s)>0$ for $s>0$ under \zeqref{eq:mixture}. Hence the integral
is strictly positive whenever $q<Q$. Then \zeqref{eq:scalar-terminal} implies
the remaining assertion.
\end{proof}
We also require a more refined bound on $u(q,x)$. 
\begin{lemma}\label{scalar:tail}
For $q\in[0,1]$ and $x\in\R$,
\begin{equation}\label{eq:sigmoid-shift}
 x-2[\kappa(1)-\kappa(q)]
 \le\arctanh u(q,x)
 \le x+2[\kappa(1)-\kappa(q)].
\end{equation}
Consequently, for constants $c,C>0$ depending only on $\xi$,
\begin{equation}\label{source:eq:sigmoid-bounds}
 1-u(q,x)\le Ce^{-2x}\quad(x\in\R),\qquad
 1-u(q,x)\ge ce^{-2x}\quad(x\ge0).
\end{equation}
Moreover, for every $\lambda>0$,
\[
 \E\exp\left[
   \lambda\sup_{0\le q\le1}|X_q|
 \right]<\infty,
 \qquad
 \E\exp\left[
   \lambda\sup_{0\le q\le1}
   |\arctanh u(q,X_q)|
 \right]<\infty.
\]
\end{lemma}
\begin{proof}
Under $\E_{q,x}$, evolve the diffusion from $X_q=x$ to time $1$.
Writing
\[
 v=\kappa(1)-\kappa(q),\qquad
 Z=\int_q^1\sqrt{\kappa'(s)}\,dB_s,\qquad
 D=\int_q^1\kappa'(s)\alpha(s)u(s,X_s)\,ds,
\]
we have
\[
 X_1=x+Z+D,\qquad Z\sim N(0,v),\qquad |D|\le v,
\]
where the inequality follows because $|\alpha u|\le1$.  Moreover,
\zeqref{eq:martingale-u} and $u(1,x)=\tanh x$ give
\[
 u(q,x)=\E_{q,x}\tanh X_1.
\]
Since $s\mapsto(1+s)^{-1}$ is convex,
\[
 1-u(q,x)
 =2\E_{q,x}\frac1{1+e^{2X_1}}
 \ge
 \frac2{1+\E_{q,x}e^{2X_1}}.
\]
Using $D\le v$ and $\E e^{2Z}=e^{2v}$,
\[
 \E_{q,x}e^{2X_1}
 \le e^{2x+2v}\E e^{2Z}
 =e^{2x+4v},
\]
and hence
\begin{equation}\label{e:oneminusu}
 1-u(q,x)\ge\frac2{1+e^{2x+4v}}.
\end{equation}
Using $1-\tanh y=2/(1+e^{2y})$, we have 
\[
 1-u(q,x)
 =\frac{2}{1+\exp\big(2\arctanh u(q,x)\big)},
\] 
which together with \zeqref{e:oneminusu} gives
$\arctanh u(q,x)\le x+2v$.  Repeating the argument with
$e^{-2X_1}$ gives $\arctanh u(q,x)\ge x-2v$, proving
\zeqref{eq:sigmoid-shift}.  Then using \zeqref{e:oneminusu} with \zeqref{eq:sigmoid-shift} and
$0\le\kappa(1)-\kappa(q)\le\kappa(1)$ gives
\[
 1-u(q,x)\le 2e^{4\kappa(1)}e^{-2x},
 \qquad
 1-u(q,x)\ge
 \frac{2}{1+e^{4\kappa(1)}}e^{-2x}\quad(x\ge0),
\]
which is \zeqref{source:eq:sigmoid-bounds}.

Finally, along the diffusion started from $X_0=h$, the total drift has
absolute value at most $\kappa(1)$, while the martingale part is a
Brownian motion run for time at most $\kappa(1)$.  Thus
$\sup_{q\le1}|X_q|$ has all linear exponential moments.
Equation~\zeqref{eq:sigmoid-shift} then gives the same conclusion for
$\sup_{q\le1}|\arctanh u(q,X_q)|$.
\end{proof}

\subsection{Path density}
We next represent the law of the Parisi diffusion by a density with
respect to the Gaussian path $
 W_q=h+B_{\kappa(q)}$. 
This representation allows us to apply Gaussian integration by parts
under the reference law of $W$ and obtain the identity used later to
relate $X_r$ to the terminal magnetization.  In the path-density
formula, integrals against $\mu$ include any atoms at both endpoints. We use the same convention for branching paths below.

\begin{lemma}\label{scalar:gaussian}
Fix $S\in[Q,1]$, with $S>0$, and set
\[
 W_q=h+B_{\kappa(q)},\qquad 0\le q\le S.
\]
The law of $(X_t)_{0\le t\le S}$ is absolutely continuous with
respect to the law of $(W_t)_{0\le t\le S}$, and
for every bounded Borel measurable
$F:C([0,S])\to\R$,
\begin{equation}\label{e:pathden2}
 \E\bigl[F((X_t)_{t\le S})\bigr]
 =
 \E_W\bigl[F((W_t)_{t\le S})\mathcal D_S(W)\bigr],
\end{equation}
where $\E_W$ denotes expectation under the Gaussian law of $W$ and
\begin{equation}\label{scalar:pathdensity}
 \mathcal D_S(W)=
 \exp\left\{\Phi(S,W_S)
 -\int_{[0,S]}\Phi(t,W_t)\,\mu(dt)\right\}.
\end{equation}
Further, for $0\le r\le S$ and $f\in C^1(\R)$ satisfying
$|f(x)|+|f'(x)|\le Ce^{c|x|}$,
\begin{align}
\begin{split}
 \E[(X_r-h)f(X_S)]
 ={}&\kappa(r)\E[f'(X_S)+\tanh(X_S)f(X_S)]\\
 &-\int_{[0,S]}\kappa(r\wedge t)
          \E[u(t,X_t)f(X_S)]\,\mu(dt).
 \label{eq:scalar-ibp}
\end{split}
\end{align}
\end{lemma}
\begin{proof}
We first identify the density of the Parisi diffusion with respect to
$W$. Realize its Gaussian reference law as
$W_q=h+\int_0^q\sqrt{\kappa'(s)}\,d\widetilde B_s$, which has
the same path law as $h+B_{\kappa(q)}$. In this realization,
\[
 dW_q=\sqrt{\kappa'(q)}\,d\widetilde B_q,
\]
while \zeqref{eq:parisi-sde} has the additional drift
$\kappa'(q)\alpha(q)u(q,X_q)$.  Since $0\le\alpha\le1$ and $|u|\le1$,
\[
 \int_0^S \kappa'(q)\alpha(q)^2u(q,W_q)^2\,dq
 \le \kappa(S)<\infty.
\]
So, Girsanov's theorem \cite[Theorem~5.7 and Corollary~5.9]{Miller2016}
applies, with logarithmic density
\[
 \int_0^S\sqrt{\kappa'(q)}\alpha(q)u(q,W_q)\,d\widetilde B_q
 -\frac12\int_0^S\kappa'(q)\alpha(q)^2u(q,W_q)^2\,dq.
\]
On the other hand, It\^o's formula and the Parisi PDE give
\[
 d\Phi(q,W_q)
 =-\frac12\kappa'(q)\alpha(q)u(q,W_q)^2\,dq
 +\sqrt{\kappa'(q)}u(q,W_q)\,d\widetilde B_q.
\]
Hence the logarithmic density is
\[
 \int_0^S\alpha(q)\,d\Phi(q,W_q).
\]
Since $d\alpha=\mu$, $\alpha(0-)=0$, and $\alpha(S)=1$, Stieltjes
integration by parts gives
\[
 \int_0^S\alpha(q)\,d\Phi(q,W_q)
 =
 \Phi(S,W_S)-\int_{[0,S]}\Phi(q,W_q)\,\mu(dq),
\]
which proves \zeqref{scalar:pathdensity}.  The integral over $[0,S]$
includes any atoms of $\mu$ at the endpoints; in particular, an atom
at zero contributes $-\mu(\{0\})\Phi(0,h)$.

We next derive \zeqref{eq:scalar-ibp}.  By
\zeqref{scalar:pathdensity},
\[
 \E[(X_r-h)f(X_S)]
 =
 \E_W[(W_r-h)\mathcal D_S(W)f(W_S)].
\]
To apply Gaussian integration by parts, first approximate the
$\mu$-integral in \zeqref{scalar:pathdensity} by a finite sum
\[
 \int_{[0,S]}\Phi(t,W_t)\,\mu(dt)
 \approx
 \sum_{j=1}^m a_j\Phi(t_j,W_{t_j}).
\]
The resulting integrand is a continuously differentiable function
of the finite Gaussian vector
\[
 (W_r-h,W_S-h,W_{t_1}-h,\ldots,W_{t_m}-h),
\]
whose covariances satisfy
\[
 \Cov(W_r,W_t)=\kappa(r\wedge t).
\]
The usual finite-dimensional Gaussian integration-by-parts formula
therefore gives two types of terms.  Differentiating
$e^{\Phi(S,W_S)}f(W_S)$ gives
\[
 \kappa(r)\mathcal D_S(W)
 \bigl[f'(W_S)+u(S,W_S)f(W_S)\bigr],
\]
while differentiating the $j$th term in the approximating sum gives
\[
 -a_j\kappa(r\wedge t_j)
 \mathcal D_S(W)u(t_j,W_{t_j})f(W_S).
\]
Passing to the limit in the finite-dimensional identity gives
\[
\begin{aligned}
 \E_W[(W_r-h)\mathcal D_S(W)f(W_S)]
 =&\kappa(r)\E_W\!\left[
 \mathcal D_S(W)\bigl(f'(W_S)+u(S,W_S)f(W_S)\bigr)\right]\\
 &-\int_{[0,S]}\kappa(r\wedge t)
 \E_W\!\left[
 \mathcal D_S(W)u(t,W_t)f(W_S)\right]\,\mu(dt).
\end{aligned}
\]
Using \zeqref{e:pathden2} to replace the
weighted law of $W$ by the law of $X$, and using
$u(S,\cdot)=\tanh$ gives \zeqref{eq:scalar-ibp}. 
The passage from the finite sums to the integral is justified because 
$t\mapsto\Phi(t,W_t)$ and $t\mapsto\kappa(r\wedge t)u(t,W_t)$ are
almost surely continuous, while $|u|\le1$ gives a uniform bound by
$e^{C+C\sup_{t\le S}|W_t|}$ times the prescribed exponential growth of
$f$ and $f'$.  Since $\sup_{t\le S}|W_t|$ has all linear exponential
moments, dominated convergence applies to both the density and the
integration-by-parts terms.
\end{proof}

Finally, we observe that the definition of $\nu_r^{q,h}$ and
\zeqref{eq:martingale-u} imply the useful identity
\begin{equation}\label{eq:branch-product}
 \int z_1z_2\,\nu_r^{S,h}(dz_1,dz_2)=\Gamma(r)
       \quad(0\le r\le S).
\end{equation}
Indeed, the conditional mean of either terminal magnetization $u(S,X_S^i)$ given
$X_r$ is $u(r,X_r)$. At zero field the  version of
\zeqref{eq:branch-product} that accommodates negative $r$ is $\operatorname{sgn}(r)\Gamma(|r|)$.

\section{Finite-replica pressure comparisons}\label{sec:sources}

In this section we prove the estimates used to identify the conditional
cavity laws in \zcref{sec:cavity}. In \zcref{src:replica-groups} we arrange
replicas in groups with prescribed overlaps and choose products of their
spins to characterize scalar magnetization moments. \zcref{src:replica-interpolation}
proves the corresponding upper bound for the constrained pressure. In
\zcref{src:source-suppression-section} we combine this comparison with
Gaussian concentration to show that configurations satisfying the overlap
constraints but violating the moment predictions have exponentially small
probability under the original Gibbs measure.

Throughout this section the covariance and deterministic field are
those of \zeqref{eq:mixture}--\zeqref{eq:hamiltonian}. Put
\[
 L=\kappa(1),\qquad
 B_\xi=\|\kappa'\|_{L^\infty([-1,1])}=\xi''(1)<\infty.
\]
All thermodynamic
comparisons below concern the original Gibbs measure $G_{N,h}$.

\subsection{Equal-sized replica groups and scalar moments}
\label{src:replica-groups}

For the two-group construction, fix an integer $k\ge1$, set $d=2k$,
and index the coordinates of the two groups of $k$ by $(g,j)$, where $g\in\{1,2\}$ and
$1\le j\le k$. Let $(g,j)$ have label $(g-1)k+j$.
For $0\le r\le S\le1$, prescribe the matrix
\[
 D_{ab}=
 \begin{cases}
 1,&a=b,\\
 S,&a\ne b\text{ in the same group},\\
 r,&a,b\text{ in different groups}.
 \end{cases}
\]
For $\eta>0$, the overlap constraint is
\begin{equation}\label{src:replica-shell}
 \mathcal C_\eta(D)=
 \left\{(\sigma^1,\ldots,\sigma^d):
       |R_{ab}-D_{ab}|<\eta\text{ for every }a<b\right\}.
\end{equation}
In the one-group construction, $d=k$ and every off-diagonal entry
of $D$ equals $S$. If $d=1$, the constraint is the entire
configuration space.

Choose integers $a,b\ge0$ with $a+b>0$ and $a,b\le k$ in the
two-group construction, and let
\[
 P_0(\epsilon)=
 \prod_{j=1}^a\epsilon^{1,j}
 \prod_{j=1}^b\epsilon^{2,j},
 \qquad \epsilon\in\{-1,1\}^{2k}.
\]
By convention, empty products are one. In the one-group construction, choose
$1\le a\le k$ and use $P_0(\epsilon)=\prod_{j=1}^a\epsilon^j$.
In either case $P_0$ takes values in $\{-1,1\}$. Define
\begin{align}
 A_N(\boldsymbol\sigma)
 &=\frac1N\sum_{i=1}^N
       P_0(\sigma_i^1,\ldots,\sigma_i^d),\notag\\
 \mathcal F_N^{D,P_0}(\theta)
 &=\frac1N\log
   \sum_{\mathcal C_\eta(D)}
   \exp\left\{
       \sum_{\ell=1}^d
          \left(H_N(\sigma^\ell)+h\sum_{i=1}^N\sigma_i^\ell\right)
       +\theta N A_N(\boldsymbol\sigma)\right\}.
 \label{src:replica-pressure}
\end{align}
Each site contributes the perturbation factor $e^{\theta P_0}$.
The constrained sum uses the counting measure on all spin coordinates,
including those absent from $P_0$.
An empty constrained sum has pressure $-\infty$ and Gibbs mass zero.

To define the coefficient $c$ of $\theta$ in the comparison bound,
construct $d$ solutions of \zeqref{eq:parisi-sde}, all starting
from $h$. In the two-group case, drive all solutions by the same
Brownian motion up to time $r$. Between $r$ and $S$, use one
Brownian motion for each group, shared by all copies in that
group and independent between groups. After $S$, use independent
Brownian motions for the individual copies. At each splitting
time, every continuation starts from the value already reached,
and the new Brownian increments are independent of the past.  At time one, sample signs independently
conditional on the terminal fields, with
\begin{equation}\label{e:termsign}
 \Prob_{\rm sc}(\epsilon^\ell=e\mid X_1^\ell)
 =\frac{e^{eX_1^\ell}}{2\cosh X_1^\ell},
 \qquad e\in\{-1,1\}.
\end{equation}
The one-group construction shares its path through $S$ and then
uses independent continuations. Write
\begin{equation}\label{src:source-coefficient}
 c(P_0;r,S)=\E_{\rm sc}P_0(\epsilon^1,\ldots,\epsilon^d)
\end{equation}
in the two-group case, and $c(P_0;S)$ in the one-group case.
Here $\E_{\rm sc}$ averages only the scalar paths and terminal signs.

\begin{example}
Fix $0\le r\le S\le1$ and consider the moment
\[
 \int z_1^2z_2\,\nu_r^{S,h}(dz_1,dz_2).
\]
Take four replicas, with groups $\{1,2\}$ and $\{3,4\}$, and
prescribe the overlap matrix
\[
 D=
 \begin{pmatrix}
 1&S&r&r\\
 S&1&r&r\\
 r&r&1&S\\
 r&r&S&1
 \end{pmatrix}.
\]
Thus the constraint requires $|R_{ab}-D_{ab}|<\eta$ for
every $a<b$. Choose
\[
 P_0(\epsilon)=\epsilon^1\epsilon^2\epsilon^3,
\]
so the additional term in the joint Hamiltonian is
\[
 \theta\sum_{i=1}^N
 \sigma_i^1\sigma_i^2\sigma_i^3.
\]
Replica $4$ remains in the sum of the four Hamiltonians and
in the constraints involving $R_{14}$, $R_{24}$, and $R_{34}$,
although its spin does not appear in this additional term.

Construct four solutions $X^1,\ldots,X^4$ of
\zeqref{eq:parisi-sde}, starting from $h$. They use common
Brownian increments up to time $r$. Between $r$ and $S$,
the increments are shared within each group and independent
between groups. After $S$, the four solutions use independent
Brownian increments. At each split, the new increments are
independent of the past, and each solution continues from
its current value. In particular,
\[
 X_S^1=X_S^2,\qquad X_S^3=X_S^4.
\]
Sample the terminal signs according to the conditional law
defined above. Given the four paths up to time $S$, these
signs are independent, and the martingale property gives
\[
 \E_{\rm sc}[\epsilon^a\mid\mathcal F_S]
 =u(S,X_S^a),\qquad 1\le a\le4,
\]
where $\mathcal F_S$ is the sigma-field generated by these
paths up to time $S$. Consequently,
\[
 \begin{aligned}
 c(P_0;r,S)
 &=\E_{\rm sc}[\epsilon^1\epsilon^2\epsilon^3]=\E_{\rm sc}\!\left[
      u(S,X_S^1)^2u(S,X_S^3)
    \right]=\int z_1^2z_2\,\nu_r^{S,h}(dz_1,dz_2).
 \end{aligned}
\]

The four-replica interpolation below uses the coefficient
\[
 \gamma(q)=
 \begin{cases}
 \alpha(q)/4,&0\le q<r,\\
 \alpha(q)/2,&r\le q<S,\\
 \alpha(q),&S\le q\le1.
 \end{cases}
\]
The denominators are the numbers of replicas sharing each
scalar field on the respective intervals. With this choice,
the scalar contribution to the interpolation bound at
$\theta=0$ is
\[
 4\Phi(0,h)-2\int_0^1q\kappa'(q)\alpha(q)\,dq
 =4P,
\]
and the comparison proved below gives
\[
 \E\mathcal F_N^{D,P_0}(\theta)
 \le
 4P+\theta\int z_1^2z_2\,\nu_r^{S,h}(dz_1,dz_2)
 +\frac{\theta^2}{2}+C_{4,\xi}\eta^2.
\]
Thus four replicas suffice to obtain this moment as the
coefficient of $\theta$ in the comparison bound.
\end{example}

We now formalize the computation in the previous example. 
\begin{lemma}\label{src:replica-coefficient}
For the monomials just defined,
\[
 c(P_0;r,S)=\nu_r^{S,h}(z_1^az_2^b),\qquad
 c(P_0;S)=\nu_{S,h}(z^a).
\]
For fixed $\xi$, $h$, group size, and monomial $P_0$, the
two-group map $
 (r,S)\mapsto c(P_0;r,S)$ 
is continuous on
$\{(r,S)\in[0,1]^2:0\le r\le S\}$, and the one-group map $S\mapsto c(P_0;S)$
is continuous on $[0,1]$. 
At zero field the two-group coefficient extends continuously to
$-S\le r<0$ by reflection of the second group's terminal signs.
\end{lemma}

\begin{proof}
Let $X^{g,j}$ denote the scalar diffusion associated with replica
$j$ in group $g$. The $k$ diffusions in each group coincide up
to time $S$, so their values at that time satisfy
\[
 Y_g:=X_S^{g,1}=\cdots=X_S^{g,k},\qquad g=1,2.
\]
Conditional on $(Y_1,Y_2)$, the individual continuations after
$S$ and the terminal signs are independent.  Since $u(1,x)=\tanh x$,
\zeqref{eq:martingale-u} and \eqref{e:termsign} give
\[
 \E_{\rm sc}[\epsilon^{g,j}\mid Y_1,Y_2]
 =\E_{\rm sc}[\tanh(X_1^{g,j})\mid Y_g]
 =u(S,Y_g).
\]
Thus
\[
 c(P_0;r,S)
 =\E_{\rm sc}[u(S,Y_1)^a u(S,Y_2)^b]
 =\nu_r^{S,h}(z_1^az_2^b).
\]
For one group, the same argument gives
$c(P_0;S)=\E_{\rm sc}[u(S,X_S)^a]=\nu_{S,h}(z^a)$.

For continuity, let $B^0,B^1,B^2$ be independent Brownian
motions. For each $r$, drive the two scalar diffusions
$X^{r,1},X^{r,2}$ by
\[
 W_q^{r,g}
 =B^0_{q\wedge r}+B_q^g-B_{q\wedge r}^g,
 \qquad g=1,2.
\]
This couples all choices of the splitting time on one
probability space. The drift $\kappa'\alpha u$ is bounded
and spatially Lipschitz with bounds $B_\xi$. For each $g$,
the noise difference between the equations with splitting
times $r$ and $r'$ has quadratic variation
\[
 2\int_{r\wedge r'}^{r\vee r'}\kappa'(q)\,dq
 \le 2B_\xi|r-r'|.
\]
Doob's inequality and Gronwall's inequality therefore give
\[
 \E\max_{g=1,2}\sup_{q\le1}
 |X_q^{r,g}-X_q^{r',g}|^2
 \le C_\xi|r-r'|.
\]
The bounded drift and diffusion coefficient also give
$\E|X_S^{r,g}-X_{S'}^{r,g}|^2\le C_\xi|S-S'|$.
Consequently,
\[
 \E\max_{g=1,2}
 |X_S^{r,g}-X_{S'}^{r',g}|^2
 \le C_\xi\bigl(|r-r'|+|S-S'|\bigr).
\]
Since $u$ is continuous and bounded by one, the products
defining $c(P_0;r,S)$ converge in $L^1$ whenever
$(r',S')\to(r,S)$. This proves continuity on
$\{0\le r\le S\le1\}$, including its boundary. The same
argument with one diffusion proves continuity of $c(P_0;S)$.

At $h=0$, reflection defines
\[
 c(P_0;r,S)=(-1)^b c(P_0;|r|,S),\qquad -S\le r<0.
\]
When $r=0$, the two group magnetizations are independent
and symmetric, so $c(P_0;0,S)=(-1)^b c(P_0;0,S)$.
The extension is therefore continuous across $r=0$.
\end{proof}

\subsection{The finite-replica comparison}
\label{src:replica-interpolation}

\begin{lemma}\label{src:replica-upper}
For either the one-group or two-group construction, every $\theta\in\R$, every $\eta>0$,
and every $N$ for which the constraint $\mathcal C_\eta(D)$ is nonempty,
\begin{equation}\label{src:source-upper}
 \E\mathcal F_N^{D,P_0}(\theta)
 \le dP+\theta c+\frac{\theta^2}{2}
          +\frac{d(d-1)}4 B_\xi\eta^2,
\end{equation}
where $c$ is the corresponding coefficient in
Lemma~\zcref[noname]{src:replica-coefficient}.
At zero field the pair assertion also holds for $-S\le r<0$,
using its reflected coefficient.
\end{lemma}

\begin{proof}
We apply the Gaussian interpolation of
\cite[Lemmas~1--2, Eqs.~(40)--(47)]{PanVector} to compare
the constrained pressure with a recursion involving Gaussian
fields. Our choice of coefficients makes its value at
$\theta=0$, after subtracting the interpolation penalty,
equal to $dP$. We then bound the change caused by
$\theta P_0$ and the error from the overlap tolerance $\eta$.
We first treat step functions $\alpha$ and nonnegative
cross overlaps. The approximation argument and the
zero-field reflection are given at the end.

\emph{The path and recursion.}
For the two-group construction, set
\[
 B_{ab}(q)=
 \begin{cases}
 q,&a=b,\\
 q\wedge S,&a\ne b\text{ in the same group},\\
 q\wedge r,&a,b\text{ in different groups},
 \end{cases}
 \qquad 0\le q\le1.
\]
Let $m(q)$ be the number of replicas sharing a common Gaussian
field increment at level $q$. 
For the two-group construction,
\[
 m(q)=
 \begin{cases}
 2k,&0\le q<r,\\
 k,&r\le q<S,\\
 1,&S\le q\le1.
 \end{cases}
\]
We use a deterministic function $\gamma:[0,1]\to[0,1]$
to specify the exponential weighting in the backward
Gaussian recursion. If $\gamma=a>0$ on a step, the recursion
averages a random next-step value $F$ by
$a^{-1}\log\E e^{aF}$, where the expectation is over that
step's Gaussian increment. At $a=0$, it uses $\E F$.
We choose
\begin{equation}\label{src:replica-gamma}
 \gamma(q)=\frac{\alpha(q)}{m(q)}.
\end{equation}
This division will make a group of size $m$ contribute
$m$ times the scalar Parisi value when $\theta=0$.
Empty intervals are omitted. For one group, take diagonal
$q$, off-diagonal $q\wedge S$, and set $m(q)=k$ below $S$
and $m(q)=1$ at and above $S$. In both cases $B(0)=0$,
$B(1)=D$, and $\gamma$ is nondecreasing with values in $[0,1]$.

For matrices $A,B$ of size $d$, write
\[
 A:B=\sum_{a,b=1}^d A_{ab}B_{ab},\qquad
 \Xi(A)=\sum_{a,b=1}^d\xi(A_{ab}),\qquad
 \Theta(A)=A:\nabla\Xi(A)-\Xi(A).
\]
Here $\nabla\Xi(A)$ is the entrywise matrix
$(\kappa(A_{ab}))$. The entries of $B(q)$ are differentiable
except possibly at $r$ and $S$, where the groups of replicas
sharing a scalar field change. For
$q\in(0,1)\setminus\{r,S\}$, put
\[
 C(q)=\frac{d}{dq}\nabla\Xi(B(q)).
\]
Set $C(q)=0$ at the excluded points; these values do not
affect any integral in $q$. Within each group of replicas
sharing a scalar field, every entry of $C(q)$ equals
$\kappa'(q)$; between different current groups its entries
are zero. Thus $C(q)$ is positive semidefinite and
\[
 \sum_{a,b}|C_{ab}(q)|\le d^2B_\xi.
\]
The scalar function $\Theta(B(q))$ is nondecreasing.

The perturbed terminal function is
\begin{equation}\label{src:source-terminal}
 b_\theta(x)=
 \log\sum_{\epsilon\in\{-1,1\}^d}
       e^{x\cdot\epsilon+\theta P_0(\epsilon)},
 \qquad x\in\R^d.
\end{equation}
Since $\alpha$ is a step function, so is $\gamma$.
Choose a partition $0=t_0<\cdots<t_m=1$ containing the
splitting times of the construction, with $\gamma=a_j$
on $[t_j,t_{j+1})$. Let $Z_j$ be independent centered
Gaussian vectors with covariance
$\int_{t_j}^{t_{j+1}}C(q)\,dq$.
Starting with $V_m=b_\theta$, define
\[
 V_j(x)=
 \begin{cases}
 a_j^{-1}\log\E_{Z_j}e^{a_jV_{j+1}(x+Z_j)},&a_j>0,\\
 \E_{Z_j}V_{j+1}(x+Z_j),&a_j=0.
 \end{cases}
\]
For $t_j\le q<t_{j+1}$, define $U_\theta(q,x)$ by the
same averaging formula with terminal function $V_{j+1}$
and Gaussian covariance $\int_q^{t_{j+1}}C(s)\,ds$, and
set $U_\theta(1,x)=b_\theta(x)$. Thus
$U_\theta(t_j,x)=V_j(x)$. The subscript in $U_\theta$
always denotes the perturbation parameter. We write $U$ for
$U_\theta$ in the estimates below.

We next show that the recursion depends continuously on
$\gamma$ in $L^1([0,1])$. This will define $U_\theta$ for
general $\alpha$ and allow zero, repeated, or unit
coefficients in the cascade calculation. At the terminal
time and after each averaging step with coefficient in $[0,1]$,
\begin{equation}\label{src:replica-derivative-bounds}
 |\partial_aU|\le1,\qquad D^2U\ge0,\qquad
 0\le\partial_{aa}U\le1-(\partial_aU)^2,\qquad
 |\partial_{ab}U|\le1.
\end{equation}
At the terminal time, the gradient and Hessian are the
mean and covariance of the Ising signs. For an averaging
step with next-step function $F$, differentiation gives
gradient $\E_j\nabla F$ and Hessian
\[
 \E_jD^2F+a_j\Cov_j(\nabla F),
\]
where $\E_j$ denotes expectation under the Gaussian law
of $Z_j$ weighted by $e^{a_jF(x+Z_j)}$ and normalized
to have mass one. Positivity is preserved. If the bounds
hold for $F$, the $a$th diagonal entry is at most
\[
 1-\E_j[(\partial_aF)^2]
    +a_j\Var_j(\partial_aF)
 \le1-(\E_j[\partial_aF])^2.
\]
The off-diagonal bound follows from positivity of the
Hessian and the diagonal bounds.

On each interval the recursion solves
\[
 U_q+\frac12C(q):D^2U
       +\frac{\gamma(q)}2\nabla U^{\mathsf T}C(q)\nabla U=0.
\]
Subtract recursions with coefficients $\gamma,\widetilde\gamma$
and the same terminal function. Their difference solves
a linear equation with drift
\[
 \frac{\gamma}{2}C(\nabla U+\nabla\widetilde U)
\]
and forcing bounded in absolute value by
$\frac12|\gamma-\widetilde\gamma|\sum_{a,b}|C_{ab}|$.
The drift is bounded and spatially Lipschitz by
\zeqref{src:replica-derivative-bounds}. Applying It\^o's
formula along the diffusion with this drift and covariance
matrix $C(q)$, interval by interval, gives
\begin{equation}\label{src:replica-recursion-continuity}
 \sup_{q,x}|U(q,x)-\widetilde U(q,x)|
 \le\frac12\int_0^1|\gamma-\widetilde\gamma|
                         \sum_{a,b}|C_{ab}|\,dq.
\end{equation}
The stopping times used to keep the diffusion in bounded
spatial regions can be removed because the gradients are
bounded and the total covariance is finite. In particular,
if step functions $\gamma_n$ converge to $\gamma$ in
$L^1([0,1])$, their recursion values satisfy
\[
 \sup_{q,x}|U^{(n)}(q,x)-U^{(p)}(q,x)|
 \le\frac{d^2B_\xi}{2}\|\gamma_n-\gamma_p\|_{L^1}.
\]
They therefore converge uniformly to a limit independent
of the approximating sequence.

\emph{The finite-level cascade used in the interpolation.}
We use cascade weights to express the Gaussian recursion
as an expected logarithm of a partition sum. For this
calculation, temporarily allow any strictly increasing
coefficient list $0<a_0<\cdots<a_{m-1}<1$, with
$\gamma(q)=a_j$ on $[t_j,t_{j+1})$. We keep the partition
and covariance path fixed and return to $\gamma(q)=\alpha(q)/m(q)$
after deriving the interpolation bound.

At every vertex of depth $j$, independently generate
Poisson points on $(0,\infty)$ with intensity
$a_ju^{-1-a_j}\,du$. Multiply points along each depth-$m$
branch and normalize the products to weights
$(v_\beta)_{\beta\in\mathbb N^m}$. The normalized weights
sum to one almost surely.
Attach independent Gaussian increments to each edge,
independently of the Poisson points and the original
Hamiltonian. The vector increments have covariance
\[
 \nabla\Xi(B(t_{j+1}))-\nabla\Xi(B(t_j)),
\]
and the scalar increments have variance
$\Theta(B(t_{j+1}))-\Theta(B(t_j))$.
Take independent copies of the vector increments for each
site $i$, independently of the scalar increments. Let
$Z_i(\beta)$ and $Y(\beta)$ be the sums of the vector and
scalar increments along the path to $\beta$. Their
covariances are
\[
 \E Z_i(\beta)Z_i(\beta')^{\mathsf T}
       =\nabla\Xi(B_{\beta,\beta'}),\qquad
 \E Y(\beta)Y(\beta')=\Theta(B_{\beta,\beta'}),
\]
where $B_{\beta,\beta'}=B(t_\ell)$ if the leaves share
exactly $\ell$ initial edges; identical leaves use
$B(t_m)=D$.

For these fields, the cascade recursion identity
\cite[Lemma~3, Eq.~(3.13)]{PT} gives
\[
 \E\log\sum_\beta v_\beta
    e^{b_\theta(x+Z_i(\beta))}
 =U_\theta(0,x),
 \qquad x\in\R^d.
\]
The expectation is over the cascade weights and Gaussian
increments. To see the normalization at one level,
multiply each Poisson point by an independent positive
mark $V$ with $\E V^{a_j}<\infty$. The intensity is
multiplied by $\E V^{a_j}$, so the expected logarithm
of the total mass increases by
$a_j^{-1}\log\E V^{a_j}$. The cascade identity applies
this calculation successively from the leaves to the
root, producing the logarithmic averaging steps above.

We now check that the necessary moment conditions hold here. Write
$\Gamma_{\rm E}(s)=\int_0^\infty x^{s-1}e^{-x}\,dx$
for Euler's gamma function. A total mass with Poisson
intensity $au^{-1-a}\,du$, where $0<a<1$, has Laplace
transform $\exp\{-\Gamma_{\rm E}(1-a)t^a\}$.
Integrating this transform shows that its negative moments
are finite and its positive moments of order less than
$a$ are finite. In particular its logarithm is integrable.
The strict increase of the coefficients permits these
moment bounds to be applied from the bottom level upwards.
Each added mark is the exponential of a function of a
finite Gaussian vector, and that function has at most
linear growth. These marks have every positive power
moment required in conditional marking.

\emph{The interpolation derivative.}
Keep the external field and perturbation fixed and set
\[
 \mathcal H_{N,t}(\boldsymbol\sigma,\beta)
 =\sqrt t\sum_{\ell=1}^dH_N(\sigma^\ell)
  +\sqrt{1-t}\sum_{i=1}^N
       (\sigma_i^1,\ldots,\sigma_i^d)\cdot Z_i(\beta)
  +\sqrt{tN}\,Y(\beta).
\]
Define
\[
 \varphi_N(t)
 =\frac1N\E\log\!\left[
   \sum_{\beta\in\mathbb N^m}v_\beta
   \sum_{\boldsymbol\sigma\in\mathcal C_\eta(D)}
   \exp\!\left\{
     \mathcal H_{N,t}(\boldsymbol\sigma,\beta)
     +h\sum_{\ell=1}^d\sum_{i=1}^N\sigma_i^\ell
     +\theta NA_N(\boldsymbol\sigma)
   \right\}
 \right].
\]
Its expectation includes the disorder, cascade, and auxiliary
Gaussian fields. The sampling brackets below refer to this
normalized interpolating measure.

For two vector configurations, let
$A_{ab}=N^{-1}\sum_i\sigma_i^a\tau_i^b$ be their cross-overlap
matrix, and define
\[
 \mathcal E_\Xi(A,B)
 =\Xi(A)-\Xi(B)-\nabla\Xi(B):(A-B).
\]
For two independent samples
$(\boldsymbol\sigma^{\,1},\beta^1)$ and
$(\boldsymbol\sigma^{\,2},\beta^2)$ from the interpolating
measure, write
\[
 A^{u,v}_{ab}
 =\frac1N\sum_i\sigma_i^{u,a}\sigma_i^{v,b},
 \qquad u,v\in\{1,2\}.
\]
For fixed configurations and leaves, put
$B=B_{\beta,\beta'}$. The covariance of the interpolating
Hamiltonian, divided by $N$, is
\[
 \frac1N\E\!\left[
   \mathcal H_{N,t}(\boldsymbol\sigma,\beta)
   \mathcal H_{N,t}(\boldsymbol\tau,\beta')\right]
 =t\Xi(A)+(1-t)A:\nabla\Xi(B)+t\Theta(B).
\]
Its derivative is $\mathcal E_\Xi(A,B)$. Thus the scalar
field supplies precisely the term needed to obtain this
Taylor remainder. Gaussian integration by parts gives
\begin{equation}\label{src:replica-interpolation-derivative}
 \varphi_N'(t)=\frac12\E\left\langle
    \mathcal E_\Xi(A^{1,1},D)
    -\mathcal E_\Xi(A^{1,2},B_{\beta^1,\beta^2})
                                   \right\rangle_t
\end{equation}
for almost every $t\in(0,1)$. We establish this identity
by first restricting to finitely many leaves and then
passing to the limit in its integrated form.

Evenness and the nonnegative coefficients imply convexity
of $\xi$ on $[-1,1]$. Taylor's integral formula and the
bound for $\xi''$ give
\[
 0\le\xi(x)-\xi(y)-\kappa(y)(x-y)
       \le\frac{B_\xi}{2}(x-y)^2.
\]
Thus the second term of
\zeqref{src:replica-interpolation-derivative} is nonpositive.
The matrix $A^{1,1}$ has diagonal entries one, and its
$d(d-1)$ off-diagonal entries differ from those of $D$
by less than $\eta$. It follows that
\begin{equation}\label{src:replica-shell-error}
 \varphi_N'(t)\le e_d(\eta),
 \qquad e_d(\eta)=\frac{d(d-1)}4B_\xi\eta^2.
\end{equation}
We now justify the derivative identity and its integration
for the full cascade.

Order the Poisson points at each vertex decreasingly and
restrict the leaves to $\Lambda_L=\{1,\ldots,L\}^m$,
retaining the original weights $v_\beta$. Denote the
resulting partition sum by $Z_{N,t}^{(L)}$ and put
\[
 \varphi_N^{(L)}(t)=\frac1N\E\log Z_{N,t}^{(L)}.
\]
Conditional on the cascade weights, this partition sum
depends on a finite Gaussian vector. This remains true
for an infinite mixture, since the original Hamiltonian
has only $2^N$ values at fixed $N$. Gaussian integration
by parts therefore proves
\zeqref{src:replica-interpolation-derivative} with
$\varphi_N^{(L)}$ and the normalized restricted sampling
measure in place of their full-cascade counterparts.
Write $g_L(t)$ for this restricted right-hand side. Thus
\[
 \varphi_N^{(L)}(t)-\varphi_N^{(L)}(s)
 =\int_s^t g_L(u)\,du,
 \qquad 0\le s<t\le1.
\]
We pass to the limit separately on the two sides of
this identity.

For the expected logarithms, every $\Lambda_L$ contains
$\beta_0=(1,\ldots,1)$. Fix
$\boldsymbol\sigma_0\in\mathcal C_\eta(D)$ and put
$K_N=N(d|h|+|\theta|)$. Keeping this one summand gives
\[
 \log Z_{N,t}^{(L)}
 \ge\log v_{\beta_0}
       +\mathcal H_{N,t}(\boldsymbol\sigma_0,\beta_0)-K_N.
\]
The lower bound is integrable. Indeed, the moment bounds
above give integrable logarithms for the total cascade
mass and for the largest Poisson point at each vertex
along $\beta_0$. Hence $\E|\log v_{\beta_0}|<\infty$.
For the positive parts,
\[
 (\log Z_{N,t}^{(L)})^+\le Z_{N,t},\qquad
 \sup_{0\le t\le1}\E Z_{N,t}
 \le2^{Nd}\exp\{K_N+C_{d,\xi}N\}<\infty.
\]
The second bound follows from Gaussian exponential
moments and $\sum_\beta v_\beta=1$. Since
$Z_{N,t}^{(L)}\uparrow Z_{N,t}$, these bounds give
\[
 \lim_{L\rightarrow \infty} \varphi_N^{(L)}(t) = \varphi_N(t)
\]
for every fixed $t\in[0,1]$.

For the right-hand side, the normalized restricted sampling
measures converge in total variation to the full sampling
measure at each fixed $t$. Moreover,
\[
 |\mathcal E_\Xi(A,B)|
 \le2d^2\bigl(\|\xi\|_\infty+\|\kappa\|_\infty\bigr),
\]
where the norms are taken on $[-1,1]$. Thus $g_L(t)$
converges to the full right-hand side $g(t)$ and is bounded
uniformly in $L$ and $t$. Dominated convergence now gives
\[
 \varphi_N(t)-\varphi_N(s)=\int_s^t g(u)\,du,
 \qquad 0\le s<t\le1.
\]
Consequently $\varphi_N$ is absolutely continuous and
satisfies \zeqref{src:replica-interpolation-derivative}
almost everywhere.

\emph{The endpoints.}
At $t=1$, the scalar Gaussian field is independent of
the spin configuration, so the partition sum factors.
For a scalar increment $V$ of variance $\Delta\Theta$,
one recursion step contributes
\[
 \frac1{a_j}\log\E e^{a_j\sqrt N V}
 =\frac{Na_j}{2}\Delta\Theta.
\]
Dividing by $N$ and summing over the steps gives
\[
 \varphi_N(1)=\E\mathcal F_N^{D,P_0}(\theta)
       +\frac12\int_0^1\gamma(q)\,d\Theta(B(q)).
\]
At $t=0$, dropping the overlap constraint bounds the sum
from above. The coordinate sums then factor over sites
and have terminal function \zeqref{src:source-terminal}.
Independence of their Gaussian increments makes their
backward recursions additive. Consequently
\[
 \varphi_N(0)\le U_\theta(0,h\mathbf1).
\]
Combining the two endpoints with
\zeqref{src:replica-shell-error} gives
\begin{equation}\label{src:replica-endpoints}
 \E\mathcal F_N^{D,P_0}(\theta)
 \le U_\theta(0,h\mathbf1)
       -\frac12\int_0^1\gamma\,d\Theta(B)+e_d(\eta).
\end{equation}

For a general list $0\le a_0\le\cdots\le a_{m-1}\le1$, set
\[
 a_j^{(\varepsilon)}
 =(1-\varepsilon)a_j+\varepsilon\frac{j+1}{m+1},
 \qquad 0<\varepsilon<1,
\]
and let $\gamma^{(\varepsilon)}=a_j^{(\varepsilon)}$
on $[t_j,t_{j+1})$. These coefficients are strictly
increasing and lie in $(0,1)$. Apply the interpolation
bound just proved with $a_j^{(\varepsilon)}$, keeping
the partition and covariance path fixed.
Estimate~\zeqref{src:replica-recursion-continuity}
gives uniform convergence of the recursion values.
The penalties converge as well, since
\[
 \left|\int_0^1
   (\gamma^{(\varepsilon)}-\gamma)\,d\Theta(B)\right|
 \le d^2B_\xi
    \|\gamma^{(\varepsilon)}-\gamma\|_{L^1}.
\]
Thus \zeqref{src:replica-endpoints} holds for every
nondecreasing step coefficient with values in $[0,1]$.
We now return to the choice $\gamma(q)=\alpha(q)/m(q)$.

\emph{The cancellation at $\theta=0$.}
When $\theta=0$, the terminal spin sum factors over the
$d$ replica coordinates, giving
\[
 b_0(x)
 =\log\sum_{\epsilon\in\{-1,1\}^d}e^{x\cdot\epsilon}
 =\sum_{\ell=1}^d\log(2\cosh x_\ell),
 \qquad x\in\R^d.
\]
We evaluate the recursion at field vectors whose
coordinates agree within each current group. Consider
a step $[s,t]$ on which $\alpha=\alpha_0$ and no group
splits. If a group of size $m$ contributes $m\Phi(t,y)$
at time $t$, its common Gaussian increment gives
\[
 \frac{m}{\alpha_0}\log
 \E\exp\{\alpha_0\Phi(t,x+Z)\}
 =m\Phi(s,x),
 \qquad Z\sim N(0,\kappa(t)-\kappa(s)).
\]
Here the recursion coefficient is $\alpha_0/m$.
The equality is the scalar Parisi recursion, obtained
from the backward heat equation for $e^{\alpha_0\Phi}$
on $[s,t]$. For $\alpha_0=0$, the scalar PDE itself
is linear and the recursion uses ordinary expectation.
Contributions from distinct groups add because their
Gaussian increments are independent. At a splitting time,
all new groups start at the same field value, and their
sizes sum to the size of the original group. Their
contributions therefore agree with that of the original
group. Starting from the terminal function $b_0$ and
iterating backwards gives
\[
 U_0(0,h\mathbf1)=d\Phi(0,h).
\]

At level $q$, there are $d/m(q)$ groups, each containing
$m(q)^2$ ordered pairs of coordinates, including equal
indices. Within each group, $B_{ab}(q)=q$, so each pair
contributes $q\kappa'(q)$ to
$\frac{d}{dq}\Theta(B(q))$. Between groups, $B_{ab}$ is
constant and contributes zero. Therefore
\begin{equation}\label{src:replica-penalty}
 \frac12\int_0^1\gamma\,d\Theta(B)
 =\frac12\int_0^1\frac{\alpha(q)}{m(q)}
                    d\,m(q)q\kappa'(q)\,dq
 =\frac d2\int_0^1q\kappa'(q)\alpha(q)\,dq.
\end{equation}
Subtracting this term from $U_0(0,h\mathbf1)=d\Phi(0,h)$
gives $d$ times the scalar Parisi functional at $\alpha$.
For the minimizing distribution, this equals $dP$.

\emph{The effect of the added spin product.}
It remains to bound
$w(q,x)=U_\theta(q,x)-U_0(q,x)$. We compare the two
recursions using the Gaussian transitions weighted
by the values at $\theta=0$. Consider a step from
$t_{j+1}$ to $t_j$ with coefficient $a=a_j\in(0,1]$
and Gaussian increment $Z_j$. For a starting field $x$,
define the probability measure
\[
 Q_{0,j}^x(dz)
 =
 \frac{e^{aU_0(t_{j+1},x+z)}}
      {\E e^{aU_0(t_{j+1},x+Z_j)}}
 \,\Law(Z_j)(dz).
\]
Subtracting the two logarithmic recursions gives
\[
 \begin{aligned}
 w(t_j,x)
 &=\frac1a\log\int
     e^{a w(t_{j+1},x+z)}\,Q_{0,j}^x(dz)\\
 &\le\log\int
     e^{w(t_{j+1},x+z)}\,Q_{0,j}^x(dz).
 \end{aligned}
\]
Indeed, the logarithm of the moment generating function
is convex and vanishes at zero, so its value at $a\in(0,1]$
is at most $a$ times its value at one. For $a=0$, take
$Q_{0,j}^x=\Law(Z_j)$. The difference of the recursions is
then the expectation of $w(t_{j+1},x+Z_j)$, and Jensen's
inequality gives the same upper bound.

To identify these transitions, consider a step $[s,t]$
on which $\alpha=\alpha_0$ and each current group has
size $m$. When the coordinates in a group all start
at $x$, their common field at time $t$ has distribution
\[
 \exp\{\alpha_0[\Phi(t,y)-\Phi(s,x)]\}
 \,N(x,\kappa(t)-\kappa(s))(dy).
\]
Indeed, the group's contribution to $U_0$ is $m\Phi$
and the recursion coefficient is $\alpha_0/m$.
The scalar recursion normalizes this density.
Since $e^{\alpha_0\Phi}$ solves the backward heat
equation, the resulting diffusion has drift
\[
 \kappa'(q)\partial_x\log e^{\alpha_0\Phi(q,x)}
 =\alpha_0\kappa'(q)u(q,x).
\]
Conditional on their starting fields, distinct groups
have independent transitions. Starting at $h\mathbf1$,
the kernels $Q_{0,j}^x$ therefore give the coupled
scalar diffusion law in \zeqref{src:source-coefficient}.

Exponentiating and iterating over the steps gives
\[
 U_\theta(0,h\mathbf1)-U_0(0,h\mathbf1)
 \le\log\E_{\rm sc}
       e^{b_\theta(X_1)-b_0(X_1)},
 \qquad X_1=(X_1^1,\ldots,X_1^d).
\]
By the definition of the terminal function,
\[
 e^{b_\theta(X_1)-b_0(X_1)}
 =
 \frac{\sum_{\epsilon\in\{-1,1\}^d}
       e^{X_1\cdot\epsilon+\theta P_0(\epsilon)}}
      {\sum_{\epsilon\in\{-1,1\}^d}e^{X_1\cdot\epsilon}}
 =
 \E_{\rm sc}[e^{\theta P_0(\epsilon)}\mid X_1].
\]
Hence the preceding upper bound is
$f(\theta)=\log\E_{\rm sc}e^{\theta P_0(\epsilon)}$.
We have $f(0)=0$ and $f'(0)=c$. Moreover, $f''(\theta)$
is the variance of $P_0$ under the probability law with
density
$e^{\theta P_0}/\E_{\rm sc}e^{\theta P_0}$
relative to the scalar path and sign law. Since
$P_0\in\{-1,1\}$, this variance is at most one.
Therefore, for every $\theta\in\R$,
\[
 U_\theta(0,h\mathbf1)-U_0(0,h\mathbf1)
 \le f(\theta)\le\theta c+\frac{\theta^2}{2}.
\]

\emph{Passing to a general distribution function.}
Choose step distribution functions $\alpha_j$ converging
to $\alpha$ in $L^1([0,1])$, using partitions that contain
the splitting times of the construction. Put
$\gamma_j(q)=\alpha_j(q)/m(q)$, and let $U_{\theta,j}$ and $c_j$ 
be the recursion and scalar coefficient obtained from
$\alpha_j$. Then
$\|\gamma_j-\gamma\|_{L^1}\le\|\alpha_j-\alpha\|_{L^1}$, so
\zeqref{src:replica-recursion-continuity} gives
\[
 U_{\theta,j}(0,h\mathbf1)\longrightarrow
 U_\theta(0,h\mathbf1).
\]
The interpolation penalties also converge, since
$|d\Theta(B(q))/dq|\le d^2B_\xi$ almost everywhere.

Let $\Phi_j$ be the scalar Parisi solution for $\alpha_j$
and put $u_j=\partial_x\Phi_j$. The associated probability
measures converge weakly, and
\cite[Proposition~2, Eqs.~(18)--(23)]{ACunique} gives
$\|u_j-u\|_\infty\to0$ on $[0,1]\times\R$.
For the diffusion drifts $b_j=\kappa'\alpha_ju_j$ and
$b=\kappa'\alpha u$, we therefore have
\[
 \int_0^1\sup_{x\in\R}|b_j(q,x)-b(q,x)|\,dq
 \le B_\xi\bigl(
       \|\alpha_j-\alpha\|_{L^1}+\|u_j-u\|_\infty
       \bigr)\longrightarrow0.
\]
Construct the diffusions for $\alpha_j$ and $\alpha$
using the same Brownian motions and splitting times.
The uniform spatial Lipschitz bound on the drifts
and Gronwall's inequality give, for each replica $\ell$,
\[
 \E_{\rm sc}\sup_{0\le q\le1}
 |X_q^{j,\ell}-X_q^\ell|^2\longrightarrow0.
\]
The formulas in Lemma~\zcref[noname]{src:replica-coefficient},
together with $\|u_j-u\|_\infty\to0$ and $|u_j|\le1$,
then give convergence of the coefficients $c_j$ to $c$.

The scalar recursion also satisfies
\zeqref{src:replica-recursion-continuity}, with $d=1$,
so $\Phi_j(0,h)\to\Phi(0,h)$. Passing to the limit in
the calculation at $\theta=0$ and the bound for the
added spin product therefore gives
\[
 U_0(0,h\mathbf1)
   -\frac12\int_0^1\gamma\,d\Theta(B)=dP,
 \qquad
 U_\theta(0,h\mathbf1)-U_0(0,h\mathbf1)
   \le\theta c+\frac{\theta^2}{2}.
\]
The interpolation bound
\zeqref{src:replica-endpoints} passes to the limit by
the same recursion and penalty estimates. Combining
these inequalities proves \zeqref{src:source-upper}
for nonnegative cross overlaps.

Finally, suppose $h=0$ and $-S\le r<0$.
Reversing every spin in the second group leaves each even
Hamiltonian and every within-group overlap unchanged, and
maps the overlap constraint with cross overlap $r$ to the
one with cross overlap $|r|$. It replaces the monomial
$P_0$ by $(-1)^bP_0$, where $b$ is the number of its factors
from the second group. Apply the bound at $|r|$ with perturbation 
parameter $(-1)^b\theta$. The linear term is then
$\theta(-1)^bc(P_0;|r|,S)$ and the quadratic term is
$\theta^2/2$. By
Lemma~\zcref[noname]{src:replica-coefficient},
$(-1)^bc(P_0;|r|,S)$ is the reflected coefficient.
This proves the negative-overlap assertion.
\end{proof}

\subsection{Suppression of deviations under the original Gibbs measure}
\label{src:source-suppression-section}

\begin{lemma}\label{src:source-suppression}
Fix the one-group or two-group construction, with $d$ replicas,
and let $\delta>0$. There are $\eta_0,C,c_0>0$ and $N_0$ such that
\begin{equation}\label{src:source-suppression-bound}
 \E G_{N,h}^{\otimes d}
 \left(\mathcal C_\eta(D)\cap\{|A_N-c|>\delta\}\right)
 \le C e^{-c_0N}
\end{equation}
for every $N\ge N_0$, every $0<\eta\le\eta_0$, and every
monomial $P_0$ from the chosen construction.
The bound is uniform over deterministic $S\in[0,1]$ in
the one-group case and $0\le r\le S\le1$ in the two-group
case. At $h=0$, the latter range extends to $|r|\le S\le1$.
The constant $\eta_0$ depends only on $\xi,d,\delta$.
The constants $C,c_0,N_0$ may also depend on $h$, but are
independent of the overlap levels, $\eta$, and $P_0$.
\end{lemma}

\begin{proof}
We may assume
$\mathcal C_\eta(D)$ is nonempty. We first state concentration
bounds for the constrained pressure and the original pressure
$F_{N,h}$, whose expectation is $p_N$.

For a fixed replica tuple, the Gaussian random variable
$\sum_{\ell=1}^d H_N(\sigma^\ell)$ has a coefficient vector
of Euclidean norm at most $d\sqrt{N\xi(1)}$ when expressed
in terms of independent standard Gaussian variables.
The gradient of $\mathcal F_N^{D,P_0}(\theta)$ is $N^{-1}$
times a Gibbs average of these vectors. Thus its Lipschitz
constant is at most $d\sqrt{\xi(1)/N}$, uniformly in
$\theta$, the overlap constraints, and $P_0$. The same
argument gives Lipschitz constant $\sqrt{\xi(1)/N}$ for
$F_{N,h}$. Gaussian concentration
\cite[Theorem~3.25]{VanHandel2016} therefore gives
\[
 \Prob(|V_N-\E V_N|>u)
 \le 2\exp\{-u^2/(2\ell_N^2)\},
 \qquad u>0,
\]
where $V_N=\mathcal F_N^{D,P_0}(\theta)$ and
$\ell_N=d\sqrt{\xi(1)/N}$ for the constrained pressure,
and $V_N=F_{N,h}$ and $\ell_N=\sqrt{\xi(1)/N}$ for the
original pressure. 

Choose  $\theta=\delta$ and put
\[
 g=\theta\delta-\frac{\theta^2}{2}=\frac{\delta^2}{2}.
\]
Choose $\eta_0>0$ so that $e_d(\eta_0)\le g/4$.
By the Parisi formula, there is $N_0$, depending only on
$\xi,h,d,\delta$, such that
\[
 |p_N-P|\le \frac{g}{8d},
 \qquad N\ge N_0.
\]
The concentration bounds imply that, outside a disorder
event of probability at most $C_1e^{-c_1N}$,
\[
 \mathcal F_N^{D,P_0}(\pm\theta)
 \le \E\mathcal F_N^{D,P_0}(\pm\theta)+\frac g4,
 \qquad
 F_{N,h}\ge p_N-\frac{g}{8d}.
\]
Here $C_1,c_1$ are independent of the overlap levels,
$\eta$, and $P_0$.

For replica tuples with $A_N\ge c+\delta$, multiplication
by $e^{\theta N A_N}$ increases each weight by at least
$e^{\theta N(c+\delta)}$. Consequently,
\[
 G_{N,h}^{\otimes d}
   (\mathcal C_\eta(D)\cap\{A_N\ge c+\delta\})
 \le \exp\left\{
 N\left[\mathcal F_N^{D,P_0}(\theta)
          -dF_{N,h}-\theta(c+\delta)\right]\right\}.
\]
On the disorder event where the concentration bounds hold,
Lemma~\zcref[noname]{src:replica-upper} gives
\[
 \begin{aligned}
 \mathcal F_N^{D,P_0}(\theta)
       -dF_{N,h}-\theta(c+\delta)
 &\le -g+e_d(\eta)+\frac g4
             +d(P-p_N)+\frac g8\\
 &\le -g+\frac g4+\frac g4+\frac g8+\frac g8
 =-\frac g4.
 \end{aligned}
\]
For the lower tail, multiplication by $e^{-\theta N A_N}$
similarly gives
\[
 G_{N,h}^{\otimes d}
   (\mathcal C_\eta(D)\cap\{A_N\le c-\delta\})
 \le \exp\left\{
 N\left[\mathcal F_N^{D,P_0}(-\theta)
          -dF_{N,h}+\theta(c-\delta)\right]\right\}.
\]
The pressure bound at $-\theta$ has linear term $-\theta c$,
so the same calculation bounds this exponent by $-Ng/4$.

On the exceptional disorder event, the Gibbs probability
of the union of the two tails is at most one. Hence
\[
 \E G_{N,h}^{\otimes d}
 \left(\mathcal C_\eta(D)\cap\{|A_N-c|>\delta\}\right)
 \le 2e^{-Ng/4}+C_1e^{-c_1N}.
\]
This proves \zeqref{src:source-suppression-bound} with the
stated uniformity.
\end{proof}

\section{Conditional cavity laws}\label{sec:cavity}
In this section we use the estimates of \zcref{sec:sources} to identify
the conditional cavity laws needed in \zcref[nocomp]{sec:shells,sec:regression}.
\zcref{subsec:cavity-representation} constructs the limiting cavity
representation by removing and restoring one spin. In
\zcref{subsec:cavity-observables} we recover directing-vector norms from
the overlap array and control the error in approximating cavity
magnetizations using nearby replicas. \zcref{subsec:conditional-magnetizations}
combines these facts with the comparison and concentration estimates to identify the laws of one
magnetization and of a pair at equal norms, conditional on the full overlap
array. This conditioning retains overlaps with all additional replicas.

\subsection{Removing and restoring one spin}
\label{subsec:cavity-representation}

We use the standard one-spin cavity decomposition and Gibbs
reweighting; see \cite[Chapter~4]{PanchenkoBook} and
\cite[Section~2]{AuffingerJagannathSpinDistributions}.
We give the details needed to retain the added spin jointly
with the overlap array.

Write a configuration of the $(N+1)$-site system as
$(\sigma,\epsilon)\in\{-1,1\}^N\times\{-1,1\}$.
Keeping only interaction terms whose indices lie in
$\{1,\ldots,N\}$ gives the reduced Hamiltonian
\[
 H_N^{\rm red}(\sigma)
 =\sum_{p\in2\mathbb N}\frac{b_{p,N}}{N^{(p-1)/2}}
   \sum_{i_1,\ldots,i_p\le N}
   g^{(p)}_{i_1\cdots i_p}\sigma_{i_1}\cdots\sigma_{i_p},
 \qquad
 b_{p,N}=b_p\left(\frac{N}{N+1}\right)^{(p-1)/2}.
\]
The Gaussian variables are those of the full system
with all indices at most $N$. 
With the external field retained on the remaining sites,
the reduced Gibbs measure is
\[
 G_{N,h}^{\rm red}(\sigma)
 =\frac{1}{Z_{N,h}^{\rm red}}
   \exp\left(H_N^{\rm red}(\sigma)
             +h\sum_{i=1}^N\sigma_i\right).
\]
Summing over the last spin in the full Gibbs measure
additionally multiplies these weights by
\[
 \sum_{\epsilon\in\{-1,1\}}
 \exp\left\{
 H_{N+1}(\sigma,\epsilon)-H_N^{\rm red}(\sigma)
 +h\epsilon
 \right\},
\]
followed by normalization. 

Along a subsequence on which the overlap arrays of the
reduced and full systems converge jointly, we call their
limiting laws the \emph{predecessor} and \emph{successor}. 
The following lemma
describes the cavity reweighting that transforms a directing
measure of the predecessor into one of the successor. Passing to a further subsequence if necessary, we also
include the last-coordinate spins
$(\sigma_{N+1}^{\ell})_{\ell\ge1}$ in this joint convergence
in distribution. We denote the limiting signs by
$(\epsilon_\ell)_{\ell\ge1}$ and call them the
\emph{added-site spins}. In the directing-measure
representation, $\epsilon_\ell$ is a separate
$\{-1,1\}$-valued variable associated with the successor
vector $v_\ell$, not a coordinate of that vector.

\begin{lemma}\label{cav:representation}
Every physical overlap law admits the following representation
as a successor, jointly with a physical predecessor at the
same interaction coefficients and field.
Let $G$ be a directing measure of the predecessor.
Conditional on $G$, let $Y$ be a centered Gaussian field with
\[
 \E[Y(v)Y(w)\mid G]=\kappa(v\cdot w).
\]
Put
\[
 q(v)=\|v\|^2,\qquad
 a(v)=e^{[\kappa(1)-\kappa(q(v))]/2},\qquad
 Z_h=\int a(v)\cosh(h+Y(v))\,G(dv).
\]
Then $1\le Z_h<\infty$ almost surely, and a directing
measure of the successor is
\begin{equation}\label{eq:cavity-tilt}
 \widehat G(dv)
 =\frac{a(v)\cosh(h+Y(v))}{Z_h}\,G(dv).
\end{equation}
Conditional on $G,Y$ and vectors $v_1,\ldots,v_n$ sampled
independently from $\widehat G$, the added-site spins are
independent signs, with $\epsilon_\ell$ having mean
\begin{equation}\label{eq:cavity-mark}
 M(v_\ell)=\tanh(h+Y(v_\ell)).
\end{equation}
The measures $G$ and $\widehat G$ are mutually absolutely
continuous almost surely.
\end{lemma}

\begin{proof}
Fix a subsequence realizing the prescribed successor law.
For $\sigma\in\{-1,1\}^N$ and $\epsilon\in\{-1,1\}$, separate
the interaction terms according to the number of occurrences
of the index $N+1$. This gives
\[
 H_{N+1}(\sigma,\epsilon)
 =H_N^{\rm red}(\sigma)+\epsilon z_N(\sigma)
   +W_N(\sigma,\epsilon).
\]
Here $H_N^{\rm red}$ contains the terms with no occurrence
of $N+1$ and has coefficients $b_{p,N}$. 
The term $\epsilon z_N$ contains those with exactly one
occurrence, and $W_N$ contains those with at least two.
These three Gaussian processes are independent because
they use disjoint sets of Gaussian coefficients.

There are $p$ possible positions for the unique occurrence
of $N+1$ in an interaction of order $p$. Consequently,
\[
 \begin{aligned}
 \E z_N(\sigma)z_N(\tau)
 &=\sum_{p\in2\mathbb N}
     p b_p^2\left(\frac{N}{N+1}\right)^{p-1}
     R(\sigma,\tau)^{p-1}=\kappa\left(\frac{N}{N+1}R(\sigma,\tau)\right).
 \end{aligned}
\]
Since $|\kappa'|\le B_\xi$, this covariance differs from
$\kappa(R(\sigma,\tau))$ by at most $B_\xi/(N+1)$,
uniformly in $\sigma,\tau$. The repeated-index estimate
in Appendix~\zcref[noname]{app:cavity} gives
\[
 \sup_{\sigma,\epsilon}
 \E W_N(\sigma,\epsilon)^2\le \frac{C_\xi}{N}.
\]
Let $\widetilde G_{N+1,h}$ be the Gibbs measure obtained
by omitting $W_N$. The independent Gaussian remainder
estimate in the same appendix then yields
\[
 \E\|G_{N+1,h}-\widetilde G_{N+1,h}\|_{\rm TV}
 \le \frac{C_\xi}{\sqrt N}.
\]
Thus this omission preserves every subsequential joint
limit of finitely many overlaps and added-site spins.

We next identify the overlap limits of the reduced systems.
Compare $G_{N,h}^{\rm red}$ with the original $N$-site
measure $G_{N,h}$, using the same Gaussian coefficients
and external field. The relative-entropy comparison in
Appendix~\zcref[noname]{app:cavity} proves
\[
 \E D(G_{N,h}^{\rm red}\Vert G_{N,h})\longrightarrow0,
\]
where $D$ denotes relative entropy. Pinsker's inequality
therefore gives
\[
 \E\|G_{N,h}^{\rm red}-G_{N,h}\|_{\rm TV}\longrightarrow0.
\]
The same conclusion holds for every fixed number of
replicas. After extracting a further subsequence along
which the reduced overlap arrays converge jointly with
the successor arrays, their limiting law is therefore
physical at the original coefficients and field.
Let $G$ direct this predecessor law.

For the system with $W_N$ omitted, restoring the last spin
has the exact conditional sampling formula
\[
 \widetilde G_{N+1,h}(\sigma,\epsilon)
 =
 G_{N,h}^{\rm red}(\sigma)
 \frac{e^{\epsilon(h+z_N(\sigma))}}{2\mathcal Z_N},
 \qquad
 \mathcal Z_N
 =G_{N,h}^{\rm red}\cosh(h+z_N)\ge1.
\]
In particular, summing over $\epsilon$ reweights the
reduced measure by $\cosh(h+z_N)$, whereas prescribing
$\epsilon=e$ gives numerator $e^{e(h+z_N)}/2$.

Appendix~\zcref[noname]{app:cavity-joint} justifies passing
this sampling formula to the limit against bounded
continuous functions of finitely many overlaps, with any
fixed collection of added-site spins prescribed. The
argument approximates $\mathcal Z_N$ by an average over
auxiliary replicas, with an error tending to zero uniformly
in $N$. This permits taking the overlap limit first and
then sending the number of auxiliary replicas to infinity.

We now identify the joint limit of the overlaps of replicas
$\sigma^1,\ldots,\sigma^n$ sampled from $G_{N,h}^{\rm red}$
and the variables $z_N(\sigma^1),\ldots,z_N(\sigma^n)$.
Conditional on $G$, sample $v_1,\ldots,v_n$ independently
from $G$. Represent the off-diagonal overlaps by
$v_\ell\cdot v_j$ and the additional variables by a
conditionally centered Gaussian vector $(z_1,\ldots,z_n)$
with
\[
 \E[z_\ell z_j\mid G,v_1,\ldots,v_n]
 =
 \begin{cases}
 \kappa(v_\ell\cdot v_j),&\ell\ne j,\\
 \kappa(1),&\ell=j.
 \end{cases}
\]
Let $Y$ be a centered Gaussian field, independent of the
sampled vectors conditional on $G$, with
\[
 \E[Y(v)Y(w)\mid G]=\kappa(v\cdot w).
\]
We can realize the Gaussian vector above as
\[
 z_\ell=Y(v_\ell)+\eta_\ell,\qquad
 \eta_\ell\sim
 N\bigl(0,\kappa(1)-\kappa(q(v_\ell))\bigr),
\]
where, conditional on $G$ and the vectors, the
$\eta_\ell$ are independent of one another and of $Y$.
The field $Y$ gives the required off-diagonal covariances,
while each $\eta_\ell$ increases the variance from
$\kappa(q(v_\ell))$ to $\kappa(1)$.
These noises are independent for distinct replica labels,
even when the sampled vectors coincide.

For a prescribed sign $e\in\{-1,1\}$, integration over
the additional noise gives
\[
 \E_\eta e^{e(h+Y(v)+\eta)}
 =
 e^{[\kappa(1)-\kappa(q(v))]/2}
 e^{e(h+Y(v))}
 =a(v)e^{e(h+Y(v))}.
\]
The limiting normalizing factor is consequently
\[
 Z_h
 =\int \E_\eta\cosh(h+Y(v)+\eta)\,G(dv)
 =\int a(v)\cosh(h+Y(v))\,G(dv).
\]
Since $q(v)\le1$ and $\kappa$ is nondecreasing on $[0,1]$,
we have $a(v)\ge1$ and $Z_h\ge1$. Moreover,
\[
 \E[Z_h\mid G]
 =\int a(v)e^{\kappa(q(v))/2}\cosh h\,G(dv)
 =e^{\kappa(1)/2}\cosh h<\infty.
\]
Thus $Z_h$ is finite almost surely.

It follows that, conditional on $G,Y$, the limiting joint
law of $n$ successor vectors and prescribed added-site
signs $e_1,\ldots,e_n$ is
\[
 \prod_{\ell=1}^n
 \left[
  \frac{a(v_\ell)e^{e_\ell(h+Y(v_\ell))}}{2Z_h}
  \,G(dv_\ell)
 \right].
\]
Summing over the signs gives independent vectors with
law $\widehat G$ from \zeqref{eq:cavity-tilt}.
Conditional on these vectors, the signs are independent
and satisfy
\[
 \Prob(\epsilon_\ell=e\mid G,Y,v_1,\ldots,v_n)
 =\frac{e^{e(h+Y(v_\ell))}}
        {2\cosh(h+Y(v_\ell))},
 \qquad e\in\{-1,1\}.
\]
Their conditional means are therefore
$\tanh(h+Y(v_\ell))$, proving \zeqref{eq:cavity-mark}.

For two full replicas $(\sigma^\ell,\epsilon_\ell)$ and
$(\sigma^j,\epsilon_j)$, their overlap is
\[
 R_{\ell j}^{\rm full}
 =\frac{N}{N+1}R(\sigma^\ell,\sigma^j)
   +\frac{\epsilon_\ell\epsilon_j}{N+1}.
\]
Its difference from the overlap of their first $N$ spins
is at most $2/(N+1)$. Hence $\widehat G$ directs the
prescribed successor overlap law. Finally, its density
with respect to $G$ is strictly positive and finite
$G$-almost everywhere, proving mutual absolute continuity.
\end{proof}

We write $v^1,v^2,\ldots$ for independent samples from
$\widehat G$ conditional on $G,Y$, and set $R_{ij}=v^i\cdot v^j$
for $i\ne j$. For Gaussian integration by parts, we use the
Gram diagonal $R_{ii}=q_i=\|v^i\|^2$. Put $Y_i=Y(v^i)$ and
$M_i=M(v^i)$. Expectations average over $G$, the Gaussian
field $Y$ conditional on $G$, and the replicas conditional
on $G,Y$.
For every $n\ge1$ and $c\ge0$,
\begin{equation}\label{cav:fieldtails}
 \E\widehat G^{\otimes n}
       \exp\left(c\sum_{i=1}^n|Y_i|\right)<\infty.
\end{equation}
Indeed, expand each factor of $\widehat G$ using
\zeqref{eq:cavity-tilt}, use $Z_h^{-n}\le1$, and bound each
numerator weight by $e^{\kappa(1)/2+|h|+|Y_i|}$.
The remaining integral is over $G^{\otimes n}$ and $Y$.
Conditional on $G$ and these vectors, the fields are Gaussian
with variances at most $\kappa(1)$, so their exponential
moments give the bound.

An interaction order $p$ is called active when $b_p>0$.

\begin{lemma}\label{cav:moments}
Let $B_\mu=\int q\kappa(q)\,\mu(dq)$. Every physical limit
satisfies
\begin{equation}\label{eq:active-moment}
 \E[R_{12}\kappa(R_{12})]=B_\mu.
\end{equation}
In any successor representation, site exchangeability and the
conditional spin law give, for every bounded Borel test function
$F$ of finitely many signed overlaps $R_{ij}$ with $i<j$, 
\begin{equation}\label{eq:siteidentity}
 \E[F R_{12}]=\E[F M_1M_2].
\end{equation}
\end{lemma}

\begin{proof}
Equation~\zeqref{eq:active-moment} follows from
Corollary~\zcref[noname]{cor:active-moments}, since
$r\mapsto r\kappa(r)$ is bounded and continuous on $[-1,1]$.

Now let $F$ be a bounded continuous function of finitely
many signed overlaps. Write $R^{(N+1)}$ for the overlap
array of replicas from the full $(N+1)$-site system.
The overlaps are unchanged by a common permutation of
the sites, so site exchangeability gives
\[
 \begin{aligned}
 \E\langle F(R^{(N+1)})R_{12}^{(N+1)}\rangle
 &=\frac1{N+1}\sum_{i=1}^{N+1}
   \E\langle F(R^{(N+1)})\sigma_i^1\sigma_i^2\rangle\\
 &=\E\langle
   F(R^{(N+1)})\sigma_{N+1}^1\sigma_{N+1}^2\rangle.
 \end{aligned}
\]
By the joint convergence in
Lemma~\zcref[noname]{cav:representation}, this identity
passes to the successor limit as
\[
 \E[F(R)R_{12}]=\E[F(R)\epsilon_1\epsilon_2].
\]
Conditional on $G,Y$ and all sampled vectors appearing
in this expression, the signs $\epsilon_1,\epsilon_2$
are independent with means $M_1,M_2$. This proves
\zeqref{eq:siteidentity} for continuous $F$.
For each fixed finite set of overlap coordinates, the
two sides define finite signed Borel measures. Their
equality against continuous functions identifies these
measures, extending the identity to bounded Borel $F$.
\end{proof}

\begin{lemma}\label{cav:gaussian}
Fix $n\ge1$ and $1\le i\le n$. In a successor representation,
conditional on $G,Y$, let $v^1,\ldots,v^{n+1}$ be independent
samples from $\widehat G$. Set
\[
 Y_j=Y(v^j),\qquad M_j=\tanh(h+Y_j),
 \qquad 1\le j\le n+1,
\]
and write $R_{jk}=v^j\cdot v^k$ and
$\mathbf Y=(Y_1,\ldots,Y_n)$.
Let $F$ be a bounded Borel function of the signed overlaps
$(R_{jk})_{1\le j<k\le n}$.
Suppose $g\in C^1(\R^n)$ satisfies
\[
 |g(y)|+\sum_{j=1}^n|\partial_jg(y)|
 \le C\exp\left(c\sum_{j=1}^n|y_j|\right)
\]
for some $C,c>0$. Then
\begin{align}
 \E[FY_i g(\mathbf Y)]
 &=\E\Bigl[F\Bigl(
   \sum_{j=1}^n\kappa(R_{ij})\partial_jg(\mathbf Y)
   +g(\mathbf Y)\sum_{\ell=1}^n\kappa(R_{i\ell})M_\ell
   -n g(\mathbf Y)\kappa(R_{i,n+1})M_{n+1}
   \Bigr)\Bigr].
 \label{cav:gaussian-multivariate}
\end{align}
The diagonal entries in this identity are
$R_{ii}=\|v^i\|^2$. All terms are integrable.
\end{lemma}

\begin{proof}
We condition on $G$ and express sampling from
$\widehat G^{\otimes n}$ using its density relative to
$G^{\otimes n}$,
\[
 \rho_Y(v^1,\ldots,v^n)
 =Z_h^{-n}\prod_{\ell=1}^n
   a(v^\ell)\cosh(h+Y(v^\ell)).
\]
This lets us differentiate with respect to the Gaussian
field while holding the sampled vectors, and hence $F$,
fixed.

First consider a finite Gaussian expansion
\[
 Y(v)=\sum_{\alpha=1}^m\zeta_\alpha f_\alpha(v),
 \qquad
 K(v,w)=\sum_{\alpha=1}^m f_\alpha(v)f_\alpha(w),
\]
where the $\zeta_\alpha$ are independent standard Gaussian
variables. We prove the identity with $K(v^i,v^j)$ in
place of $\kappa(R_{ij})$, keeping $a(v)$ fixed.
Differentiating the normalizing constant gives
\[
 \partial_{\zeta_\alpha}\log Z_h
 =\frac{G[a f_\alpha\sinh(h+Y)]}{Z_h}
 =\widehat G[f_\alpha M].
\]
Therefore
\[
 \partial_{\zeta_\alpha}\rho_Y
 =\rho_Y\left[
   \sum_{\ell=1}^n f_\alpha(v^\ell)M_\ell
   -n\widehat G[f_\alpha M]
 \right].
\]
The sum comes from differentiating the $n$ factors
$\cosh(h+Y(v^\ell))$, and the final term comes from
differentiating $Z_h^{-n}$.

For fixed vectors, Gaussian integration by parts gives
\[
 \E_\zeta[Y_i g(\mathbf Y)\rho_Y]
 =\sum_{\alpha=1}^m f_\alpha(v^i)
   \E_\zeta\!\left[
     \partial_{\zeta_\alpha}
       \bigl(g(\mathbf Y)\rho_Y\bigr)
   \right].
\]
The derivative of $g$ is
$\sum_j f_\alpha(v^j)\partial_jg(\mathbf Y)$.
Combining this with the derivative of $\rho_Y$ and summing
over $\alpha$ gives the three terms in
\zeqref{cav:gaussian-multivariate}, with covariance kernel
$K$. For the denominator term, use
\[
 \int K(v^i,w)M(w)\,\widehat G(dw)
 =
 \E\!\left[
 K(v^i,v^{n+1})M_{n+1}
 \,\middle|\,G,Y,v^1,\ldots,v^n
 \right].
\]
Multiplying by $F$ and integrating the fixed vectors
against $G^{\otimes n}$ proves the finite-dimensional
identity.

For the general field, take finite Gaussian projections
$Y^{(m)}$ with covariance kernels $K_m$. These satisfy
\[
 K_m(v,w)\longrightarrow\kappa(v\cdot w),
 \qquad
 |K_m(v,w)|\le\kappa(1),
\]
including on the diagonal. Under the coupling by Gaussian
projection,
\[
 \E\bigl[|Y^{(m)}(v)-Y(v)|^2\mid G,v\bigr]
 =\kappa(q(v))-K_m(v,v)\longrightarrow0.
\]
The variance bound and Gaussian moment formulas therefore
give convergence in every finite $L^p$, jointly over the
field and independent samples from $G$.

Define the projected sampling measures using $Y^{(m)}$
and the same function $a$. Their normalizers are at least
one and converge to $Z_h$ in every finite $L^p$.
The Gaussian exponential bounds used in
\zeqref{cav:fieldtails} hold uniformly for these projections.
Together with the growth assumptions on $g$ and its
derivatives, they give uniform integrability of every
term in the finite-dimensional identity and permit passage
to the limit. Note that $F$ is fixed during this
approximation. 
This proves the identity and the asserted integrability.
\end{proof}

\subsection{Squared norms and nearby replicas}\label{subsec:cavity-observables}
Let $\mathscr R$ be the completed sigma-field generated by all
successor off-diagonal overlaps. The next two lemmas recover squared norms from this array and bound
the difference between cavity magnetizations at nearby directing vectors.

\begin{lemma}\label{lem:norm-observable}
For each $i$, put $q_i=\|v^i\|^2$. With $\mathscr R$
completed under the joint law, the variables $q_i$ and,
for every fixed deterministic $S$, the  mass
$w_S=\widehat G(\{v:\|v\|^2=S\})$ are
$\mathscr R$-measurable.
More explicitly, complete the observed overlap array by
\[
 \overline R_{jj}=1,\qquad
 \overline R_{jk}=R_{jk}\quad(j\ne k).
\]
Then, almost surely, simultaneously for every $i$,
\begin{equation}\label{eq:norm-observable}
 q_i=
 \sup_{\substack{J\subset\mathbb N\setminus\{i\}\ {\rm finite}\\
                 (c_j)_{j\in J}\in\mathbb Q^J}}
 \left\{2\sum_{j\in J}c_jR_{ij}
       -\sum_{j,k\in J}c_jc_k\overline R_{jk}\right\},
\end{equation}
where the empty set is allowed and contributes zero.
Consequently, for distinct $i,j$ and every $\epsilon>0$,
\begin{equation}\label{mark:observable-ball}
 \{v^j\in B_\epsilon(v^i)\}
 =\{q_i+q_j-2R_{ij}<\epsilon^2\}\in\mathscr R.
\end{equation}
For every fixed deterministic $S$,
\begin{equation}\label{mark:observable-shell-mass}
 w_S=\lim_{n\to\infty}\frac1n
             \sum_{i=1}^n\one_{\{q_i=S\}}
 \quad\text{almost surely}.
\end{equation}
\end{lemma}

\begin{proof}
We give an explicit countable version of the Gram
reconstruction in the proof of \cite[Lemma~4]{DS}.
The point to check is that the reconstruction still recovers
$q_i$ when the diagonal entries are set to one. We will
approximate $v^i$ by averages of many nearby replicas, making
the contribution of this diagonal discrepancy arbitrarily
small.

First fix $i$ and one choice of $J$ and $(c_j)_{j\in J}$
in \zeqref{eq:norm-observable}, and put
$z=\sum_{j\in J}c_jv^j$. Since the off-diagonal entries
are the inner products of the sampled vectors,
\[
 \begin{aligned}
 2\sum_{j\in J}c_jR_{ij}
 -\sum_{j,k\in J}c_jc_k\overline R_{jk}
 &=2v^i\cdot z-\|z\|^2
      -\sum_{j\in J}c_j^2(1-q_j)\\
 &=q_i-\|v^i-z\|^2
      -\sum_{j\in J}c_j^2(1-q_j)
 \le q_i.
 \end{aligned}
\]
Here the last inequality uses $q_j\le1$. Thus the supremum
is at most $q_i$.

For the reverse inequality, condition on $G,Y,v^i$.
Almost surely, $v^i$ belongs to the norm support of
$\widehat G$, so every ball $B_\epsilon(v^i)$ has positive
$\widehat G$-mass. Under this conditioning, the remaining
replicas are independent with law $\widehat G$. Hence, for
every positive rational $\epsilon$, infinitely many of them
lie in $B_\epsilon(v^i)$ almost surely.

Choose $m$ distinct such replica labels, let $J$ be their
set, and take $c_j=1/m$ for each $j\in J$. Then
\[
 \left\|v^i-\frac1m\sum_{j\in J}v^j\right\|<\epsilon,
 \qquad
 \sum_{j\in J}c_j^2(1-q_j)\le\frac1m.
\]
The corresponding expression in the supremum is therefore
at least $q_i-\epsilon^2-1/m$. Letting $m\to\infty$ and then
letting positive rational $\epsilon\rightarrow 0$ proves
\zeqref{eq:norm-observable}. The intersections needed over
$i$ and rational $\epsilon$ are countable, so the identity
holds simultaneously for all $i$ on one event of probability
one.

The supremum in \zeqref{eq:norm-observable} is countable,
and each expression uses only the observed off-diagonal
overlaps and the fixed diagonal entries. This proves
measurability of every $q_i$. The identity
\[
 \|v^i-v^j\|^2=q_i+q_j-2R_{ij}
\]
then gives \zeqref{mark:observable-ball}.

Finally, fix a deterministic $S$. Conditional on $G,Y$,
the indicators $\one_{\{q_i=S\}}$ are independent with
common mean $\widehat G(q=S)$. The conditional strong law
proves \zeqref{mark:observable-shell-mass}, and its right-hand
side is $\mathscr R$-measurable. This assertion holds outside
a null set for each fixed $S$, and outside one common null
set for any prescribed countable collection of values of $S$.
\end{proof}

\begin{lemma}\label{mark:ball-approximation}
Use the cavity representation in
Lemma~\zcref[noname]{cav:representation}, and let
$B_\epsilon(v)=\{w:\|w-v\|<\epsilon\}$ for $\epsilon>0$.
Sample $G$ and its conditional Gaussian field $Y$.
Conditional on $G,Y$, sample $v$ from $\widehat G$.
Conditional on $G,Y,v$, sample $w$ from
$\widehat G(\cdot\mid B_\epsilon(v))$.
Let $\E^{\rm ball}$ denote expectation over $G,Y,v,w$
under this sampling law.

Put
\[
 L=\kappa(1),\qquad
 C_\xi=\sum_{p\in2\mathbb N}p(p-1)^2b_p^2<\infty,\qquad
 C(L)=\left(2\sqrt L+\sqrt{4L+1}\right)^2.
\]
Then
\begin{equation}\label{mark:ball-estimate}
 \E^{\rm ball}|Y(w)-Y(v)|^2
       \le C(L)C_\xi\epsilon^2,\qquad
 \E^{\rm ball}|M(w)-M(v)|^2
       \le C(L)C_\xi\epsilon^2.
\end{equation}
\end{lemma}

\begin{proof}
We prove the estimate conditional on $G$, with a constant
independent of $G$. Write $\E_Y$ for expectation over the
Gaussian randomness conditional on $G$. We first bound
the variance of $Y(w)-Y(v)$ at fixed vectors, and then
account for the dependence of the sampling law on $Y$.

Define the tensor-valued map
\[
 \phi_\xi(v)=
 \bigoplus_{p\in2\mathbb N}\sqrt p\,b_p\,v^{\otimes(p-1)}.
\]
For vectors in the unit ball,
\[
 K(v,w):=\langle\phi_\xi(v),\phi_\xi(w)\rangle
       =\kappa(v\cdot w),\qquad
 \|\phi_\xi(v)\|^2=\kappa(\|v\|^2)\le L.
\]
Expanding the difference of the tensor powers one factor
at a time gives
\[
 \|v^{\otimes(p-1)}-w^{\otimes(p-1)}\|
       \le(p-1)\|v-w\|.
\]
Consequently,
\begin{equation}\label{mark:ball-increment}
 \begin{split}
 \sigma^2(v,w)
 &:=\E_Y|Y(w)-Y(v)|^2=\|\phi_\xi(w)-\phi_\xi(v)\|^2
 \le C_\xi\|v-w\|^2.
 \end{split}
\end{equation}
Exponential summability of the mixture coefficients gives
$C_\xi<\infty$. Put $\chi=C_\xi\epsilon^2$.

To differentiate the sampling density while holding the
vectors fixed, we express it relative to a probability
measure independent of $Y$. For $B_v=B_\epsilon(v)$, set
\[
 b(v)=G(B_v),\qquad
 G_{B_v}(dw)=\frac{\one_{B_v}(w)}{b(v)}G(dw),\qquad
 Q_0(dv,dw)=G(dv)G_{B_v}(dw).
\]
For $G$-almost every $v$, every open ball centered at $v$
has positive $G$-mass, so $b(v)>0$. Define $G_{B_v}$
arbitrarily on the remaining null set.

Retain the factor $a$ from \zeqref{eq:cavity-tilt}, and write
\[
 a(v)=e^{[L-\kappa(\|v\|^2)]/2},\qquad
 W(v)=a(v)\cosh(h+Y(v)),
\]
\[
 Z_h=G(W),\qquad
 \overline Z_h(v)=G_{B_v}(W).
\]
Since $1\le a\le e^{L/2}$, both normalizing constants are
at least one. They are finite almost surely, since $G(W)$
is finite and $b(v)>0$ for $G$-almost every $v$.
The two sampling steps have densities
\[
 \widehat G(dv)=\frac{W(v)}{Z_h}G(dv),\qquad
 \widehat G(dw\mid B_v)
 =\frac{W(w)}{\overline Z_h(v)}G_{B_v}(dw).
\]
Thus their joint density relative to $Q_0$ is
\[
 f(v,w;Y)=\frac{W(v)W(w)}{Z_h\overline Z_h(v)},
\]
and
\begin{equation}\label{mark:ball-two-normalizers}
 \int f\,dQ_0=1
 \qquad\text{almost surely}.
\end{equation}
Indeed, integrating first in $w$ gives $W(v)/Z_h$.
The factor $b(v)$ is included in the reference measure
$Q_0$, while the density retains both normalizing constants
that depend on $Y$.

We first perform Gaussian integration by parts for finite
projections of the field. Choose an orthonormal basis
$(e_j)$ of the Hilbert space containing the vectors
$\phi_\xi(v)$, and let $(\zeta_j)$ be independent standard
Gaussian variables. We may realize $Y$ as the Gaussian
series with partial sums
\[
 Y^{(m)}(z)=
 \sum_{j=1}^m\zeta_j\langle\phi_\xi(z),e_j\rangle.
\]
Let $P_m$ be the orthogonal projection onto the span of
$e_1,\ldots,e_m$. The covariance of $Y^{(m)}$ is
\[
 K_m(z,z')=\langle P_m\phi_\xi(z),P_m\phi_\xi(z')\rangle.
\]
Orthogonal projection gives
\[
 K_m(z,z)\le L,\qquad
 \sigma_m^2(v,w)
 :=\|P_m(\phi_\xi(w)-\phi_\xi(v))\|^2
 \le\sigma^2(v,w)\le\chi
\]
for $Q_0$-almost every pair.

Define $W_m,Z_{h,m},\overline Z_{h,m},f_m$, and
$\widehat G_m$ by replacing $Y$ with $Y^{(m)}$ in the
preceding formulas. Keep $a$, the balls, and $Q_0$
unchanged, and put
\[
 M_m(z)=\tanh(h+Y^{(m)}(z)).
\]
In particular, both normalizing constants remain at least
one, and $\int f_m\,dQ_0=1$.

Fix a pair $(v,w)$ in the reference integral and set
\[
 D_m=Y^{(m)}(w)-Y^{(m)}(v),\qquad
 c_m(z)=K_m(w,z)-K_m(v,z).
\]
The differential operator
\[
 \mathcal D_m
 =\sum_{j=1}^m
   \langle\phi_\xi(w)-\phi_\xi(v),e_j\rangle
   \frac{\partial}{\partial\zeta_j}
\]
satisfies
\[
 \mathcal D_mY^{(m)}(z)=c_m(z),\qquad
 \mathcal D_mD_m=\sigma_m^2(v,w).
\]
Differentiating the two numerator factors and the two
normalizing constants gives
\begin{align}
 \mathcal S_m
 :=\mathcal D_m\log f_m
 ={}&c_m(v)M_m(v)+c_m(w)M_m(w)\notag\\
 &-\widehat G_m[c_m(z)M_m(z)]\notag\\
 &-\widehat G_m[c_m(z)M_m(z)\mid B_v].
 \label{mark:ball-score}
\end{align}
The last two terms come from $Z_{h,m}$ and
$\overline Z_{h,m}(v)$, respectively. The vectors, the ball,
its $G$-mass, and $a$ are fixed during differentiation.

Cauchy--Schwarz in the feature space gives
\[
 |c_m(z)|
 \le\sqrt{K_m(z,z)\sigma_m^2(v,w)}
 \le\sqrt{L\sigma_m^2(v,w)}.
\]
Since $|M_m|\le1$ and both averages in
\zeqref{mark:ball-score} are under probability measures,
\begin{equation}\label{mark:ball-score-bound}
 |\mathcal S_m|
 \le4\sqrt{L\sigma_m^2(v,w)}
 \le4\sqrt{L\chi}.
\end{equation}
This bound does not involve $b(v)$.

The integrals needed for Gaussian integration by parts
are finite. Indeed,
\[
 0\le f_m
 \le W_m(v)W_m(w)
 \le e^L\cosh(h+Y^{(m)}(v))
          \cosh(h+Y^{(m)}(w)).
\]
Both field values have variance at most $L$. Gaussian
exponential moments therefore bound every fixed polynomial
moment of $D_m$ times the right side, uniformly in $m$
and the reference pair.

Since $D_m$ is the linear combination of Gaussian
coordinates defining $\mathcal D_m$, integration by parts
gives
\[
 \E_Y[D_m^2f_m]
 =\E_Y[\mathcal D_m(D_mf_m)]
 =\sigma_m^2(v,w)\E_Y f_m
   +\E_Y[D_mf_m\mathcal S_m].
\]
Put
\[
 A_m=\int\E_Y[D_m^2f_m]\,dQ_0.
\]
Integrating the preceding identity, using
$\sigma_m^2\le\chi$ and \zeqref{mark:ball-score-bound},
and applying Cauchy--Schwarz under the probability measure
with density $f_m$ gives
\[
 A_m\le\chi+4\sqrt{L\chi}\sqrt{A_m}.
\]
For $\chi>0$, the number $x=\sqrt{A_m/\chi}$ satisfies
$x^2-4\sqrt L\,x-1\le0$. Hence
\[
 A_m\le
 \left(2\sqrt L+\sqrt{4L+1}\right)^2\chi=C(L)\chi.
\]
If $\chi=0$, then $D_m=0$ almost surely for
$Q_0$-almost every pair, so $A_m=0$.

It remains to pass from the projected fields to $Y$.
Under the Gaussian series coupling,
\[
 \E_Y|Y^{(m)}(z)-Y(z)|^2
 =\|(I-P_m)\phi_\xi(z)\|^2\longrightarrow0.
\]
The $v$-marginal of $Q_0$ is $G$, and its $w$-marginal
is absolutely continuous with respect to $G$. The variance
bound therefore gives convergence of both field values
in $L^2$ under the product of the Gaussian law and $Q_0$.
Gaussian exponential moments also give
\[
 \int\E_Y\bigl[
   |W_m(v)-W(v)|+|W_m(w)-W(w)|
 \bigr]\,dQ_0\longrightarrow0.
\]
The normalizing constants converge by the inequalities
\[
 \begin{split}
 &\E_Y|Z_{h,m}-Z_h|
 +\int\E_Y|\overline Z_{h,m}(v)-\overline Z_h(v)|\,G(dv)\\
 &\qquad\le
 \int\E_Y\bigl[
   |W_m(v)-W(v)|+|W_m(w)-W(w)|
 \bigr]\,dQ_0.
 \end{split}
\]
Their lower bound one implies convergence of their
inverses. Consequently,
\[
 D_m^2f_m\longrightarrow
 |Y(w)-Y(v)|^2f(v,w;Y)
\]
in probability under the product of the Gaussian law
and $Q_0$. Fatou's lemma along an almost surely convergent
subsequence yields
\[
 \int\E_Y\bigl[
   |Y(w)-Y(v)|^2f(v,w;Y)
 \bigr]\,dQ_0
 \le C(L)\chi.
\]
Averaging over $G$ proves the first bound in
\zeqref{mark:ball-estimate}. The second follows from
$M(z)=\tanh(h+Y(z))$ and the fact that
$x\mapsto\tanh(h+x)$ is $1$-Lipschitz.
\end{proof}

\subsection{Conditional magnetization laws}
\label{subsec:conditional-magnetizations}

Auffinger and Jagannath have previously obtained a branching-diffusion
representation of spin distributions under the cavity equations
and the Ghirlanda--Guerra identities
\cite[Theorem~1.2]{AuffingerJagannathSpinDistributions}.
Here we derive the required conditional magnetization laws from
the finite-replica concentration estimates for the unperturbed model,
and subsequently use them to prove the full identities.

\begin{proposition}\label{mark:conditional}
Use the cavity representation of
Lemma~\zcref[noname]{cav:representation}, with predecessor
directing measure $G$, Gaussian field $Y$, and successor
directing measure $\widehat G$. Conditional on $G,Y$,
sample $v^1,v^2,\ldots$ independently from $\widehat G$,
and put
\[
 q_i=\|v^i\|^2,\qquad
 M_i=\tanh(h+Y(v^i)).
\]
Let $\mathscr R$ be the completed sigma-field generated
by the overlaps $(v^i\cdot v^j)_{i<j}$. By
Lemma~\zcref[noname]{lem:norm-observable}, $q_1$ is
$\mathscr R$-measurable. The conditional law of $M_1$
given $\mathscr R$ is $\nu_{q_1,h}$.  Explicitly, for every bounded
$\mathscr R$-measurable $F$ and every bounded Borel
$f:[-1,1]\to\R$,
\begin{equation}\label{mark:onepoint}
 \E[F f(M_1)]=\E[F\nu_{q_1,h}(f)].
\end{equation}

For each fixed deterministic $S\in[0,1]$, conditional on
$\mathscr R$, the pair $(M_1,M_2)$ has law
$\nu_{R_{12}}^{S,h}$ on the event $\{q_1=q_2=S\}$.
Thus, for every bounded $\mathscr R$-measurable $F$ and
every bounded Borel $\psi:[-1,1]^2\to\R$,
\begin{equation}\label{mark:pair}
 \E[F\one_{\{q_1=q_2=S\}}\psi(M_1,M_2)]
 =\E[F\one_{\{q_1=q_2=S\}}\nu_{R_{12}}^{S,h}(\psi)].
\end{equation}
The right-hand integrand is defined to be zero outside
$\{q_1=q_2=S\}$. On this event, $R_{12}\ge0$ almost surely
when $h\ne0$.

Moreover, almost surely, simultaneously for every $i$,
\begin{equation}\label{mark:radiuscontact}
 q_i=\Gamma(q_i).
\end{equation}
For each fixed deterministic $S\in[0,1]$, almost surely
on $\{q_1=q_2=S\}$,
\begin{equation}\label{mark:paircontact}
 \begin{aligned}
 R_{12}&=\Gamma(R_{12}) &&\text{if }h\ne0,\\
 |R_{12}|&=\Gamma(|R_{12}|) &&\text{if }h=0.
 \end{aligned}
\end{equation}
\end{proposition}

\begin{proof}
We identify the conditional laws by their moments. To obtain a
power of the magnetization at a sampled vector, we select several
replicas near that vector and compare the product of their added-site
spins with the corresponding scalar moment. We first derive moment estimates conditional on the full overlap array. 

Let $\varepsilon^i$ be the added-site spin associated with $v^i$ in
Lemma~\zcref[noname]{cav:representation}. Conditional on $G,Y$ and
the entire sequence of vectors, these spins are independent with
means $M_i$. All expectations below are under this joint cavity law,
including the auxiliary replicas introduced in Step~3.

\emph{1. Moment estimates conditional on the full overlap array.}
Fix a one-group or two-group construction from
\zcref{src:replica-groups}, with $d$ replicas, deterministic scalar
parameters, prescribed overlap matrix $D$, and spin monomial $P_0$.
Let $c$ be its scalar coefficient from
Lemma~\zcref[noname]{src:replica-coefficient}. Thus $c=c(P_0;S)$
in the one-group case and $c=c(P_0;r,S)$ in the two-group case.
Fix $\delta>0$ and choose a width $\eta>0$ allowed by
Lemma~\zcref[noname]{src:source-suppression}.

For a deterministic ordered tuple $t=(i_1,\ldots,i_d)$ of
distinct replica labels, put
\[
 Z_t=P_0(\varepsilon^{i_1},\ldots,\varepsilon^{i_d}),
 \qquad
 \mathcal C_\eta^t(D)=
 \left\{|R_{i_ai_b}-D_{ab}|<\eta
          \text{ for every }a<b\right\}.
\]
We claim that
\begin{equation}\label{mark:source-conditional-bound}
 \left|\E[Z_t\mid\mathscr R]-c\right|\le\delta
 \quad\text{almost surely on }\mathcal C_\eta^t(D).
\end{equation}
The conditional expectation in this estimate retains all overlaps,
including those involving labels outside $t$.

To prove the claim, let $T$ be any nonnegative continuous function
of finitely many overlap coordinates. For $d\ge2$, define
\[
 \chi(R)=\min\left\{1,
   \left(\eta-\max_{a<b}|R_{i_ai_b}-D_{ab}|\right)_+
 \right\}.
\]
This function is continuous, takes values in $[0,1]$, and is
positive exactly on $\mathcal C_\eta^t(D)$. For $d=1$, use
$\chi=1$, since there is no overlap constraint.

In the full $N$-site system, let $T_N$ and $\chi_N$ be the
same functions of the overlaps, and let
$Z_{t,N}=P_0(\sigma_N^{i_1},\ldots,\sigma_N^{i_d})$ be the
spin monomial at the site removed in the cavity construction.
The empirical spin observable $A_N$ from \zeqref{src:replica-pressure}
is the average of this monomial over all sites. A common
permutation of sites leaves the disorder law and all overlap
factors unchanged. Each term in the average therefore has the
same expectation, giving
\[
 \E\langle T_N\chi_N A_N\rangle
 =\E\langle T_N\chi_N Z_{t,N}\rangle.
\]
On $\{|A_N-c|\le\delta\}$, the difference $|A_N-c|$ is at
most $\delta$. On its complement, it is at most two, since
$|A_N|,|c|\le1$. Lemma~\zcref[noname]{src:source-suppression}
thus gives
\begin{equation}\label{mark:source-finite-bound}
 \left|\E\langle T_N\chi_N(Z_{t,N}-c)\rangle\right|
 \le\delta\,\E\langle T_N\chi_N\rangle
       +2\|T\|_\infty C e^{-c_0N}.
\end{equation}
All additional labels used by $T_N$ are included in this
expectation. After bounding $T_N$ by its supremum and integrating out
those labels, apply Lemma~\zcref[noname]{src:source-suppression}
to the event that the overlap constraints hold and
$|A_N-c|>\delta$.

We now pass to the given successor limit. The joint convergence
of overlaps and added-site spins in
Lemma~\zcref[noname]{cav:representation} applies because the
test functions are bounded and continuous and use finitely many
labels. At full-system size $N$, each full overlap differs from
its reduced-system counterpart by at most $2/N$, so either
choice gives the same limit. Consequently,
\[
 \left|\E[T\chi(Z_t-c)]\right|
 \le\delta\,\E[T\chi].
\]

We now extend this inequality to arbitrary overlap events.
Let $\mathcal L$ be the limiting overlap law on the compact
array space $\Omega$, and define the finite signed measure
\[
 \zeta_\chi(A)=
 \E[\one_{\{R\in A\}}\chi(R)(Z_t-c)].
\]
The preceding inequality says that both measures
$\delta\chi\mathcal L+\zeta_\chi$ and
$\delta\chi\mathcal L-\zeta_\chi$ integrate every nonnegative
continuous cylinder function to a nonnegative number.
Such functions uniformly approximate every nonnegative
continuous function on $\Omega$. The characterization of
positive Borel measures by their integrals against continuous
functions therefore shows that both measures are nonnegative.
The inequality extends to every bounded nonnegative Borel $T$,
and hence
\[
 \left|\chi\,\E[Z_t-c\mid\mathscr R]\right|
 \le\delta\chi
 \quad\text{almost surely}.
\]
Dividing on $\{\chi>0\}$ proves
\zeqref{mark:source-conditional-bound}. Completing the
overlap sigma-field does not change these expectations.

The later selections require these bounds to hold simultaneously.
Use all rational scalar parameter points in the permitted
domains, all integer group sizes, all their spin monomials, all
deterministic tuples of distinct labels, and accuracies
$\delta=2^{-m}$, $m\ge1$. For each construction, group size,
and accuracy, choose one positive width that works for every
parameter point and monomial; this is possible by the uniformity
in Lemma~\zcref[noname]{src:source-suppression}. Fix versions of
the conditional expectations of the spin monomials and intersect
the resulting countably many events of probability one. On that
intersection, all the required conditional bounds hold.

\emph{2. Positivity of the overlaps at nonzero field.}
We verify that, when $h\ne0$, the pair kernel will be evaluated
within its domain $0\le r\le S$. Let $q_0=\min\supp\mu$.
The contact identity \zeqref{eq:contact} and the martingale
identity \zeqref{eq:martingale-u} give
\[
 q_0=\Gamma(q_0)
 =\E u(q_0,X_{q_0})^2
 \ge\bigl(\E u(q_0,X_{q_0})\bigr)^2
 =u(0,h)^2>0.
\]
The last inequality uses $u(0,0)=0$ and the strict increase
of $x\mapsto u(0,x)$.

We apply \cite[Theorem~7(i), Eq.~(29)]{ChenGT}. In its
equations (1)--(4), take $\beta=1$ and $\gamma_p=b_p$.
The covariance is then $N\xi(R)$, and the spin sum uses
counting measure, as here. The function $\xi$ is nonzero,
even, and convex on $[-1,1]$, its coefficients satisfy the
required exponential summability by \zeqref{eq:mixture},
and the deterministic field is nonzero. The terminal potential
in that reference is $\log\cosh$, with $\log2$ added separately
to the functional; adding this constant to the potential gives
our convention and leaves the minimizer unchanged. The theorem
yields, for every $\epsilon>0$,
\begin{equation}\label{eq:positive-overlap}
 \E\left\langle
   \one_{\{R_{12}\le q_0-\epsilon\}}
 \right\rangle_{N,h}
 \le C_\epsilon e^{-c_\epsilon N}.
\end{equation}
Passing to the limit with continuous cutoffs supported below
$q_0$, and exhausting $[-1,q_0)$ by countably many such cutoffs,
gives $R_{12}\ge q_0$ almost surely. Replica exchangeability
and a countable intersection give this conclusion simultaneously
for every pair of distinct labels.

\emph{3. Selecting nearby replicas using the overlap array.}
We use additional replicas to select vectors near $v^1$ and $v^2$,
while keeping all overlaps among the original replicas in the
conditioning. Conditional on $G,Y$,
sample two infinite sequences of vectors independently from
$\widehat G$, independently also of the original sequence,
and attach their added-site spins as in
Lemma~\zcref[noname]{cav:representation}. Assign these vectors
labels distinct from the original labels. Any fixed enumeration
of the enlarged family has the same conditional independent
sampling law as the original sequence. Thus Step~1 and
Lemma~\zcref[noname]{lem:norm-observable} apply to its full
overlap array. Write $\mathscr R^+$ for the completed
sigma-field generated by this enlarged array; it contains
$\mathscr R$. We use the simultaneous conditional bounds
from Step~1 with $\mathscr R^+$ in place of $\mathscr R$.

For a positive rational $\epsilon$ and an integer $k\ge1$,
take the first $k$ vectors in auxiliary sequence $g$ that
belong to $B_\epsilon(v^g)$, for $g=1,2$. Denote them by
$w_\epsilon^{g,1},\ldots,w_\epsilon^{g,k}$, and let
$T_\epsilon^{g,j}$ be their labels. Almost surely, each
center $v^g$ lies in the support of $\widehat G$, so its
ball has positive mass and all these labels are finite.
By \zeqref{mark:observable-ball}, membership in the ball is
determined by the enlarged overlap array. The selected labels
are therefore $\mathscr R^+$-measurable.

Conditional on $G,Y,v^1,v^2$, the accepted vectors in group
$g$ are independent with law
$\widehat G(\cdot\mid B_\epsilon(v^g))$, and the two groups
are independent. To verify the normalization, put
$p=\widehat G(B_\epsilon(v^g))>0$. For a Borel set $A$,
summing over the number of trials before the first acceptance
gives
\[
 \sum_{\ell=0}^\infty(1-p)^\ell
       \widehat G(A\cap B_\epsilon(v^g))
 =\frac{\widehat G(A\cap B_\epsilon(v^g))}{p}.
\]
The unused trials remain independent with law $\widehat G$,
so this calculation applies at every successive acceptance.

To use the conditional moment estimates at these random labels,
let $T_\epsilon$ denote a selected ordered tuple and, as in
Step~1, let $Z_t$ denote its spin monomial at deterministic
labels $t$. Partitioning by the countably many possible tuples
gives
\begin{equation}\label{mark:source-selected-labels}
 \E[Z_{T_\epsilon}\mid\mathscr R^+]
 =\sum_t\one_{\{T_\epsilon=t\}}
        \E[Z_t\mid\mathscr R^+].
\end{equation}
Indeed, each selection event is $\mathscr R^+$-measurable,
and the series is dominated in absolute value by one. Thus
the conditional bound for each deterministic tuple also holds
for the selected tuple whenever it satisfies the corresponding
overlap constraint.

The selected vectors satisfy the required constraints because
their inner products are close to those of their centers.
Since all vectors have norm at most one,
\begin{equation}\label{mark:source-selected-overlaps}
 \begin{aligned}
 |w_\epsilon^{g,j}\cdot w_\epsilon^{g,\ell}-q_g|
 &<2\epsilon &&(j\ne\ell),\\
 |w_\epsilon^{1,j}\cdot w_\epsilon^{2,\ell}-R_{12}|
 &<2\epsilon.
 \end{aligned}
\end{equation}
For example, the first bound follows from
\[
 |w_\epsilon^{g,j}\cdot w_\epsilon^{g,\ell}
      -v^g\cdot v^g|
 \le\|w_\epsilon^{g,j}-v^g\|\,
             \|w_\epsilon^{g,\ell}\|
       +\|v^g\|\,\|w_\epsilon^{g,\ell}-v^g\|
 <2\epsilon.
\]
For the cross-group bound, use $v^1$ and $v^2$ as the two
centers in this inequality.

We also control the change in the spin moments caused by this
selection. For integers $0\le a,b\le k$, take the monomial
using the first $a$ selected spins in group one and the first
$b$ in group two, and put
\[
 P_\epsilon=
 \prod_{j=1}^a M(w_\epsilon^{1,j})
 \prod_{j=1}^b M(w_\epsilon^{2,j}).
\]
Conditional on $G,Y$ and all vectors, the selected labels are
known and their distinct spins are independent. Hence the
conditional expectation of $Z_{T_\epsilon}$ is $P_\epsilon$.
In particular, for every bounded $\mathscr R^+$-measurable $H$,
\[
 \E[H Z_{T_\epsilon}]=\E[H P_\epsilon].
\]
Each pair $(v^g,w_\epsilon^{g,j})$ has, marginally, the
sampling law of Lemma~\zcref[noname]{mark:ball-approximation}.
Telescoping the difference of two products with factors in
$[-1,1]$, and then applying that lemma and Cauchy--Schwarz
to each factor, gives
\begin{equation}\label{mark:source-product-error}
 \E|P_\epsilon-M_1^aM_2^b|
 \le(a+b)\sqrt{C(L)C_\xi}\,\epsilon,
 \qquad L=\kappa(1).
\end{equation}
The case $b=0$ uses only the first auxiliary sequence.
The bound is independent of the masses of the sampled balls.

\emph{4. One-point moments at the random squared norm.}
Fix an integer $a\ge1$, take $k=\max\{a,2\}$, and use the
one-group monomial formed from the first $a$ selected spins.
The choice $k\ge2$ supplies within-group overlaps even when
$a=1$.
We approximate the  squared norm $q_1$ by finitely many
deterministic parameters, so that Step~1 applies.

Fix $\delta=2^{-m}$ and the common admissible width $\eta$ for
this group size and accuracy. By the continuity in
Lemma~\zcref[noname]{src:replica-coefficient}, there are
finitely many rational points $s_1,\ldots,s_J\in[0,1]$ such
that, for every $s\in[0,1]$, at least one index $j$ satisfies
\[
 |s-s_j|<\eta/4,
 \qquad
 |\nu_{s,h}(z^a)-\nu_{s_j,h}(z^a)|\le\delta.
\]
Choose the least such index for $s=q_1$. This index is
$\mathscr R$-measurable because $q_1$ is $\mathscr R$-measurable and
the scalar coefficient is continuous.

If $0<\epsilon<\eta/8$, the within-group bound in
\zeqref{mark:source-selected-overlaps} gives
\[
 |w_\epsilon^{1,\ell}\cdot w_\epsilon^{1,\ell'}-s_j|
 <2\epsilon+\eta/4<\eta,
 \qquad \ell\ne\ell'.
\]
The selected tuple therefore satisfies the one-group constraint
at $s_j$. Equations~\zeqref{mark:source-conditional-bound}
and \zeqref{mark:source-selected-labels} give
\[
 \left|\E[Z_{T_\epsilon}\mid\mathscr R^+]
              -\nu_{q_1,h}(z^a)\right|\le2\delta.
\]
One $\delta$ comes from the conditional moment estimate and the other
from replacing $s_j$ by $q_1$. The choice among the finitely
many $s_j$ is valid because all their conditional estimates
hold on the common event from Step~1.

For any bounded $\mathscr R$-measurable $F$, use
Step~3 to replace the spin product by $M_1^a$. We obtain
\[
 \left|\E[F M_1^a]-\E[F\nu_{q_1,h}(z^a)]\right|
 \le\|F\|_\infty
       \left(2\delta+a\sqrt{C(L)C_\xi}\,\epsilon\right).
\]
First let positive rational $\epsilon\downarrow0$, keeping
$a,k,\delta,\eta$ and the finite parameter set fixed.
Then let $\delta=2^{-m}\downarrow0$. This proves
\[
 \E[F M_1^a]=\E[F\nu_{q_1,h}(z^a)].
\]
For $a=0$, the identity holds because $\nu_{q_1,h}$ is a
probability measure.

\emph{5. Pair moments at a fixed common squared norm.}
Fix the deterministic $S\in[0,1]$ in the statement and put
\[
 I_S=\one_{\{q_1=q_2=S\}}.
\]
Fix integers $a,b\ge0$ with $a+b>0$, take
$k=\max\{a,b,2\}$, and use two groups of size $k$ with
the spin monomial from Step~3. Coordinates absent from this
monomial remain in all the overlap constraints.

On $\{I_S=1\}$, Cauchy--Schwarz gives $|R_{12}|\le S$.
If $h\ne0$, Step~2 also gives $R_{12}\ge0$. Thus
$(R_{12},S)$ belongs to the domain of the scalar pair kernel.
Fix $\delta=2^{-m}$ and its common width $\eta$. The function
\[
 (r,s)\longmapsto\nu_r^{s,h}(z_1^az_2^b)
\]
is uniformly continuous on $0\le r\le s\le1$, and at zero
field on the larger compact domain $|r|\le s\le1$, by
Lemma~\zcref[noname]{src:replica-coefficient}. Choose finitely
many rational points in the relevant domain so that every
parameter pair is within $\eta/4$ in both coordinates of
one of them and the corresponding coefficients differ by
at most $\delta$.

On $\{I_S=1\}$, choose the least index with these properties
for $(R_{12},S)$. This choice is $\mathscr R$-measurable. When
$0<\epsilon<\eta/8$, both bounds in
\zeqref{mark:source-selected-overlaps} show that the selected
tuple satisfies the constraint at this rational parameter
point: each constrained overlap differs from its prescribed
value by less than $2\epsilon+\eta/4<\eta$.
Applying the conditional moment bound at the selected labels
and parameter point, and then accounting for the coefficient
error, gives
\begin{equation}\label{mark:source-selected-bound}
 \left|\E[FI_S Z_{T_\epsilon}]
       -\E[FI_S\nu_{R_{12}}^{S,h}(z_1^az_2^b)]\right|
 \le2\delta\|F\|_\infty
\end{equation}
for every bounded $\mathscr R$-measurable $F$.
As in the statement, the integrand in the second expectation is defined
to be zero outside $\{I_S=1\}$.

By Step~3, $Z_{T_\epsilon}$ can first be replaced in this
expectation by $P_\epsilon$. Equation~\zeqref{mark:source-product-error}
then bounds the additional error in replacing $P_\epsilon$
by $M_1^aM_2^b$ by
\[
 \|F\|_\infty(a+b)\sqrt{C(L)C_\xi}\,\epsilon.
\]
Letting positive rational $\epsilon\downarrow0$ with the
degree, accuracy, width, and parameter set fixed, and then
letting $\delta=2^{-m}\downarrow0$, proves
\[
 \E[FI_S M_1^aM_2^b]
 =\E[FI_S\nu_{R_{12}}^{S,h}(z_1^az_2^b)].
\]
The constant monomial gives the same identity when $a=b=0$.

We now pass from moments to bounded Borel test functions.
The scalar laws are Borel probability kernels by their weak
continuity. For a fixed bounded nonnegative overlap multiplier
$F$, the two sides of the one-point moment identity define
finite measures on $[-1,1]$. They agree on monomials, hence
on polynomials, and then on continuous functions by uniform
polynomial approximation. Uniqueness of finite Borel measures
on this compact interval gives equality for every bounded
Borel test function. The same argument on $[-1,1]^2$, with
the multiplier $FI_S$, applies to the pair moments.
Writing a general bounded $F$ as the difference of its positive
and negative parts proves \zeqref{mark:onepoint} and
\zeqref{mark:pair}. Since $F$ was arbitrary in the original
signed-overlap sigma-field $\mathscr R$, these are the
conditional laws stated in the proposition.

\emph{6. Identities for squared norms and overlaps.}
We combine the conditional moment estimates with the identity
that relates a spin product to an overlap. By
\zeqref{eq:siteidentity} and conditional independence of the
added-site spins, for every deterministic pair of distinct
labels $i,j$ and every bounded cylinder function $H$ of the
enlarged overlap array,
\[
 \E[H\varepsilon^i\varepsilon^j]=\E[H R_{ij}].
\]
The enlarged sequence has the same joint sampling law as the
original sequence, so the identity applies to its labels as
well. The monotone-class theorem extends it to every bounded
$\mathscr R^+$-measurable $H$, giving
\[
 \E[\varepsilon^i\varepsilon^j\mid\mathscr R^+]=R_{ij}.
\]
A countable intersection makes this identity simultaneous for
all deterministic pairs. Partitioning by selected labels as
in \zeqref{mark:source-selected-labels} then gives the same
identity for any pair of distinct labels selected using the
overlap array.

To prove \zeqref{mark:radiuscontact}, use the one-group
construction of Step~4 with $a=k=2$. Its scalar coefficient
is $\nu_{s,h}(z^2)=\Gamma(s)$. The conditional moment estimate
in that step, followed by the preceding site identity, gives
\[
 \left|R_{T_\epsilon^{1,1},T_\epsilon^{1,2}}
            -\Gamma(q_1)\right|\le2\delta
\]
on the common event of probability one, for every sufficiently
small positive rational $\epsilon$. The geometric bound in
\zeqref{mark:source-selected-overlaps} therefore yields
\[
 |q_1-\Gamma(q_1)|\le2\epsilon+2\delta.
\]
Letting rational $\epsilon\downarrow0$ and then
$\delta=2^{-m}\downarrow0$ proves $q_1=\Gamma(q_1)$.
Replica exchangeability and a countable intersection prove
the assertion simultaneously for all $i$.

For \zeqref{mark:paircontact}, apply \zeqref{mark:pair} with
$\psi(z_1,z_2)=z_1z_2$. By \zeqref{eq:branch-product},
\[
 \nu_{R_{12}}^{S,h}(z_1z_2)=
 \begin{cases}
 \Gamma(R_{12}),&h\ne0,\\
 \operatorname{sgn}(R_{12})\Gamma(|R_{12}|),&h=0
 \end{cases}
 \quad\text{on }\{I_S=1\}.
\]
The site identity with multiplier $FI_S$ gives
$\E[FI_S M_1M_2]=\E[FI_S R_{12}]$. 
Since $F$ is arbitrary,
$R_{12}=\nu_{R_{12}}^{S,h}(z_1z_2)$ almost surely on
$\{I_S=1\}$. 
This gives
$R_{12}=\Gamma(R_{12})$ when $h\ne0$. At zero field and
$R_{12}\ne0$, multiplication by $\operatorname{sgn}(R_{12})$
gives $|R_{12}|=\Gamma(|R_{12}|)$; at $R_{12}=0$, the
identity holds because $\Gamma(0)=0$.

Every thermodynamic limit in Step~1 uses finitely many
deterministic replica labels. The selections are made in the
limiting array. For each monomial, the group size, 
accuracy, width, and finite parameter set are fixed before
the ball radius tends to zero; the accuracy then tends
to zero. The pair assertions have been proved for each fixed
deterministic $S$. Any prescribed countable family of such
values can be treated on one event of probability one.
\end{proof}

\section{Reduction to spheres}\label{sec:shells}

In this section, we begin the proof of \zcref{prop:marginal-sphere} below, which identifies the distribution of $S_{12}$. We reduce its proof to identifying the overlap distribution
under restrictions of the directing measure to sets of fixed norm.
After treating the case $Q=0$, we use the conditional cavity laws
in \zcref{subsec:radial-reduction} to exclude squared norms below
$Q$ and show that the remaining squared norms belong to a
deterministic countable set. In \zcref{subsec:normalized-shells},
we restrict the directing measure to each such set of positive
mass and divide by that mass to obtain a probability measure.
We show that the cavity laws remain valid for these restrictions,
with the reweighting normalized over each restricted set.
Finally, we show that identifying the overlap distribution under
every such restriction determines the full overlap distribution
and forces the directing measure to be concentrated on vectors
of squared norm $Q$.

\subsection{Distribution of an overlap}

\begin{proposition}
\label{prop:marginal-sphere}
Under \zeqref{eq:mixture}, every physical overlap law satisfies
$S_{12}\sim\mu_{\xi,h}$, and every physical directing
measure is supported on the sphere of squared radius $Q$.
Consequently, for continuous test functions,
\begin{align}
 \lim_{N\to\infty}\E\langle f(R_{12})\rangle_{N,h}
 &=\int f\,d\mu_{\xi,h} &&(h\ne0),\label{eq:main-field}\\
 \lim_{N\to\infty}\E\langle f(|R_{12}|)\rangle_{N,0}
 &=\int f\,d\mu_{\xi,0}.\label{eq:main-zero}
\end{align}
\end{proposition}
The proof is completed in Section~\zcref[noname]{sec:regression}, where the
same argument on each sphere identifies the marginal and proves the
Ghirlanda--Guerra identities.

\begin{lemma}\label{lem:zero-measure}
If $h=0$ and $\mu=\delta_0$, then \zeqref{eq:main-zero} holds and
every physical directing measure is supported at the origin.
\end{lemma}
\begin{proof}
Choose an even interaction order $p$ with $b_p>0$.
Corollary~\zcref[noname]{cor:active-moments} gives
\[
 \lim_{N\to\infty}\E\langle |R_{12}|^p\rangle_{N,h}
 =\int q^p\,\mu(dq)=0.
\]
It follows that $R_{12}=0$ almost surely in every physical
overlap limit.

Let $G$ direct any such limit. Since $G$ is supported in
the unit ball, the operator
\[
 T_G=\int v\otimes v\,G(dv),
 \qquad (v\otimes v)x=(v\cdot x)v,
\]
is positive and trace class. Moreover,
\[
 \E\tr(T_G^2)
 =\E\int (v\cdot w)^2\,G(dv)G(dw)
 =\E R_{12}^2=0.
\]
The trace is nonnegative, so $\tr(T_G^2)=0$ almost surely.
Thus $T_G=0$ almost surely, and consequently
\[
 \int\|v\|^2\,G(dv)=\tr T_G=0.
\]
Hence $G=\delta_0$ almost surely.
\end{proof}

In the rest of this section $Q>0$. This includes every nonzero-field
case, since $q_0=\min\supp\mu>0$ when $h\ne0$, as shown in
Section~\zcref[noname]{subsec:conditional-magnetizations}.

\subsection{A radial identity and the possible squared norms}\label{subsec:radial-reduction}
To prove that $q_1\ge Q$ almost surely, we construct a scalar
function $\mathcal J_h$ bounded above by $B_\mu$, with equality
exactly on $[Q,1]$. Proposition~\zcref[noname]{prop:radial-reduction}
will use Gaussian integration by parts to show that
$\E\mathcal J_h(q_1)=B_\mu$. The nonnegative difference
$B_\mu-\mathcal J_h(q_1)$ must then vanish almost surely,
excluding squared norms below $Q$.

\begin{lemma}\label{lem:radial-invariant}
For $q\in[0,1]$, put $M_q=u(q,X_q)$ and define
\begin{equation}\label{eq:radial-invariant}
 \mathcal J_h(q)=\kappa(q)
      -\E[M_q(\arctanh M_q-h)],\qquad
 B_\mu=\int t\kappa(t)\,\mu(dt),
\end{equation}
where expectation is over the scalar diffusion.
Then $\mathcal J_h$ is absolutely continuous on $[0,1]$,
and for almost every $q\in[0,1]$,
\begin{equation}\label{eq:radial-derivative}
 \mathcal J_h'(q)=\kappa'(q)
 \left[1-\E\left(\frac{V(q,X_q)}{1-M_q^2}\right)^2\right].
\end{equation}
Moreover, $\mathcal J_h$ is strictly increasing on $[0,Q]$,
and
\begin{equation}\label{eq:radial-inequality}
 \mathcal J_h(q)<B_\mu\quad(0\le q<Q),\qquad
 \mathcal J_h(q)=B_\mu\quad(Q\le q\le1).
\end{equation}
\end{lemma}
\begin{proof}
Set
\[
 g(m)=m(\arctanh m-h),\qquad -1<m<1.
\]
Then $g''(m)=2/(1-m^2)^2$. Apply It\^o's formula to
$g(M_q)$ using \zeqref{eq:martingale-u}, first stopping at
\[
 \tau_k=\inf\{t\in[0,1]:|M_t|\ge1-k^{-1}\}\wedge1.
\]
Lemma~\zcref[noname]{lem:slack} gives
\[
 0<\frac{V(q,X_q)}{1-M_q^2}\le1,
\]
so the drift term in It\^o's formula is bounded by
$\kappa'(q)$. Moreover,
\[
 \sup_{q\le1}|g(M_q)|
 \le |h|+\sup_{q\le1}|\arctanh M_q|
\]
is integrable by Lemma~\zcref[noname]{scalar:tail}.
The paths of $M$ are continuous and remain in $(-1,1)$,
so $\tau_k\to1$ almost surely. Dominated convergence
therefore removes the stopping and yields
\[
 \mathcal J_h(q)-\mathcal J_h(0)
 =\int_0^q\kappa'(s)
   \left[1-\E\left(
     \frac{V(s,X_s)}{1-M_s^2}
   \right)^2\right]ds.
\]
This proves absolute continuity and
\zeqref{eq:radial-derivative}.

By Lemma~\zcref[noname]{lem:slack}, the ratio inside the
expectation is strictly less than one for $s<Q$ and
equals one for $s\ge Q$. Since $\kappa'(s)>0$ for $s>0$,
the integral formula shows that $\mathcal J_h$ is strictly
increasing on $[0,Q]$ and constant on $[Q,1]$.

It remains to evaluate this constant. By
\zeqref{eq:scalar-terminal}, $M_Q=\tanh X_Q$.
The martingale property gives, for $0\le t\le Q$,
\[
 \E[M_t\tanh X_Q]
 =\E[M_t\,\E(M_Q\mid\mathcal F_t)]
 =\E M_t^2=\Gamma(t),
\]
where $\mathcal F_t$ is the Brownian filtration of the
scalar diffusion. Apply \zeqref{eq:scalar-ibp} with
$r=S=Q$ and $f=\tanh$. Since
$f'(x)+\tanh(x)f(x)=1$, we obtain
\[
 \E[(X_Q-h)\tanh X_Q]
 =\kappa(Q)-\int_{[0,Q]}\kappa(t)\Gamma(t)\,\mu(dt)
 =\kappa(Q)-B_\mu.
\]
The last equality uses $\supp\mu\subset[0,Q]$ and
$\Gamma(t)=t$ on $\supp\mu$, by \zeqref{eq:contact}.
Finally, $\arctanh M_Q=X_Q$, so
\zeqref{eq:radial-invariant} gives $\mathcal J_h(Q)=B_\mu$.
The strict increase below $Q$ and constancy above $Q$
now prove \zeqref{eq:radial-inequality}.
\end{proof}

\begin{proposition}\label{prop:radial-reduction}
Almost surely, every physical directing measure is supported
on vectors whose squared norms belong to
\begin{equation}\label{eq:radial-set}
 \{Q\}\cup\mathcal Z_+,\qquad
 \mathcal Z_+=\{s\in(Q,1):\Gamma(s)=s\}.
\end{equation}
The set $\mathcal Z_+$ is deterministic and has no accumulation
point in $(Q,1)$. In particular, it is countable.
\end{proposition}

\begin{proof}
Represent an arbitrary physical overlap law as a successor
using Lemma~\zcref[noname]{cav:representation}. Apply
Lemma~\zcref[noname]{cav:gaussian} with $n=1$, $F=1$, and
$g(y)=\tanh(h+y)$. Since $\sech^2+\tanh^2=1$, this gives
\[
 \E[Y_1M_1]
 =\E\kappa(q_1)-\E[\kappa(R_{12})M_1M_2]
 =\E\kappa(q_1)-B_\mu.
\]
The last equality follows from \zeqref{eq:siteidentity},
with multiplier $\kappa(R_{12})$, and
\zeqref{eq:active-moment}.

To express the left side through the scalar law, put
$f(m)=m(\arctanh m-h)$ for $|m|<1$. Apply
\zeqref{mark:onepoint} to bounded continuous truncations
of $f$. Under the cavity law, $f(M_1)=Y_1M_1$ is integrable
by \zeqref{cav:fieldtails}. Under the scalar diffusion law, 
Lemma~\zcref[noname]{scalar:tail} gives
\[
 \E\sup_{0\le s\le1}|f(u(s,X_s))|
 \le |h|+\E\sup_{0\le s\le1}
                 |\arctanh u(s,X_s)|<\infty.
\]
Dominated convergence therefore removes the truncations,
including at the random parameter $q_1$, and yields
\[
 \E[Y_1M_1]
 =\E[\nu_{q_1,h}(f)]
 =\E\kappa(q_1)-\E\mathcal J_h(q_1).
\]
Comparison with the integration-by-parts identity proves
\[
 \E\mathcal J_h(q_1)=B_\mu.
\]
By \zeqref{eq:radial-inequality},
$B_\mu-\mathcal J_h(q_1)$ is nonnegative and strictly
positive when $q_1<Q$. Hence $q_1\ge Q$ almost surely.
Equation~\zeqref{mark:radiuscontact} also gives
$\Gamma(q_1)=q_1$, and $\Gamma(1)<1$ excludes $q_1=1$.
Thus $q_1$ belongs to the set in \zeqref{eq:radial-set}
almost surely.

We next prove that the zeros defining $\mathcal Z_+$ are
isolated. For $q>Q$, put $v=\kappa(q)-\kappa(Q)>0$.
On $[Q,1]$, the scalar drift is $\kappa'(q)\tanh x$, so
the transition density from $X_Q=x$ to $X_q=y$ is
\[
 e^{-v/2}\frac{\cosh y}{\cosh x}\varphi_v(y-x)\,dy,
\]
where $\varphi_v$ is the centered Gaussian density of
variance $v$. Define heat convolution by
\[
 (P_vf)(x)=\int_{\R}f(y)\varphi_v(y-x)\,dy.
\]
Since $u(q,x)=\tanh x$ for $q\ge Q$, integrating
$\sech^2 y$ against the transition density gives
\begin{equation}\label{eq:gamma-heat}
 1-\Gamma(q)
 =e^{-v/2}\E\bigl[
    \sech(X_Q)(P_v\sech)(X_Q)\bigr].
\end{equation}
For bounded inputs, heat convolution is real analytic
in $v>0$, with derivative bounds uniform in $x$ on
compact positive-time intervals. These bounds permit
differentiation under the expectation in
\zeqref{eq:gamma-heat}. Moreover, \zeqref{eq:mixture}
makes $\kappa$ analytic on a neighborhood of $[0,1]$.
Thus $\Gamma(q)-q$ is real analytic on $(Q,1)$.
It is not identically zero, since continuity at one
and $\Gamma(1)<1$ give
\[
 \lim_{q\uparrow1}(\Gamma(q)-q)=\Gamma(1)-1<0.
\]
Its zeros therefore have no accumulation point in
$(Q,1)$ and form a countable set. This set is deterministic
because $\Gamma$ is determined by the fixed mixture and field.

Finally, by \zeqref{eq:contact} and $\Gamma(1)<1$,
\[
 A:=\{Q\}\cup\mathcal Z_+
   =\{s\in[Q,1]:\Gamma(s)=s\}
\]
is closed. Conditioning the almost-sure restriction
$q_1\in A$ on $G,Y$ gives
\[
 \widehat G(\{v:\|v\|^2\in A\})=1
 \qquad\text{almost surely}.
\]
Lemma~\zcref[noname]{lem:norm-observable} determines sampled
squared norms from the overlap array, so the same restriction
holds for every directing measure of the chosen physical law.
That law was arbitrary, proving the proposition. Mutual
absolute continuity of $G$ and $\widehat G$ also transfers
the restriction to the predecessor $G$.
\end{proof}

\subsection{Restriction to a sphere and conditional cavity laws}
\label{subsec:normalized-shells}

We now restrict the directing measures to vectors of a fixed
squared norm and show that the conditional magnetization laws
remain valid. Lemma~\zcref[noname]{lem:norm-observable} allows
this restriction to be expressed through the overlap array.

Use the representation in
Lemma~\zcref[noname]{cav:representation}. For a deterministic
$S\in[Q,1]$, set
\[
 \Sigma_S=\{v:\|v\|^2=S\},\qquad
 w_S=\widehat G(\Sigma_S),\qquad
 E_S=\{G(\Sigma_S)>0\}.
\]
Mutual absolute continuity of $G$ and $\widehat G$ gives
$E_S=\{w_S>0\}$ almost surely. Assume $\Prob(E_S)>0$.
On $E_S$, define the normalized restrictions
\[
 G_S(dv)=\frac{\one_{\Sigma_S}(v)}{G(\Sigma_S)}\,G(dv),
 \qquad
 \widehat G_S(dv)=
 \frac{\one_{\Sigma_S}(v)}{w_S}\,\widehat G(dv).
\]

Write $\E_S$ for expectation over $G,Y$ conditional on $E_S$.
Under $\mathbf E_S$, first sample $G$ conditional on $E_S$,
then sample its Gaussian field $Y$, and finally sample
$v^1,v^2,\ldots$ independently from $\widehat G_S$ conditional
on $G,Y$. Thus, for a bounded function $H$ involving $n$ replicas,
\[
 \mathbf E_S[H]=\E_S[\widehat G_S^{\otimes n}(H)].
\]
For these replicas, put
\[
 R_{ij}=v^i\cdot v^j,\qquad
 Y_i=Y(v^i),\qquad M_i=\tanh(h+Y_i).
\]
Let $\mathscr R^S$ be the completion under this law of
$\sigma(R_{ij}:i<j<\infty)$.

For $0\le r\le S$, let $p_r^{S,h}$ denote the law of
$(X_S^1,X_S^2)$ for two scalar paths that agree up to time
$r$ and, conditional on their common value at $r$, evolve
independently until time $S$. When $h=0$ and $-S\le r<0$,
define $p_r^{S,0}$ as the image of $p_{|r|}^{S,0}$ under
$(x_1,x_2)\mapsto(x_1,-x_2)$.

\begin{proposition}\label{prop:shell-marking}
The normalized restrictions satisfy
\begin{equation}\label{eq:shell-cavity}
 \widehat G_S(dv)=
 \frac{\cosh(h+Y(v))}{G_S\cosh(h+Y)}\,G_S(dv).
\end{equation}
The denominator is finite and at least one almost surely. Here
\[
 G_S\cosh(h+Y)=\int\cosh(h+Y(v))\,G_S(dv)
\]
denotes expectation over $v$ sampled from $G_S$, with
$G$ and $Y$ held fixed.

Under $\mathbf E_S$, conditional on $\mathscr R^S$, $M_1$
has law $\nu_{S,h}$ and $(M_1,M_2)$ has law
$\nu_{R_{12}}^{S,h}$. At nonzero field, $R_{12}\ge0$
almost surely. In particular,
\begin{equation}\label{mark:triangle}
 \Law_{\mathbf E_S}
 (h+Y_1,h+Y_2\mid R_{12},R_{13},R_{23})
 =p_{R_{12}}^{S,h}.
\end{equation}
Equations~\zeqref{mark:radiuscontact} and
\zeqref{mark:paircontact} hold for these replicas; in particular,
$\Gamma(S)=S$.

The Gaussian integration-by-parts identity of
Lemma~\zcref[noname]{cav:gaussian} also holds under
$\mathbf E_S$, with $G,\widehat G$ replaced by
$G_S,\widehat G_S$, diagonal overlaps $R_{ii}=S$, and
normalizer $G_S\cosh(h+Y)$.
\end{proposition}

\begin{proof}
On $\Sigma_S$, the radial factor in
\zeqref{eq:cavity-tilt} is the constant
$e^{[\kappa(1)-\kappa(S)]/2}$. Restriction and normalization
therefore cancel this factor, giving
\zeqref{eq:shell-cavity}. Its denominator is at least one.
Moreover,
\[
 \E_Y[G_S\cosh(h+Y)\mid G]
 =e^{\kappa(S)/2}\cosh h<\infty,
\]
so it is finite almost surely.

We next transfer the conditional magnetization laws.
Work first with replicas sampled from the original
$\widehat G$. For $0<\delta<1/2$, choose a continuous
function $\chi_\delta:[0,1]\to[0,1]$ that is zero on
$[0,\delta]$ and one on $[2\delta,1]$. Define
\[
 W_{\delta,n}
 =\chi_\delta(w_S)w_S^{-n}
       \prod_{j=1}^n\one_{\{q_j=S\}},
\]
with value zero when $w_S=0$. This weight is bounded
and measurable with respect to the original full overlap
sigma-field by Lemma~\zcref[noname]{lem:norm-observable}.

The factor $w_S^{-n}$ cancels the conditional probability
that all $n$ replicas lie in $\Sigma_S$. More precisely,
for any bounded function $H$ of these replicas and their
field values,
\[
 \E[W_{\delta,n}H\mid G,Y]
 =\chi_\delta(w_S)\widehat G_S^{\otimes n}(H),
\]
where the right side is defined to be zero outside $E_S$.
Its absolute value is at most $\|H\|_\infty$. Since
$\chi_\delta(w_S)\to\one_{E_S}$ almost surely, dominated
convergence gives
\[
 \lim_{\delta\downarrow0}\E[W_{\delta,n}H]
 =\Prob(E_S)\,\mathbf E_S[H].
\]

Let $F$ be a bounded Borel function of finitely many
overlaps, and choose $n$ to include every replica label
used by $F$ and the magnetization test. Apply
\zeqref{mark:onepoint} and \zeqref{mark:pair} with
multiplier $W_{\delta,n}F$. On the support of this weight,
all the displayed squared norms equal $S$. Passing to
the limit and dividing by $\Prob(E_S)$ yields
\[
 \begin{aligned}
 \mathbf E_S[Ff(M_1)]
 &=\mathbf E_S[F\nu_{S,h}(f)],\\
 \mathbf E_S[F\psi(M_1,M_2)]
 &=\mathbf E_S[F\nu_{R_{12}}^{S,h}(\psi)]
 \end{aligned}
\]
for bounded Borel test functions $f$ and $\psi$. The number of
additional labels used by $F$ is arbitrary. A monotone-class
argument therefore extends these identities to every bounded
$\mathscr R^S$-measurable $F$, including the completion.
The same restriction transfers
\zeqref{mark:radiuscontact}, \zeqref{mark:paircontact}, and
the nonnegativity assertion at nonzero field. A countable
intersection gives the assertions simultaneously for all
replica labels.

Since $S\ge Q$, \zeqref{eq:scalar-terminal} gives
$u(S,x)=\tanh x$. Applying $\arctanh$ coordinatewise to
the conditional pair law therefore gives
\[
 \Law_{\mathbf E_S}(h+Y_1,h+Y_2\mid\mathscr R^S)
 =p_{R_{12}}^{S,h}.
\]
At zero field, the reflected definition is preserved
because $\arctanh$ is odd. Conditioning this identity
further on $(R_{12},R_{13},R_{23})$ proves
\zeqref{mark:triangle}.

Finally, $E_S$ depends only on $G$, so conditioning on
$E_S$ preserves the conditional Gaussian law of $Y$.
The proof of Lemma~\zcref[noname]{cav:gaussian} thus applies
with reference measure $G_S$ and the density in
\zeqref{eq:shell-cavity}. The field variance on $\Sigma_S$
is $\kappa(S)$, and the normalizer is at least one.
Gaussian exponential moments consequently give the
integrability and convergence of finite Gaussian
projections required in that proof.
\end{proof}

Let $\lambda_S$ denote the law under $\mathbf E_S$ of
$R_{12}$ when $h\ne0$, and of $|R_{12}|$ when $h=0$.
Since both sampled vectors have squared norm $S$,
Cauchy--Schwarz gives $|R_{12}|\le S$. At nonzero field,
\zeqref{eq:positive-overlap} gives
$\widehat G^{\otimes2}(R_{12}<q_0)=0$ almost surely.
Consequently, on $E_S$,
\[
 0\le\widehat G_S^{\otimes2}(R_{12}<q_0)
 \le w_S^{-2}\widehat G^{\otimes2}(R_{12}<q_0)=0.
\]
Thus $\lambda_S$ is supported in $[q_0,S]$ when $h\ne0$
and in $[0,S]$ when $h=0$.

\begin{proposition}\label{prop:shell-reduction}
Fix a physical overlap law and its successor representation.
Suppose that, for every deterministic
$S\in\{Q\}\cup\mathcal Z_+$ with $\Prob(E_S)>0$, the
normalized restriction to $\Sigma_S$ satisfies
$\lambda_S=\mu$. Then every directing measure of this
physical overlap law is supported on the sphere of squared
radius $Q$, and $S_{12}$ has law $\mu$.

If the hypothesis holds for every physical subsequential
limit, the convergence statements of
Proposition~\zcref[noname]{prop:marginal-sphere} follow.
\end{proposition}

\begin{proof}
By Proposition~\zcref[noname]{prop:radial-reduction}, the
sampled squared norms belong almost surely to the countable
set $\{Q\}\cup\mathcal Z_+$. Suppose that
$\widehat G(\{v:\|v\|^2>Q\})>0$ with positive probability.
Countability then gives a deterministic $S\in\mathcal Z_+$
such that $\Prob(w_S>0)>0$. By mutual absolute continuity,
this is equivalent to $\Prob(E_S)>0$.

Choose $\epsilon>0$ such that
$2\epsilon<\sqrt{2(S-Q)}$. Separability gives a countable
cover of $\Sigma_S$ by open balls of radius $\epsilon$.
In every environment where $\widehat G_S$ is defined,
at least one such ball $B$ has positive $\widehat G_S$-mass.
For $v^1,v^2\in B\cap\Sigma_S$,
\[
 R_{12}
 =S-\tfrac12\|v^1-v^2\|^2
 >Q.
\]
Thus
\[
 \widehat G_S^{\otimes2}(R_{12}>Q)
 \ge \widehat G_S(B)^2>0.
\]
Averaging over the environment and field yields, in either
definition of $\lambda_S$,
\[
 \lambda_S((Q,1])
 \ge\mathbf E_S[\one_{\{R_{12}>Q\}}]>0.
\]
This contradicts $\lambda_S=\mu$, since
$\supp\mu\subset[0,Q]$.

It follows that $\widehat G(\Sigma_Q)=1$ almost surely.
Mutual absolute continuity gives $G(\Sigma_Q)=1$ as well.
Hence $\Prob(E_Q)=1$ and $\widehat G_Q=\widehat G$, so the
hypothesis at $S=Q$ identifies the original law of $S_{12}$
as $\mu$. By Lemma~\zcref[noname]{lem:norm-observable},
the distribution of sampled squared norms is determined
by the overlap array. The support conclusion therefore
holds for every directing measure of this physical law.

Finally, every sequence of system sizes has a further
subsequence along which the overlap arrays converge.
If the hypothesis holds for every such physical limit,
all resulting laws of $S_{12}$ equal $\mu$. Compactness
then gives full-sequence convergence, proving the stated
consequences of
Proposition~\zcref[noname]{prop:marginal-sphere}.
\end{proof}

\section{Identification of the overlap law on a sphere}\label{sec:regression}
\label{sec:zero-identification}\label{sec:all-replica}
In this section we identify the averaged overlap law for replicas
sampled from the normalized restriction of $\widehat G$ to a sphere
$\{v:\|v\|^2=S\}$ of positive mass and prove the full
Ghirlanda--Guerra identities for these replicas.
\zcref[range]{subsec:si-coordinates,mix:section} use the
covariance coordinate $r=\kappa(q)$ to compare integration
by parts for the cavity field with the scalar Parisi identities.
\zcref[range]{source:subsection,geo:section} construct functions
that distinguish overlap levels and use them to prove
ultrametricity and, at zero field, positivity of the product
of three nonzero overlaps.
\zcref[range]{field:section,subsec:si-outside} determine the
conditional probability that a new replica has overlap above
a threshold with a specified earlier replica. Initially,
another earlier replica must have strictly positive overlap
below that threshold with the specified one. Here overlaps
are taken in absolute value at zero field.
\zcref[range]{subsec:si-weighted-exchange,subsec:si-distribution-completion}
remove this restriction by reweighting an additional replica
and identify the full conditional distribution, including
its possible endpoint atoms. The reduction in
\zcref{sec:shells} then shows that the original directing
measure is supported on $\{v:\|v\|^2=Q\}$ and completes
the proof of \zcref{thm:gg}.

Fix a deterministic $S\in[Q,1]$, with $Q>0$, such that
\[
 E_S=\{G(\{v:\|v\|^2=S\})>0\}
\]
has positive probability. Retain the normalized restrictions
$G_S,\widehat G_S$ from Section~\zcref[noname]{sec:shells}.
Under $\mathbf E_S$, we condition $G$ on $E_S$, sample $Y$
conditional on $G$, and then sample replicas independently
from $\widehat G_S$. The sigma-field $\mathscr R^S$ is
generated by all their signed overlaps and completed under
this law. Put $L=\kappa(1)$.
Proposition~\zcref[noname]{prop:shell-marking} identifies
the conditional laws of one and two magnetizations given
$\mathscr R^S$. Since $S\ge Q$, we have $u(S,x)=\tanh x$.
Applying $\arctanh$ therefore identifies the conditional
laws of the fields $h+Y(v^i)=\arctanh M_i$ with those of
the scalar fields at time $S$. Gaussian integration by parts
uses the partition function
\[
 G_S\cosh(h+Y)
 =\int\cosh(h+Y(v))\,G_S(dv)
\]
appearing in \zeqref{eq:shell-cavity}.

\subsection{Covariance coordinates and conditional field laws}
\label{subsec:si-coordinates}

The covariance of the fields at $v,w$ is $\kappa(v\cdot w)$.
We represent this covariance as an inner product and replace
the scalar time parameter $q$ by $t=\kappa(q)$, under which
the scalar diffusion has unit diffusion coefficient.
Let $\mathcal H$ be the separable Hilbert space containing
the support of $G$, and define, for $\|v\|\le1$,
\begin{equation}\label{eq:gram-lift}
 \mathcal V(v)=\bigoplus_{p\in2\mathbb N}
              \sqrt p\,b_p\,v^{\otimes(p-1)}
 \in\bigoplus_{p\in2\mathbb N}\mathcal H^{\otimes(p-1)}.
\end{equation}
This is the map $\phi_\xi$ from the proof of
Lemma~\zcref[noname]{mark:ball-approximation}, where its
convergence and Lipschitz bound were established. 
Moreover,
\[
 \mathcal V(v)\cdot\mathcal V(w)=\kappa(v\cdot w),\qquad
 \|\mathcal V(v)\|^2=\kappa(S)\quad(\|v\|^2=S).
\]
Thus $\mathcal V$ is continuous, and it is odd because
each tensor power $p-1$ is odd.

For replicas sampled from $\widehat G_S$, put
\begin{equation}\label{eq:clock-change}
 K=\kappa(S),\qquad \rho=\kappa_\#\mu,\qquad
 T_{ij}=\kappa(R_{ij})\ (i\ne j),\quad T_{ii}=K.
\end{equation}
Here $\rho$ is the distribution of $\kappa(q)$ when
$q$ has distribution $\mu$. Since $\mu$ is supported
on $[0,Q]$ and $Q\le S$, $\rho$ is supported on $[0,K]$.
The array $T$ satisfies
$T_{ij}=\mathcal V(v^i)\cdot\mathcal V(v^j)$ for all
$i,j$. Its diagonal $K$ is the variance of $Y(v^i)$;
the directing vectors have squared norm $S$, while the
finite-system self-overlap is $1$.

Define
\[
 U_{ij}=
 \begin{cases}
 T_{ij},&h\ne0,\\
 |T_{ij}|,&h=0,
 \end{cases}
 \qquad U_{ii}=K,
\]
and let $\mathcal A_n$ be the completion under $\mathbf E_S$
of $\sigma(T_{ij}:1\le i<j\le n)$. Since $\kappa$ is an
odd, strictly increasing homeomorphism from $[-S,S]$ onto
$[-K,K]$, $\mathcal A_n$ is also the completed sigma-field
generated by $(R_{ij})_{1\le i<j\le n}$. In particular,
$\mathcal A_1$ is the completion of the trivial sigma-field.
Expectations conditional on $\mathcal A_n$ are taken under
the full sampling law $\mathbf E_S$.

We now make the same change of coordinate in the scalar
PDE and diffusion. Since $\alpha=1$ on $[S,1]$, the scalar
PDE gives
\[
 \Phi(S,x)=\log(2\cosh x)+\frac{L-K}{2}.
\]
Subtracting this constant gives the terminal value
$\log(2\cosh x)$ at the new terminal time $K$. Define
\begin{equation}\label{eq:si-time-changed-scalar}
 \begin{aligned}
 \phi(t,x)&=\Phi(\kappa^{-1}(t),x)-\frac{L-K}{2},\\
 a(t)&=\rho([0,t])=\alpha(\kappa^{-1}(t)),\\
 \mathsf u(t,x)&=\phi_x(t,x),\qquad
 \mathsf V(t,x)=\phi_{xx}(t,x),\\
 \mathsf X_t&=X_{\kappa^{-1}(t)},\qquad 0\le t\le K.
 \end{aligned}
\end{equation}
If $B_q$ drives the original scalar diffusion, then
\[
 \widetilde B_t
 =\int_0^{\kappa^{-1}(t)}\sqrt{\kappa'(q)}\,dB_q
\]
is a centered continuous Gaussian process with covariance
$\E[\widetilde B_t\widetilde B_s]=\min(t,s)$, and hence
is a standard Brownian motion. Substituting
$dt=\kappa'(q)\,dq$ in the integrated PDE and SDE gives
\begin{equation}\label{eq:si-clock-pde}
 \begin{aligned}
 \phi_t&=-\tfrac12(\phi_{xx}+a\phi_x^2),
 &\phi(K,x)&=\log(2\cosh x),\\
 d\mathsf X_t
 &=a(t)\mathsf u(t,\mathsf X_t)\,dt+d\widetilde B_t,
 &\mathsf X_0&=h.
 \end{aligned}
\end{equation}
The substitution is made in the time integrals and remains
valid when $\kappa'(0)=0$.

Let $\mathbb E_{\rm sc}$ denote expectation over the scalar
paths in \zeqref{eq:si-clock-pde}. For $r\in[0,K]$, let
$p_r$ be the joint law of $(\mathsf X_K^1,\mathsf X_K^2)$
when the two paths agree up to time $r$ and, conditional
on their common value at $r$, continue independently to
time $K$. At zero field, for $r\in[-K,0)$, define $p_r$
as the image of $p_{|r|}$ under
$(x_1,x_2)\mapsto(x_1,-x_2)$.
With $W_i=h+Y(v^i)$,
Proposition~\zcref[noname]{prop:shell-marking} gives,
for distinct $i,j$,
\begin{equation}\label{eq:si-field-interface}
 \Law_{\mathbf E_S}(W_i\mid\mathscr R^S)
 =\Law_{\rm sc}(\mathsf X_K),\qquad
 \Law_{\mathbf E_S}(W_i,W_j\mid\mathscr R^S)
 =p_{T_{ij}}.
\end{equation}
Thus the conditional pair law is indexed by the field
covariance $T_{ij}$, which is also the time at which the
two scalar paths separate when $T_{ij}\ge0$.

The conditional field laws in \zeqref{eq:si-field-interface}
describe the fields once the overlaps are given. To determine
the overlap distribution itself, we use Gaussian integration
by parts under the reweighted measure. Differentiating its
partition function introduces an additional replica, and
comparison with the scalar identities constrains the overlaps
of this replica with the earlier ones. The following proposition
states the resulting conditional law and geometric conclusions.
Its proof occupies the following subsections and is completed
in \zcref{subsec:si-distribution-completion}.

\begin{proposition}\label{prop:physical-shell-identification}
Assume \zeqref{eq:mixture}, fix $h\in\R$, and let
$\mu=\mu_{\xi,h}$ and $Q=\max\supp\mu>0$.
Fix a deterministic $S\in[Q,1]$ such that
\[
 \Prob(E_S)>0,\qquad
 E_S=\{G(\{v:\|v\|^2=S\})>0\}.
\]
Use the normalized restrictions $G_S,\widehat G_S$ to
$\{v:\|v\|^2=S\}$, with the cavity formula
\zeqref{eq:shell-cavity} and conditional field laws
\zeqref{eq:si-field-interface}. Retain the sampling law
$\mathbf E_S$, the arrays $T,U$, and the sigma-fields
$\mathcal A_n$ defined above, with
$K=\kappa(S)$ and $\rho=\kappa_\#\mu$.

For every $n\ge1$,
\begin{equation}\label{eq:si-shell-conditional-law}
 \Law_{\mathbf E_S}(U_{1,n+1}\mid\mathcal A_n)
 =\frac1n\rho+\frac1n\sum_{j=2}^n\delta_{U_{1j}}
 \quad\text{almost surely}.
\end{equation}
For $n=1$, the sum is empty, so $U_{12}$ has law $\rho$.
Consequently, under $\mathbf E_S$, $R_{12}$ has law $\mu$
when $h\ne0$, and $|R_{12}|$ has law $\mu$ when $h=0$.
Thus $\lambda_S=\mu$.

Moreover, almost surely, for every triple of distinct
indices $i,j,k$,
\[
 U_{jk}\ge\min(U_{ij},U_{ik}).
\]
At zero field, almost surely for all such triples,
\begin{equation}\label{eq:si-shell-signs}
 T_{ij}\ne0,\ T_{ik}\ne0
 \quad\Longrightarrow\quad
 \begin{cases}
 T_{jk}\ne0,\\
 \operatorname{sgn}T_{jk}
 =\operatorname{sgn}T_{ij}\operatorname{sgn}T_{ik}.
 \end{cases}
\end{equation}
\end{proposition}

\subsection{Scalar and cavity integration-by-parts identities}

We state the scalar and cavity integration-by-parts formulas
that will be compared using the conditional field laws
\zeqref{eq:si-field-interface}. For the rest of this section,
write $\mathbb E_{\rm ph}$ for $\mathbf E_S$. This expectation
averages over $G$ conditional on $E_S$, then over $Y$
conditional on $G$, and finally over replicas sampled
independently from $\widehat G_S$ conditional on $G,Y$.

The scalar estimates from \zcref{sec:scalar}, after the change
of time $t=\kappa(q)$, give
\begin{equation}\label{eq:si-scalar-estimates}
 \begin{gathered}
 |\mathsf u(t,x)|<1,\qquad
 0<\mathsf V(t,x)\le1-\mathsf u(t,x)^2,\\
 d\,\mathsf u(t,\mathsf X_t)
 =\mathsf V(t,\mathsf X_t)\,d\widetilde B_t,\\
 1-\mathsf u(t,x)\le C_K e^{-2x}\quad(x\in\R),\qquad
 1-\mathsf u(t,x)\ge c_K e^{-2x}\quad(x\ge0).
 \end{gathered}
\end{equation}
These bounds hold for $0\le t\le K$, with
$0<c_K\le C_K<\infty$ depending only on $K$.

Since $u(S,x)=\tanh x$, the proof of
Lemma~\zcref[noname]{scalar:tail} can be stopped at time $S$.
After the change of time $t=\kappa(q)$, it gives
\[
 |\arctanh\mathsf u(t,x)-x|\le2(K-t).
\]
Thus the last two bounds in \zeqref{eq:si-scalar-estimates}
hold with $c_K=e^{-4K}$ and $C_K=2e^{4K}$. Also,
\[
 \sup_{0\le t\le K}|\mathsf X_t|
 \le |h|+K+\sup_{0\le t\le K}|\widetilde B_t|,
\]
so this supremum has finite exponential moments of every
positive order.

Under a reference expectation $\mathbb E_0$, let $B$ be a
standard Brownian motion and put $W_t^0=h+B_t$.
Lemma~\zcref[noname]{scalar:gaussian}, under the same change
of time, shows that the law of $\mathsf X$ relative to the
law of $W^0$ has density
\begin{equation}\label{eq:si-scalar-density}
 \mathcal D(W^0)=
 \exp\left\{\phi(K,W_K^0)
       -\int_{[0,K]}\phi(t,W_t^0)\,\rho(dt)\right\},
 \qquad \mathbb E_0\mathcal D=1.
\end{equation}
The integral includes any atoms of $\rho$ at zero and at $K$.
The constant subtracted in \zeqref{eq:si-time-changed-scalar}
cancels from the exponent because $\rho([0,K])=1$.

For $f\in C^1(\R)$ satisfying
$|f(x)|+|f'(x)|\le Ce^{c|x|}$ and $r\in[0,K]$,
the same change of time in \zeqref{eq:scalar-ibp} gives
\begin{equation}\label{eq:si-scalar-ibp}
 \begin{aligned}
 \mathbb E_{\rm sc}[(\mathsf X_r-h)f(\mathsf X_K)]
 ={}&r\,\mathbb E_{\rm sc}
 [f'(\mathsf X_K)+\tanh\mathsf X_K f(\mathsf X_K)]\\
 &-\int_{[0,K]}\min(r,t)\,
 \mathbb E_{\rm sc}
 [\mathsf u(t,\mathsf X_t)f(\mathsf X_K)]\,\rho(dt).
 \end{aligned}
\end{equation}

For the cavity formula, put
\[
 Y_i=Y(v^i),\qquad W_i=h+Y_i,\qquad M_i=\tanh W_i,
 \qquad
 Z_S=\int\cosh(h+Y(v))\,G_S(dv)\ge1.
\]
We first check integrability under the reweighted sampling
law. Relative to $G_S^{\otimes m}$, the density of the
sampled vectors is $Z_S^{-m}\prod_{i=1}^m\cosh W_i$.
Using $Z_S^{-m}\le1$ and
$\cosh W_i\le e^{|h|+|Y_i|}$ reduces the required bound
to Gaussian exponential moments at fixed reference vectors.
Their field variances are at most $K$, so
\begin{equation}\label{eq:si-field-moments}
 \mathbb E_{\rm ph}
 \exp\left(c\sum_{i=1}^m|Y_i|\right)
 \le 2^m\exp\left\{m|h|+\frac K2m^2(c+1)^2\right\}
 <\infty
 \qquad(m\ge1,\ c>0).
\end{equation}

Fix $n\ge1$ and $1\le i\le n$. Let $F$ be bounded and
$\mathcal A_n$-measurable, and suppose $g\in C^1(\R^n)$
satisfies
\[
 |g(w)|+\sum_{j=1}^n|\partial_jg(w)|
 \le C\exp\left(c\sum_{j=1}^n|w_j|\right).
\]
With $\mathbf W=(W_1,\ldots,W_n)$, Gaussian integration
by parts gives 
\begin{equation}\label{eq:unit-ibp}
 \begin{aligned}
 \mathbb E_{\rm ph}[F Y_i g(\mathbf W)]
 =\mathbb E_{\rm ph}\Bigl[F\Bigl(
 &\sum_{j=1}^nT_{ij}\partial_jg(\mathbf W)
 +g(\mathbf W)\sum_{\ell=1}^nT_{i\ell}M_\ell-n g(\mathbf W)T_{i,n+1}M_{n+1}
 \Bigr)\Bigr].
 \end{aligned}
\end{equation}
Here $v^{n+1}$ is an additional independent sample from
$\widehat G_S$, conditional on $G,Y$.

To derive the formula, condition on $G$ and integrate
the reference vectors against $G_S^{\otimes n}$. Hold
these vectors and $F$ fixed during Gaussian differentiation.
Their sampling density is
\[
 Z_S^{-n}\prod_{\ell=1}^n\cosh W_\ell.
\]
Differentiating $g$ gives the first sum in
\zeqref{eq:unit-ibp}. Differentiating the numerator factors
gives their logarithmic derivatives $M_\ell$ and hence
the second sum. For a bounded deterministic function $b$,
the derivative of the partition function is
\[
 \left.\frac{d}{ds}\right|_{s=0}\log Z_S(Y+sb)
 =\int b(v)M(v)\,\widehat G_S(dv).
\]
Gaussian integration by parts pairs this derivative with
the covariance $\kappa(v^i\cdot v)$. Differentiating
$Z_S^{-n}$ therefore contributes
\[
 -n g(\mathbf W)
   \int\kappa(v^i\cdot v)M(v)\,\widehat G_S(dv).
\]
Representing this integral by the additional replica
$v^{n+1}$ gives the last term of \zeqref{eq:unit-ibp}.

These calculations can first be made for finite orthogonal
projections of the Gaussian field. Their covariance kernels
converge to $\kappa(v\cdot w)$, their variances are bounded
by $K$, and their partition functions are at least one.
Conditional on $G$, the projected fields and their
$\cosh$ weights converge in every finite $L^p$ under
Gaussian expectation and any fixed number of independent
$G_S$-samples. Jensen's inequality gives the same convergence
for the partition functions. For $x,y\ge1$,
\[
 |x^{-n}-y^{-n}|\le n|x-y|,
\]
so their inverse powers also converge. The exponential
bounds used in \zeqref{eq:si-field-moments} hold uniformly
over the projections and justify passage of every term
to \zeqref{eq:unit-ibp}.

\subsection{Overlap identities from conditional field laws}
\label{mix:section}

To eliminate the fields and magnetizations from
\zeqref{eq:unit-ibp}, we compute the corresponding expectations
under the scalar pair law $p_r$. The resulting quantities
are deterministic functions of the covariance $r$.

Fix $f\in C^1(\R)$ satisfying
\begin{equation}\label{eq:si-terminal-tests}
 |f(x)|+|f'(x)|\le Ce^{c|x|}\qquad(x\in\R)
\end{equation}
for some $C,c>0$, and define
\begin{equation}\label{mix:profiles}
 C_f(r)=\mathbb E_{\rm sc}
 [\mathsf u(r,\mathsf X_r)f(\mathsf X_K)],
 \qquad
 A(r)=\int_r^K a(t)\,dt,\qquad 0\le r\le K.
\end{equation}
Consider two scalar paths that agree up to time $r$ and
continue independently conditional on their common value
$\mathsf X_r$. The martingale property gives
\[
 \mathbb E_{\rm sc}
 [\tanh\mathsf X_K^1\mid\mathsf X_r]
 =\mathsf u(r,\mathsf X_r).
\]
Conditional independence of the continuations therefore yields
\[
 \mathbb E_{\rm sc}
 [\tanh\mathsf X_K^1 f(\mathsf X_K^2)]
 =C_f(r).
\]
Thus $C_f(r)$ is the scalar expectation corresponding to
a magnetization at one replica multiplied by $f$ of the
field at another.

We also need the expectation with the first magnetization
replaced by its field. Integrating the scalar SDE from
$r$ to $K$ and using the same martingale property gives
\[
 \begin{aligned}
 \mathbb E_{\rm sc}^{r,x}[\mathsf X_K-h]
 &=x-h+\int_r^K a(t)\,
   \mathbb E_{\rm sc}^{r,x}
   [\mathsf u(t,\mathsf X_t)]\,dt\\
 &=x-h+A(r)\mathsf u(r,x).
 \end{aligned}
\]
Conditioning the two continuations on $\mathsf X_r$ now gives
\[
 \mathbb E_{\rm sc}
 [(\mathsf X_K^1-h)f(\mathsf X_K^2)]
 =\mathbb E_{\rm sc}
 [(\mathsf X_r-h)f(\mathsf X_K)]
  +A(r)C_f(r).
\]
To express the first term by scalar integration by parts,
define, for bounded Borel $C:[0,K]\to\R$,
\begin{equation}\label{mix:correction}
 (\mathcal KC)(r)
 =A(r)C(r)+rC(K)
  -\int_{[0,K]}\min(r,t)C(t)\,\rho(dt).
\end{equation}
Since
$C_f(K)=\mathbb E_{\rm sc}
[\tanh\mathsf X_K f(\mathsf X_K)]$,
\zeqref{eq:si-scalar-ibp} yields
\begin{equation}\label{mix:scalarK}
 \mathbb E_{\rm sc}
 [(\mathsf X_K^1-h)f(\mathsf X_K^2)]
 -r\mathbb E_{\rm sc}f'(\mathsf X_K)
 =(\mathcal KC_f)(r).
\end{equation}
The subtracted term is the derivative contribution that
also appears in the cavity integration-by-parts formula.
The calculation includes $r=0$, where the paths are
independent and the pair expectation is $A(0)C_f(0)$,
and $r=K$, where the paths coincide.

We finally express these formulas for signed covariances.
At nonzero field, define
\[
 H_f(r)=C_f(r),\qquad
 K_f(r)=(\mathcal KC_f)(r),\qquad 0\le r\le K.
\]
At zero field, restrict to odd $f$ and define, for
$-K\le r\le K$,
\[
 H_f(r)=\operatorname{sgn}(r)C_f(|r|),\qquad
 K_f(r)=\operatorname{sgn}(r)(\mathcal KC_f)(|r|),
 \qquad \operatorname{sgn}(0)=0.
\]
For negative $r$, reflecting the second endpoint changes
the sign of $f(\mathsf X_K^2)$, which gives these extensions.
At $r=0$, $\mathsf u(0,0)=0$ implies
$C_f(0)=(\mathcal KC_f)(0)=0$. Consequently, throughout
the permitted covariance range,
\[
 \begin{aligned}
 H_f(r)
 &=\int\tanh x_1\,f(x_2)\,p_r(dx_1,dx_2),\\
 K_f(r)
 &=\int(x_1-h)f(x_2)\,p_r(dx_1,dx_2)
   -r\mathbb E_{\rm sc}f'(\mathsf X_K).
 \end{aligned}
\]

\begin{lemma}\label{mix:identity}
Fix $n\ge1$, $1\le i,j\le n$, and $f\in C^1(\R)$
satisfying \zeqref{eq:si-terminal-tests}, with $f$ odd
when $h=0$. Under the sampling law defining
$\mathbb E_{\rm ph}$, almost surely,
\begin{equation}\label{mix:main}
 n\mathbb E_{\rm ph}
 [T_{i,n+1}H_f(T_{j,n+1})\mid\mathcal A_n]
 =
 \sum_{\ell=1}^nT_{i\ell}H_f(T_{j\ell})-K_f(T_{ij}).
\end{equation}
Taking $i=j$ gives
\begin{equation}\label{mix:GG}
 n\mathbb E_{\rm ph}
 [U_{i,n+1}C_f(U_{i,n+1})\mid\mathcal A_n]
 =
 \int_{[0,K]}tC_f(t)\,\rho(dt)
 +\sum_{\substack{1\le\ell\le n\\\ell\ne i}}
 U_{i\ell}C_f(U_{i\ell})
 \quad\text{almost surely}.
\end{equation}
\end{lemma}

\begin{proof}
The conditional field laws \zeqref{eq:si-field-interface}
and the scalar formulas for $H_f,K_f$ give, for
$1\le a,b\le m$,
\[
 \begin{aligned}
 \mathbb E_{\rm ph}[f(W_a)M_b\mid\mathcal A_m]
 &=H_f(T_{ab}),\\
 \mathbb E_{\rm ph}
 [Y_a f(W_b)-T_{ab}f'(W_b)\mid\mathcal A_m]
 &=K_f(T_{ab}).
 \end{aligned}
\]
Indeed, first condition on $\mathscr R^S$ and apply
the scalar pair law, then condition on $\mathcal A_m$.
When $a=b$, the one-field law gives the same formulas
with $T_{aa}=K$. The definitions of $H_f,K_f$ include
negative covariances at zero field. The scalar exponential
moments and \zeqref{eq:si-field-moments} ensure integrability.

Apply \zeqref{eq:unit-ibp} with $g(\mathbf W)=f(W_j)$.
Its derivative contribution is $T_{ij}f'(W_j)$.
Since the identity holds for every bounded
$\mathcal A_n$-measurable multiplier $F$, the preceding
conditional expectations with $m=n$ give
\[
 n\mathbb E_{\rm ph}
 [T_{i,n+1}f(W_j)M_{n+1}\mid\mathcal A_n]
 =
 \sum_{\ell=1}^nT_{i\ell}H_f(T_{j\ell})-K_f(T_{ij}).
\]
To evaluate the remaining magnetization term, we must
retain the overlaps involving replica $n+1$. Since
$T_{i,n+1}$ is $\mathcal A_{n+1}$-measurable, the first
conditional formula with $m=n+1$ gives
\[
 \begin{aligned}
 &\mathbb E_{\rm ph}
 [T_{i,n+1}f(W_j)M_{n+1}\mid\mathcal A_n]\\
 &\qquad=
 \mathbb E_{\rm ph}\!\left[
 T_{i,n+1}\,
 \mathbb E_{\rm ph}[f(W_j)M_{n+1}\mid\mathcal A_{n+1}]
 \,\middle|\,\mathcal A_n\right]\\
 &\qquad=
 \mathbb E_{\rm ph}
 [T_{i,n+1}H_f(T_{j,n+1})\mid\mathcal A_n].
 \end{aligned}
\]
This proves \zeqref{mix:main}.

For $i=j$, use
\[
 H_f(K)=C_f(K),\qquad
 K_f(K)=KC_f(K)-\int_{[0,K]}tC_f(t)\,\rho(dt).
\]
Thus the $\ell=i$ term cancels the term $KC_f(K)$
in $K_f(K)$. At nonzero field, $T_{i\ell}=U_{i\ell}\ge0$.
At zero field, the definition of $H_f$ gives
\[
 T_{i\ell}H_f(T_{i\ell})
 =|T_{i\ell}|C_f(|T_{i\ell}|)
 =U_{i\ell}C_f(U_{i\ell}).
\]
These substitutions yield \zeqref{mix:GG}, also when
$n=1$, for which the sum is empty.
\end{proof}

\subsection{Exponential test functions}
\label{source:subsection}

We choose exponential functions of the terminal field in
Lemma~\zcref[noname]{mix:identity} to distinguish different
covariance values. For $y>0$, put $\theta_y=y/K$ and define
\[
 f_y(x)=
 \begin{cases}
 e^{\theta_yx},&h\ne0,\\
 \sinh(\theta_yx),&h=0.
 \end{cases}
\]
The choice at zero field satisfies the oddness requirement
in that lemma. Set
\[
 Z_y=C_{f_y}(K)
 =\mathbb E_{\rm sc}
 [\tanh\mathsf X_K f_y(\mathsf X_K)].
\]
Whenever $Z_y>0$, normalize $C_{f_y}$ by its value at $K$,
giving
\begin{equation}\label{source:eq:profiles}
 C_y(r)=\frac{C_{f_y}(r)}{Z_y}
 =\frac{\mathbb E_{\rm sc}
 [\mathsf u(r,\mathsf X_r)f_y(\mathsf X_K)]}{Z_y},
 \qquad
 \delta_y(r)=1-C_y(r),\qquad 0\le r\le K.
\end{equation}
Thus $C_y(K)=1$. The following lemma shows that, at each
fixed $r\in(0,K)$, the difference $\delta_y(r)$ decays
exponentially in $y$, with a faster rate at larger $r$.
These different rates will distinguish overlap levels
in the identities of Lemma~\zcref[noname]{mix:identity}.

\begin{lemma}\label{source:prop:ordered}
At zero field, $Z_y>0$ for every $y>0$. At nonzero field,
$Z_y>0$ for all sufficiently large $y$. For these values
of $y$, $C_y$ is continuous and strictly increasing on
$[0,K]$, with $C_y(K)=1$. In particular,
$\delta_y(r)>0$ for $0\le r<K$ and $\delta_y(K)=0$.

There exists $y_0>0$, depending only on $\xi,h,S$, such that
\[
 \sup_{y\ge y_0}\sup_{0\le r\le K}|C_y(r)|<\infty.
\]
As $y\to\infty$, $C_y$ converges pointwise on $[0,K]$ to
\begin{equation}\label{eq:si-profile-limit}
 C_\infty(r)=
 \begin{cases}
 1,&0<r\le K,\\
 c,&r=0,
 \end{cases}
 \qquad c=\mathsf u(0,h)\in(-1,1).
\end{equation}
At zero field, $c=0$.

For each fixed $r\in(0,K)$, there are constants
$c_r,C_r>0$ and $Y_r\ge y_0$, depending only on
$\xi,h,S,r$, such that
\begin{equation}\label{source:eq:fixed-tail}
 c_re^{-2ry/K}\le\delta_y(r)\le C_re^{-2ry/K}
 \qquad(y\ge Y_r).
\end{equation}
Consequently, for every fixed $0\le t<s\le K$,
\begin{equation}\label{source:eq:separation}
 \lim_{y\to\infty}\frac{\delta_y(s)}{\delta_y(t)}=0.
\end{equation}
\end{lemma}

\begin{proof}
For $\theta>0$, define the tilted scalar expectation by
\[
 \mathbb E_\theta F
 =\frac{\mathbb E_{\rm sc}[F e^{\theta\mathsf X_K}]}
        {\mathbb E_{\rm sc}e^{\theta\mathsf X_K}}.
\]
At nonzero field, the definition of $C_y$ gives
\begin{equation}\label{source:eq:tilted-ratio}
 C_y(r)=
 \frac{\mathbb E_{\theta_y}\mathsf u(r,\mathsf X_r)}
      {\mathbb E_{\theta_y}\tanh\mathsf X_K}.
\end{equation}
This formula also holds at zero field. Indeed, reflection of
the entire scalar path changes the signs of both
$\mathsf u(r,\mathsf X_r)$ and $\tanh\mathsf X_K$.
Replacing the exponential by the hyperbolic sine therefore
leaves both unnormalized expectations unchanged.

We first estimate the fields under this tilted law.
Cameron--Martin translation in
\zeqref{eq:si-scalar-density} gives, for every bounded Borel
function $F$ of the path,
\[
 \mathbb E_\theta F(\mathsf X)
 =\frac{\mathbb E_0\!\left[
 F\bigl((h+\theta t+B_t)_{0\le t\le K}\bigr)e^{R_\theta(B)}
 \right]}
 {\mathbb E_0e^{R_\theta(B)}},
\]
where
\[
 \begin{aligned}
 R_\theta(B)={}&
 \phi(K,h+\theta K+B_K)-\phi(K,h+\theta K)\\
 &-\int_{[0,K]}
 \bigl[\phi(t,h+\theta t+B_t)-\phi(t,h+\theta t)\bigr]\,\rho(dt).
 \end{aligned}
\]
The terms depending only on the deterministic path
$t\mapsto h+\theta t$ cancel between numerator and
denominator. Since $|\phi_x|\le1$, the Brownian supremum
$\mathcal B=\sup_{0\le t\le K}|B_t|$ satisfies
\begin{equation}\label{source:eq:tilt-comparison}
 |R_\theta(B)|
 \le |B_K|+\int_{[0,K]}|B_t|\,\rho(dt)
 \le2\mathcal B.
\end{equation}
Consequently,
\[
 0<\mathbb E_0e^{-2\mathcal B}
 \le\mathbb E_0e^{R_\theta(B)}
 \le\mathbb E_0e^{2\mathcal B}<\infty,
\]
uniformly in $\theta$.

Put
\[
 d_\theta(r)=1-\mathbb E_\theta\mathsf u(r,\mathsf X_r).
\]
The upper bound in \zeqref{eq:si-scalar-estimates} gives
\[
 d_\theta(r)
 \le C_K e^{-2h-2\theta r}
 \frac{\mathbb E_0e^{4\mathcal B}}
      {\mathbb E_0e^{-2\mathcal B}}
 \le C_{K,h}e^{-2\theta r},
 \qquad 0\le r\le K.
\]
For a lower bound, fix $r>0$ and choose $R<\infty$ such
that $\mathbb P_0(\mathcal B\le R)>0$. For sufficiently
large $\theta$, $h+\theta r+B_r\ge0$ on this event.
The lower bound in \zeqref{eq:si-scalar-estimates} and
\zeqref{source:eq:tilt-comparison} then give
\[
 d_\theta(r)
 \ge c_K e^{-2h-2\theta r-4R}
 \frac{\mathbb P_0(\mathcal B\le R)}
      {\mathbb E_0e^{2\mathcal B}}
 \ge c_{K,h,R}e^{-2\theta r}.
\]
In particular, $d_\theta(K)\to0$. Since
$\mathsf u(K,x)=\tanh x$, the denominator in
\zeqref{source:eq:tilted-ratio} equals $1-d_{\theta_y}(K)$
and is at least $1/2$ for all sufficiently large $y$.
This proves eventual positivity of $Z_y$. At zero field,
$Z_y>0$ for every $y>0$, because
$\tanh x\,\sinh(\theta_yx)>0$ for $x\ne0$ and the law of
$\mathsf X_K$ is nondegenerate.

We next prove strict monotonicity in $r$. Define
\[
 h_y(t,x)=\mathbb E_{\rm sc}^{t,x}f_y(\mathsf X_K).
\]
The scalar drift is spatially Lipschitz. Solutions driven
by the same Brownian path therefore remain strictly
ordered when their initial values are strictly ordered.
Since $f_y$ is strictly increasing, $h_y(t,\cdot)$ is
strictly increasing. The function $\mathsf u(t,\cdot)$
is also strictly increasing because $\mathsf V>0$.

Let $\mathcal F_r$ be the scalar Brownian filtration.
Both $\mathsf u(t,\mathsf X_t)$ and
$h_y(t,\mathsf X_t)$ are square-integrable martingales, and
\[
 Z_yC_y(r)
 =\mathbb E_{\rm sc}
 [\mathsf u(r,\mathsf X_r)h_y(r,\mathsf X_r)].
\]
Thus, for $0\le r<s\le K$,
\[
 Z_y[C_y(s)-C_y(r)]
 =
 \mathbb E_{\rm sc}\!\left[
 \operatorname{Cov}\bigl(
 \mathsf u(s,\mathsf X_s),h_y(s,\mathsf X_s)
 \mid\mathcal F_r\bigr)\right]>0.
\]
To see the strict inequality, boundedness of the drift
and Girsanov's theorem imply that, conditional on
$\mathcal F_r$, the law of $\mathsf X_s$ is equivalent to
$N(\mathsf X_r,s-r)$. The covariance of two strictly
increasing functions under this law is positive. Indeed,
it equals half the expected product of their differences
at two independent samples, and that product is strictly
positive almost surely. Continuity of $C_y$ follows from
continuity of $\mathsf u(r,\mathsf X_r)$, its boundedness,
and integrability of $f_y(\mathsf X_K)$.
Together with $C_y(K)=1$, this proves the assertions about
monotonicity and positivity of $\delta_y$.

Finally, \zeqref{source:eq:tilted-ratio} gives
\[
 \delta_y(r)=
 \frac{d_{\theta_y}(r)-d_{\theta_y}(K)}
      {1-d_{\theta_y}(K)}.
\]
For each fixed $r\in(0,K)$, the preceding bounds imply
\[
 d_{\theta_y}(K)
 =O(e^{-2\theta_yK})
 =o(e^{-2\theta_yr}).
\]
Subtracting this term from the two-sided bounds for
$d_{\theta_y}(r)$ proves
\zeqref{source:eq:fixed-tail}.

For all sufficiently large $y$, the denominator in
\zeqref{source:eq:tilted-ratio} is at least $1/2$, so
$|C_y(r)|\le2$ uniformly in $r$. At every fixed $r>0$,
both numerator and denominator tend to one. At $r=0$,
the numerator is the deterministic value
$\mathsf u(0,h)$, which equals zero when $h=0$.
This proves \zeqref{eq:si-profile-limit}, with
$c=\mathsf u(0,h)\in(-1,1)$.

For $0<t<s<K$, division of the bounds in
\zeqref{source:eq:fixed-tail} gives
\[
 0\le\frac{\delta_y(s)}{\delta_y(t)}
 \le\frac{C_s}{c_t}e^{-2(s-t)y/K}\longrightarrow0.
\]
The cases $t=0$ and $s=K$ follow from
$\delta_y(0)\to1-c>0$ and $\delta_y(K)=0$.
This proves \zeqref{source:eq:separation}.
\end{proof}

Taking $f=f_y/Z_y$ in
Lemma~\zcref[noname]{mix:identity} gives $C_f=C_y$.
The preceding lemma implies
\[
 rC_y(r)\longrightarrow r,\qquad 0\le r\le K,
\]
with a bound uniform in $r$ and all sufficiently large $y$.
At $r=0$, both sides vanish regardless of the value of
$C_y(0)$. Put
\[
 m_1=\int_{[0,K]}t\,\rho(dt)=K-A(0).
\]
Multiply \zeqref{mix:GG} by any bounded
$\mathcal A_n$-measurable function and take expectations.
Dominated convergence then gives
\begin{equation}\label{eq:si-first-moment}
 n\mathbb E_{\rm ph}[U_{i,n+1}\mid\mathcal A_n]
 =m_1+\sum_{\substack{\ell\le n\\\ell\ne i}}U_{i\ell},
 \qquad n\ge1,\quad 1\le i\le n,
\end{equation}
because the multiplier is arbitrary.

For each fixed $r\in(0,K]$, dominated convergence in
\zeqref{mix:correction} also gives
\begin{equation}\label{eq:si-limit-correction}
 \lim_{y\to\infty}(\mathcal KC_y)(r)
 =(\mathcal KC_\infty)(r)=A(0).
\end{equation}
Indeed, $C_\infty(r)=C_\infty(K)=1$, and the factor
$\min(r,t)$ vanishes at $t=0$. Hence
\[
 \begin{aligned}
 (\mathcal KC_\infty)(r)
 &=A(r)+r-\int_{[0,K]}\min(r,t)\,\rho(dt)\\
 &=A(r)+r-\int_0^r(1-a(t))\,dt
 =A(0).
 \end{aligned}
\]

\subsection{Signs and ultrametricity}
\label{geo:section}

\begin{lemma}\label{geo:sign}
Under the sampling law defining $\mathbb E_{\rm ph}$,
almost surely for every triple of distinct labels $i,j,k$,
\[
 T_{ij}\ne0,\quad T_{ik}\ne0
 \quad\Longrightarrow\quad T_{jk}\ne0.
\]
At zero field, whenever these overlaps are nonzero,
\[
 \operatorname{sgn}T_{jk}
 =\operatorname{sgn}T_{ij}\operatorname{sgn}T_{ik}.
\]

At zero field, the conditional laws of the unsigned array
$U$ also satisfy
\[
 \Law_{\rm ph}(U\mid T_{12}=r)
 =\Law_{\rm ph}(U\mid |T_{12}|=r)
\]
for $\Law_{\rm ph}(|T_{12}|)$-almost every $r\in(0,K]$.
\end{lemma}

\begin{proof}
At nonzero field, \zeqref{eq:positive-overlap} implies
that every off-diagonal $T_{ij}$ is strictly positive.
It therefore remains to consider zero field.

On the event $\{T_{12}\ne0\}$, put
$r=|T_{12}|$ and $s=\operatorname{sgn}T_{12}$.
Equation~\zeqref{eq:si-first-moment} with $n=2$ gives
\[
 2\mathbb E_{\rm ph}[|T_{13}|\mid T_{12}]=m_1+r.
\]
Next apply \zeqref{mix:main} with $n=2$, $i=1$, $j=2$,
and $f=f_y/Z_y$, and multiply by $s$. This gives
\[
 \begin{aligned}
 &2\mathbb E_{\rm ph}
 [sT_{13}\operatorname{sgn}(T_{23})C_y(|T_{23}|)
 \mid T_{12}]\\
 &\hspace{2cm}
 =KC_y(r)+r-(\mathcal KC_y)(r).
 \end{aligned}
\]
Letting $y\to\infty$ and using
\zeqref{eq:si-limit-correction}, we obtain
\[
 2\mathbb E_{\rm ph}
 [sT_{13}\operatorname{sgn}(T_{23})\mid T_{12}]
 =K+r-A(0)=m_1+r.
\]
The passage to the limit follows by dominated convergence
after multiplication by bounded functions of $T_{12}$
supported on $\{T_{12}\ne0\}$.

Subtracting the two limiting identities shows that
\[
 \mathbb E_{\rm ph}
 \bigl[|T_{13}|-sT_{13}\operatorname{sgn}(T_{23})
 \mid T_{12}\bigr]=0
 \quad\text{on }\{T_{12}\ne0\}.
\]
The integrand is nonnegative, so it vanishes almost surely
on this event. If $T_{13}\ne0$, it follows that
$T_{23}\ne0$ and
\[
 \operatorname{sgn}T_{23}
 =s\,\operatorname{sgn}T_{13}.
\]
Permutation of the labels and a countable intersection
give the assertions for every triple.

For the conditional-law assertion, recall that independent
replica sign flips preserve the original zero-field
overlap law. This invariance passes to the normalized
restriction to $\{v:\|v\|^2=S\}$. Indeed, put
\[
 w_S=\widehat G(\{v:\|v\|^2=S\}).
\]
On $\{w_S>0\}$, restricting and normalizing the sampling
law of $m$ replicas multiplies its density by
\[
 w_S^{-m}\prod_{\ell=1}^m
 \one_{\{\|v^\ell\|^2=S\}}.
\]
The squared norms are measurable from overlaps and
unchanged by replica sign flips. The same holds for
$w_S$, which is the limiting empirical fraction of
replicas with squared norm $S$. Thus this density and
the conditioning event $E_S=\{w_S>0\}$ are invariant.
The law defining $\mathbb E_{\rm ph}$ consequently
retains replica sign-flip invariance.

A flip of replica $2$ leaves the entire unsigned array
$U$ unchanged and reverses $T_{12}$. Hence, for every
bounded Borel function $F$ of $U$ and Borel set
$B\subset(0,K]$,
\[
 \begin{aligned}
 \mathbb E_{\rm ph}[F(U)\one_{\{T_{12}\in B\}}]
 &=\mathbb E_{\rm ph}[F(U)\one_{\{-T_{12}\in B\}}]\\
 &=\tfrac12\mathbb E_{\rm ph}
 [F(U)\one_{\{|T_{12}|\in B\}}].
 \end{aligned}
\]
Taking $F=1$ gives the corresponding relation between
the conditioning marginals. Disintegration then proves
the stated equality of conditional laws.
\end{proof}

\begin{lemma}\label{geo:um}
Under the sampling law defining $\mathbb E_{\rm ph}$,
almost surely
\[
 U_{ij}\ge\min(U_{ik},U_{jk})
\]
for every triple of distinct labels $i,j,k$.
\end{lemma}

\begin{proof}
It suffices to exclude triples whose three overlaps have
a unique minimum. Put
\[
 r=U_{12},\qquad X=U_{13},\qquad Y=U_{23}.
\]
We condition throughout on the signed overlap $T_{12}$.
On $\{r>0\}$, Lemma~\zcref[noname]{geo:sign} converts
the signed cross terms in \zeqref{mix:main} into
$XC(Y)$ and $YC(X)$, after multiplication by
$\operatorname{sgn}(T_{12})$ at zero field.
If either $X$ or $Y$ vanishes, that lemma implies that
both vanish, so these formulas still apply.

Fix $C=C_y$. The $n=2$ diagonal identities
\zeqref{mix:GG} for $i=1,2$ give
\[
 \mathbb E_{\rm ph}[XC(X)+YC(Y)\mid T_{12}]
 =\int_{[0,K]}tC(t)\,\rho(dt)+rC(r).
\]
The cross identities \zeqref{mix:main} for
$(i,j)=(1,2),(2,1)$ give
\[
 \mathbb E_{\rm ph}[XC(Y)+YC(X)\mid T_{12}]
 =(K-A(r))C(r)
   +\int_{[0,K]}\min(r,t)C(t)\,\rho(dt).
\]
Subtracting and using
\[
 K-r-A(r)=\int_{(r,K]}(t-r)\,\rho(dt)
\]
yields
\begin{equation}\label{geo:dissipation}
 \mathbb E_{\rm ph}
 [(X-Y)(C(X)-C(Y))\mid T_{12}]
 =\int_{(r,K]}(t-r)(C(t)-C(r))\,\rho(dt).
\end{equation}
Both integrands are nonnegative because $C$ is increasing.
Moreover, $C(t)-C(r)\le1-C(r)=\delta_y(r)$ for $t>r$.
Consequently,
\[
 0\le
 \mathbb E_{\rm ph}
 [(X-Y)(\delta_y(Y)-\delta_y(X))\mid T_{12}]
 \le K\delta_y(r).
\]

Take $y$ through sufficiently large positive integers,
so that $Z_y>0$ and all these identities hold outside
one null set. On $\{0<r<K\}$, divide by $\delta_y(r)>0$.
The resulting nonnegative random variable has conditional
expectation at most $K$. Suppose that $X\ne Y$ and
$\min(X,Y)<r$, and write
\[
 z=\min(X,Y),\qquad w=\max(X,Y).
\]
The divided integrand then equals
\[
 (w-z)\frac{\delta_y(z)-\delta_y(w)}{\delta_y(r)}
 =(w-z)\frac{\delta_y(z)}{\delta_y(r)}
       \left(1-\frac{\delta_y(w)}{\delta_y(z)}\right)
 \longrightarrow\infty.
\]
Indeed, $z<r$ and $z<w$, so
\zeqref{source:eq:separation} applies to both ratios,
including when $z=0$. Conditional Fatou's lemma and
the bound $K$ therefore imply
\[
 \mathbb P_{\rm ph}
 (X\ne Y,\ \min(X,Y)<r\mid T_{12})=0
 \quad\text{on }\{0<r<K\}.
\]

On $\{r=K\}$, the right side of
\zeqref{geo:dissipation} is zero. Strict monotonicity of
$C_y$ then forces $X=Y$. On $\{r=0\}$,
Lemma~\zcref[noname]{geo:sign}, applied with common
label $3$, prevents $X$ and $Y$ from both being positive.

If three overlaps had a unique minimum, we could relabel
the replicas so that $r$ is one of the larger overlaps.
The remaining overlaps would then be unequal and have
minimum below $r$, contradicting the preceding conclusions.
Replica exchangeability and a countable intersection
give the assertion for every triple.
\end{proof}

\subsection{Cancellation of the terms involving two old replicas}
\label{field:section}

Fix $n\ge2$. Our next goal is to determine
$\mathbb P_{\rm ph}(U_{1,n+1}>\tau\mid\mathcal A_n)$
on the event $0<U_{12}<\tau$, for $\tau\in(0,K)$.
We first prepare the scalar integration-by-parts identity
that will cancel the terms involving only replicas $1$
and $2$ in the corresponding cavity identity.

Throughout this subsection, work on $\{0<|T_{12}|<K\}$
and put $r=|T_{12}|$. At zero field, if $T_{12}<0$,
we change the signs of the variables associated with
replica $2$. More precisely, set
$e_2=\operatorname{sgn}(T_{12})$ and $e_j=1$ for
$j\ne2$, and replace
\[
 (Y_j,W_j,M_j,T_{jk})
 \quad\text{by}\quad
 (e_jY_j,e_jW_j,e_jM_j,e_je_kT_{jk}).
\]
Apply this change separately on the two events
$\{T_{12}>0\}$ and $\{T_{12}<0\}$. During Gaussian
differentiation, the reference vectors and hence the
signs $e_j$ are fixed. Since $h=0$, the transformed
variables still satisfy $W_j=Y_j$ and $M_j=\tanh W_j$.
Their overlap between replicas $1$ and $2$ is $r>0$,
and their conditional pair field law is $p_r$.

Below, we use the same symbols for the transformed
variables. All expectations remain under the original
sampling law, and every $\mathcal A_m$ remains the
original signed-overlap sigma-field, including the
original sign of $T_{12}$. At nonzero field no sign
change is needed.

For $z>0$ and $\theta>1$, define
\begin{equation}\label{eq:si-pair-normalizer}
 Z_z(r,\theta)
 =\mathbb E_{\rm sc}
 e^{z(\mathsf X_K^1-\theta\mathsf X_K^2)},
\end{equation}
where the two scalar paths coincide through $r$ and
have independent continuations conditional on their
common value at $r$. Scalar exponential moment bounds
give $0<Z_z(r,\theta)<\infty$.
The conditional field law
\zeqref{eq:si-field-interface} gives, for every finite
$m\ge n$,
\[
 \mathbb E_{\rm ph}
 [e^{z(W_1-\theta W_2)}\mid\mathcal A_m]
 =Z_z(r,\theta).
\]
Whenever the numerators are absolutely integrable, put
\begin{equation}\label{eq:si-pair-tilts}
 \begin{aligned}
 \mathbb E_z[H\mid\mathcal A_m]
 &=\frac{\mathbb E_{\rm ph}
 [H e^{z(W_1-\theta W_2)}\mid\mathcal A_m]}
 {Z_z(r,\theta)},\\
 \mathbb E_{{\rm sc},z}J
 &=\frac{\mathbb E_{\rm sc}
 [J e^{z(\mathsf X_K^1-\theta\mathsf X_K^2)}]}
 {Z_z(r,\theta)}.
 \end{aligned}
\end{equation}

If $H$ is bounded and $\mathcal A_m$-measurable for
some finite $m\ge n$, conditioning first on
$\mathcal A_m$ gives
\[
 \begin{aligned}
 &\mathbb E_{\rm ph}
 [H e^{z(W_1-\theta W_2)}\mid\mathcal A_n]\\
 &\qquad
 =\mathbb E_{\rm ph}[H Z_z(r,\theta)\mid\mathcal A_n]
 =Z_z(r,\theta)\mathbb E_{\rm ph}[H\mid\mathcal A_n].
 \end{aligned}
\]
Here $r$ is $\mathcal A_n$-measurable. Consequently,
\begin{equation}\label{eq:si-tilt-preserves-array}
 \mathbb E_z[H\mid\mathcal A_n]
 =\mathbb E_{\rm ph}[H\mid\mathcal A_n].
\end{equation}
Thus exponential reweighting changes the field
distribution while preserving the conditional law of
every finite collection of additional overlaps.

We now compute the density of the scalar pair law
relative to its Gaussian counterpart, with $r\in(0,K)$
fixed. Let $B^1,B^2$ be standard Brownian paths that
coincide through $r$ and have independent increments
after $r$. Put
\[
 W_s^{0,j}=h+B_s^j,\qquad j=1,2,
\]
and write $W_s^0$ for their common value when $s\le r$.
Let $\mathbb E_0$ denote expectation over these Gaussian
paths. The scalar pair law has density
$e^{\mathcal S_r(W^0)}$, where
\begin{equation}\label{field:tree-weight}
 \begin{aligned}
 \mathcal S_r(W^0)={}&
 \phi(K,W_K^{0,1})+\phi(K,W_K^{0,2})
 -a(r)\phi(r,W_r^0)\\
 &-\int_{[0,r]}\phi(s,W_s^0)\,\rho(ds)\\
 &-\sum_{j=1}^2
 \int_{(r,K]}\phi(s,W_s^{0,j})\,\rho(ds).
 \end{aligned}
\end{equation}
Indeed, the calculation leading to
\zeqref{eq:si-scalar-density}, applied to the common
segment, gives the logarithmic density
\[
 a(r)\phi(r,W_r^0)
 -\int_{[0,r]}\phi(s,W_s^0)\,\rho(ds).
\]
For continuation $j$, conditional on the field at $r$,
it gives
\[
 \phi(K,W_K^{0,j})-a(r)\phi(r,W_r^0)
 -\int_{(r,K]}\phi(s,W_s^{0,j})\,\rho(ds).
\]
Adding the three expressions proves
\zeqref{field:tree-weight}. Each segment density is a
normalized Girsanov density, so
$\mathbb E_0e^{\mathcal S_r}=1$.

To integrate by parts against the centered second
endpoint $W_K^{0,2}-h$, we differentiate this density.
The covariance of $W_K^{0,2}-h$ with a field at time $s$
is $s$ on 
the common segment, $r$ on the first continuation,
and $s$ on the second continuation. The terminal
potentials therefore contribute
$r\tanh W_K^{0,1}+K\tanh W_K^{0,2}$.
The negative potential terms contribute $-J_r$,
which, under the scalar pair law, is
\begin{equation}\label{eq:si-J}
 \begin{aligned}
 J_r={}&r a(r)\mathsf u(r,\mathsf X_r)
 +\int_{[0,r]}s\,\mathsf u(s,\mathsf X_s)\,\rho(ds)\\
 &+\int_{(r,K]}
 \bigl[r\,\mathsf u(s,\mathsf X_s^1)
       +s\,\mathsf u(s,\mathsf X_s^2)\bigr]\,\rho(ds).
 \end{aligned}
\end{equation}
All coefficients are nonnegative, and their sum is
\[
 r a(r)+\int_{[0,r]}s\,\rho(ds)
 +\int_{(r,K]}(r+s)\,\rho(ds)
 =r+m_1.
\]
Since $|\mathsf u|\le1$, we obtain
$|J_r|\le r+m_1\le2K$.
The intervals in these formulas place any atom at $r$
in the integral over $[0,r]$. 
\begin{lemma}\label{lem:si-pair-subtraction}
Fix $n\ge2$, $z>0$, and $\theta>1$, and use the sign
convention and tilted expectations defined above.
Almost surely on $\{0<r<K\}$,
\begin{equation}\label{eq:all-replica-subtraction}
 n\mathbb E_z[T_{2,n+1}M_{n+1}\mid\mathcal A_n]
 =\mathbb E_{{\rm sc},z}J_r
 +\sum_{\ell=3}^nT_{2\ell}
       \mathbb E_z[M_\ell\mid\mathcal A_n].
\end{equation}
The sum is empty when $n=2$.
\end{lemma}

\begin{proof}
Apply \zeqref{eq:unit-ibp} with $i=2$ and
\[
 g(\mathbf W)=e^{z(W_1-\theta W_2)}.
\]
Since $T_{21}=r$ and $T_{22}=K$, the derivative term is
\[
 \sum_{j=1}^nT_{2j}\partial_jg
 =z(r-\theta K)g.
\]
Using arbitrary bounded $\mathcal A_n$-measurable
multipliers and dividing the resulting conditional
identity by $Z_z(r,\theta)$ gives
\begin{equation}\label{eq:si-physical-pair}
 \begin{aligned}
 \mathbb E_z[Y_2\mid\mathcal A_n]-z(r-\theta K)
 ={}&r\mathbb E_z[M_1\mid\mathcal A_n]
     +K\mathbb E_z[M_2\mid\mathcal A_n]\\
 &+\sum_{\ell=3}^nT_{2\ell}
       \mathbb E_z[M_\ell\mid\mathcal A_n]\\
 &-n\mathbb E_z[T_{2,n+1}M_{n+1}\mid\mathcal A_n].
 \end{aligned}
\end{equation}

For the scalar calculation, integrate by parts against
$W_K^{0,2}-h$ under the Gaussian pair law, with weight
\[
 e^{\mathcal S_r(W^0)+z(W_K^{0,1}-\theta W_K^{0,2})}.
\]
Differentiating the exponential factor involving the
two endpoints gives $z(r-\theta K)$. Differentiating
$\mathcal S_r$ gives the two terminal magnetizations
with coefficients $r,K$, followed by $-J_r$, as
computed in \zeqref{eq:si-J}. After normalization,
\begin{equation}\label{eq:si-scalar-pair}
 \begin{aligned}
 \mathbb E_{{\rm sc},z}[\mathsf X_K^2-h]-z(r-\theta K)
 ={}&r\mathbb E_{{\rm sc},z}\tanh\mathsf X_K^1+K\mathbb E_{{\rm sc},z}\tanh\mathsf X_K^2
   -\mathbb E_{{\rm sc},z}J_r.
 \end{aligned}
\end{equation}

The conditional field law
\zeqref{eq:si-field-interface} implies
\[
 \mathbb E_z[\psi(W_1,W_2)\mid\mathcal A_n]
 =\mathbb E_{{\rm sc},z}
   \psi(\mathsf X_K^1,\mathsf X_K^2)
\]
whenever the expectations are absolutely integrable.
Apply this with $\psi(x_1,x_2)=x_2-h$, $\tanh x_1$,
and $\tanh x_2$. Since $Y_2=W_2-h$, the centered
field terms and the two magnetization terms agree in
\zeqref{eq:si-physical-pair} and
\zeqref{eq:si-scalar-pair}. Subtracting leaves precisely
\zeqref{eq:all-replica-subtraction}.

To justify the scalar integration by parts, first
replace the measure integrals in
\zeqref{field:tree-weight} by finite quadrature,
retaining any atoms at $0,r,K$ exactly. The sum of
the absolute potential coefficients is
\[
 2+a(r)+\rho([0,r])+2\rho((r,K])=4.
\]
The linear growth bound for $\phi$ and the bound
$|\phi_x|\le1$ therefore dominate the approximating
weights and their Gaussian derivatives by
\[
 C\exp\left\{
 C\max_{j=1,2}\sup_{0\le t\le K}|W_t^{0,j}-h|
 \right\}
\]
for fixed $z,\theta$. This bound is integrable under
the Gaussian pair law. Continuity of $\phi$ and
$\mathsf u$ along the paths permits passage from
quadrature to the measure integrals by dominated
convergence. Thus both identities hold for every
fixed finite $z,\theta$.
\end{proof}

We choose conditional versions of these identities before
allowing $\theta$ to depend on an overlap. For each $n\ge2$,
fix a regular conditional law of the fields and overlaps
of the first $n+1$ replicas given $\mathcal A_n$.
First prove the integrated identities for rational $z>0$
and $\theta>1$, against arbitrary bounded
$\mathcal A_n$-measurable functions. The corresponding
conditional identities then hold on one event of
probability one for all these parameters and replica counts.

By \zeqref{eq:si-field-moments}, we may also require on
this event that
\[
 \mathbb E_{\rm ph}\!\left[
 \exp\left(c\sum_{i=1}^{n+1}|Y_i|\right)
 \,\middle|\,\mathcal A_n\right]<\infty
 \qquad(n\ge2,\ c\in\mathbb N).
\]
These conditional moments dominate the exponential
functions and their derivatives on bounded parameter
sets. Conditional dominated convergence therefore makes
the conditional expectations continuous in $z,\theta$.
The scalar expectations are continuous as well, and
$Z_z(r,\theta)$ is positive and bounded away from zero
on compact parameter sets. Thus the identities extend,
on the same event, to every finite $z>0$ and $\theta>1$.

Approximation by simple functions then permits bounded
$\mathcal A_n$-measurable choices of the parameters.
In particular, for fixed $\tau\in(0,K)$, choose
$\theta=\tau/r$ on $\{1/M\le r<\tau\}$ and take the
union over positive integers $M$. This gives the
identities on $\{0<r<\tau\}$.
During Gaussian differentiation, the reference vectors
are fixed, so their overlaps and the parameters chosen
from those overlaps are held constant.

\subsection{Threshold probabilities with a smaller positive overlap}
\label{subsec:si-outside}

An old replica with $0<U_{1j}<\tau$ allows us to apply
the comparison of two cavity fields from the preceding
subsection. We use it to determine the conditional
probability that $U_{1,n+1}>\tau$.

Write $\ell=\Law_{\rm ph}(U_{12})$ for the overlap
marginal, which has not yet been identified. We call
$\tau\in(0,K)$ a common continuity threshold if
\[
 \rho(\{\tau\})=\ell(\{\tau\})=0.
\]

\begin{lemma}\label{prop:outside-attachment}
Fix $n\ge2$ and a deterministic common continuity
threshold $\tau$. Almost surely on the event
\[
 \{\text{there exists }j\in\{2,\ldots,n\}
       \text{ with }0<U_{1j}<\tau\},
\]
we have
\begin{equation}\label{eq:outside-attachment}
 n\Prob_{\rm ph}(U_{1,n+1}>\tau\mid\mathcal A_n)
 =
 \rho((\tau,K])
 +\sum_{j=2}^n\one_{\{U_{1j}>\tau\}}.
\end{equation}
\end{lemma}

\begin{proof}
By relabeling the old replicas, it suffices to work on
$\{0<U_{12}<\tau\}$. Put $r=U_{12}$ and use the sign
convention of the preceding subsection, so that
$T_{12}=r>0$. Choose
\[
 \theta=\frac{\tau}{r}>1.
\]
The preceding argument justifies this overlap-dependent
choice in the identities for every finite $z>0$.
We will pass to the limit $z\to\infty$ in
\zeqref{eq:all-replica-subtraction}.

We first estimate the cavity fields $Y_k$ under the 
exponential tilt. For $m\ge n$ and $1\le k\le m$, put
\[
 b_k=T_{k1}-\theta T_{k2},\qquad E_k=Y_k-zb_k.
\]
Apply \zeqref{eq:unit-ibp} with $m$ replicas, in
direction $Y_k$, to
$E_ke^{z(W_1-\theta W_2)}$.
The reference vectors, and hence $b_k$ and $\theta$,
are held fixed during differentiation.
Differentiating $E_k$ contributes $K$, while
differentiating the exponential contributes $zb_kE_k$.
Moving the latter term to the left gives
\[
 \mathbb E_z[E_k^2\mid\mathcal A_m]
 =K+\mathbb E_z[E_kD_k\mid\mathcal A_m],
 \qquad
 D_k=\sum_{\ell=1}^mT_{k\ell}M_\ell
      -mT_{k,m+1}M_{m+1}.
\]
The coefficient $m$ comes from differentiating
$Z_S^{-m}$ in the sampling density. Since
$|D_k|\le2mK$, conditional Cauchy--Schwarz gives
$x^2\le K+2mKx$, where
$x=(\mathbb E_z[E_k^2\mid\mathcal A_m])^{1/2}$.
Therefore
\begin{equation}\label{eq:all-replica-residual}
 \mathbb E_z[E_k^2\mid\mathcal A_m]
 \le(2mK+\sqrt K)^2.
\end{equation}
Chebyshev's inequality implies that $Y_k/z$ converges
to $b_k$ in tilted conditional probability.
Since $M_k=\tanh(h+Y_k)$ and $|M_k|\le1$, it follows
on $\{b_k\ne0\}$ that
\[
 \mathbb E_z[
 |M_k-\operatorname{sgn}(b_k)|\mid\mathcal A_m]
 \longrightarrow0.
\]

We next express these limits through threshold
indicators. For every $k\notin\{1,2\}$,
the sign relations and ultrametricity give
\begin{equation}\label{eq:all-replica-sign-threshold}
 T_{2k}\operatorname{sgn}(T_{1k}-\theta T_{2k})
 =2r\one_{\{U_{1k}>\tau\}}-U_{2k}
 \qquad\text{when }U_{1k}\ne\tau.
\end{equation}
Indeed, if either $U_{1k}$ or $U_{2k}$ vanishes, both
vanish by Lemma~\zcref[noname]{geo:sign}.
Otherwise, $T_{1k}$ and $T_{2k}$ have a common sign
$\varepsilon$ because $T_{12}>0$. Put
$x=U_{1k}$ and $y=U_{2k}$.
If $x>r$, ultrametricity gives $y=r$, and
\[
 b_k=\varepsilon(x-\tau),\qquad
 T_{2k}\operatorname{sgn}(b_k)
 =r\operatorname{sgn}(x-\tau).
\]
The remaining possibilities are $x=y\le r$ and
$x=r<y$. In both cases $x-\theta y<0$, so
$T_{2k}\operatorname{sgn}(b_k)=-y$.
These cases prove \zeqref{eq:all-replica-sign-threshold}.

In particular, if $b_k=0$ and $U_{1k}\ne\tau$, then
$T_{2k}=0$. Thus no magnetization limit is needed on
that event. Replica exchangeability and
$\ell(\{\tau\})=0$ exclude $U_{1k}=\tau$ almost surely,
and hence conditionally almost surely.
For old labels $k=3,\ldots,n$, the preceding estimates
therefore give the limits of
$T_{2k}\mathbb E_z[M_k\mid\mathcal A_n]$ directly.

For $k=n+1$, use \zeqref{eq:all-replica-residual}
with $m=n+1$. This introduces replica $n+2$ through
$D_k$, for which only $|M_{n+2}|\le1$ is needed.
The conditional expectation
$\mathbb E_z[T_{2,n+1}M_{n+1}\mid\mathcal A_{n+1}]$
is bounded in absolute value by $K$. 
The tower property and
\zeqref{eq:si-tilt-preserves-array} give
\[
 \begin{aligned}
 &\mathbb E_z[T_{2,n+1}M_{n+1}\mid\mathcal A_n]\\
 &\quad=
 \mathbb E_{\rm ph}\!\left[
 \mathbb E_z[T_{2,n+1}M_{n+1}\mid\mathcal A_{n+1}]
 \,\middle|\,\mathcal A_n\right]\longrightarrow
 2r\Prob_{\rm ph}(U_{1,n+1}>\tau\mid\mathcal A_n)
 -\mathbb E_{\rm ph}[U_{2,n+1}\mid\mathcal A_n].
 \end{aligned}
\]
Here conditional dominated convergence applies under
the conditional law of the additional overlaps given
$\mathcal A_n$, which is unchanged by the tilt.

It remains to compute the scalar term
$\mathbb E_{{\rm sc},z}J_r$.
Cameron--Martin translation of the reference Brownian
pair gives the displacement $z\lambda$, where
\[
 \lambda(s)=s(1-\theta)\quad(0\le s\le r),
 \qquad
 \lambda^1(s)=s-\tau,\quad
 \lambda^2(s)=r-\theta s\quad(r\le s\le K).
\]
Set
\[
 \mathcal B=
 \max_{j=1,2}\sup_{0\le s\le K}|W_s^{0,j}-h|.
\]
After subtracting the deterministic term
$\mathcal S_r(h+z\lambda)$, the translated logarithmic
density has remainder
\[
 R_z(W^0)=
 \mathcal S_r(W^0+z\lambda)-\mathcal S_r(h+z\lambda).
\]
The bound $|\phi_x|\le1$ and the total absolute
coefficient $4$ in \zeqref{field:tree-weight} give
$|R_z|\le4\mathcal B$. Hence
\[
 \frac{e^{R_z(W^0)}}{\mathbb E_0e^{R_z(W^0)}}
 \le
 \frac{e^{4\mathcal B}}{\mathbb E_0e^{-4\mathcal B}},
\]
an integrable bound independent of $z$.

The common displacement is negative for $s\in(0,r]$,
and the displacement on continuation $2$ is negative
throughout $[r,K]$. On continuation $1$, its sign is
that of $s-\tau$. The scalar estimates and oddness of
$\mathsf u$ imply
$\mathsf u(s,x)\to\pm1$ as $x\to\pm\infty$.
Since $\rho(\{\tau\})=0$, evaluating $J_r$ along
the translated paths therefore gives the deterministic
limit
\begin{equation}\label{eq:all-replica-scalar-limit}
 \begin{aligned}
 \lim_{z\to\infty}\mathbb E_{{\rm sc},z}J_r
 ={}&-r a(r)-\int_{[0,r]}s\,\rho(ds)\\
 &+\int_{(r,K]}
 [r\operatorname{sgn}(s-\tau)-s]\,\rho(ds)\\
 ={}&-(m_1+r)+2r\rho((\tau,K]).
 \end{aligned}
\end{equation}
To justify the expectation limit under the varying
densities, subtract this deterministic value from
$J_r$ and apply dominated convergence using
$|J_r|\le2K$ and the preceding density bound.
The coefficient at zero vanishes, while the
displacements at $r$ and $K$ are nonzero, so atoms
at those points are included in the calculation.
For the last equality, the terms with coefficient
$-s$ sum to $-m_1$, and the remaining terms sum to
\[
 -r a(r)-r\rho((r,K])+2r\rho((\tau,K])
 =-r+2r\rho((\tau,K]).
\]

Put
\[
 p=\Prob_{\rm ph}(U_{1,n+1}>\tau\mid\mathcal A_n),
 \qquad
 B=\sum_{\ell=3}^nU_{2\ell},
 \qquad
 N_\tau=\sum_{\ell=3}^n\one_{\{U_{1\ell}>\tau\}}.
\]
By \zeqref{eq:si-first-moment}, the left side of
\zeqref{eq:all-replica-subtraction} converges to
\[
 2nrp-(m_1+r+B).
\]
By \zeqref{eq:all-replica-scalar-limit} and
\zeqref{eq:all-replica-sign-threshold}, its right
side converges to
\[
 -(m_1+r)+2r\rho((\tau,K])+2rN_\tau-B.
\]
Canceling $-m_1-r-B$ and dividing by $2r>0$ gives
\[
 np=\rho((\tau,K])+N_\tau.
\]
Since $U_{12}=r<\tau$, the missing indicator for label
$2$ is zero. This proves \zeqref{eq:outside-attachment}
on $\{0<U_{12}<\tau\}$.
The sign change leaves all $U$-entries unchanged.
Relabeling the old replicas and taking the finite
union over the possible labels $j=2,\ldots,n$
completes the proof.
\end{proof}

\subsection{Removing the restriction on the old overlaps}
\label{subsec:si-weighted-exchange}

The preceding lemma requires $U_{1j}\in(0,\tau)$ for
some old replica $j$. To remove this restriction, we
assign an additional replica a nonnegative weight
$\psi(U_{1,n+1})$ that vanishes at zero and is small
on $[\tau,K]$. Exchanging two additional replicas
relates the error in the desired conditional probability
to a weighted error for the enlarged sample.
When $U_{1,n+1}\in(0,\tau)$, that enlarged sample
satisfies the hypothesis of
Lemma~\zcref[noname]{prop:outside-attachment}, so its
error vanishes.

\begin{lemma}\label{lem:si-weighted-witness}
Fix $n\ge1$ and a deterministic common continuity
threshold $\tau$. Let $\psi:[0,K]\to[0,\infty)$ be
bounded and Borel, with $\psi(0)=0$ and
$\int_{[0,K]}\psi\,d\rho=1$. Suppose, almost surely,
\begin{equation}\label{eq:si-witness-diagonal}
 m\mathbb E_{\rm ph}
 [\psi(U_{1,m+1})\mid\mathcal A_m]
 =1+\sum_{j=2}^m\psi(U_{1j}),
 \qquad m=n,n+1.
\end{equation}
For $m=n,n+1$, define
\[
 \begin{aligned}
 p_m&=\Prob_{\rm ph}(U_{1,m+1}>\tau\mid\mathcal A_m),\\
 B_m&=\rho((\tau,K])
       +\sum_{j=2}^m\one_{\{U_{1j}>\tau\}},\\
 \Delta_m&=mp_m-B_m,
 \end{aligned}
 \qquad
 D_n=1+\sum_{j=2}^n\psi(U_{1j}).
\]
Then, almost surely,
\begin{equation}\label{eq:si-weighted-defect}
 \frac{D_n}{n}\Delta_n
 =\mathbb E_{\rm ph}
 [\psi(U_{1,n+1})\Delta_{n+1}\mid\mathcal A_n].
\end{equation}
If $\sup_{t\in[\tau,K]}\psi(t)\le\varepsilon$
for some $\varepsilon\ge0$, then
\begin{equation}\label{eq:si-defect-bound}
 \left|p_n-\frac{B_n}{n}\right|
 \le(n+1)\varepsilon
 \quad\text{almost surely}.
\end{equation}
\end{lemma}

\begin{proof}
Put $a=n+1$, $b=n+2$, and $q=U_{1a}$.
The tower property gives
\[
 \begin{aligned}
 H&:=\mathbb E_{\rm ph}
 [\psi(q)p_{n+1}\mid\mathcal A_n]=\mathbb E_{\rm ph}
 [\psi(q)\one_{\{U_{1b}>\tau\}}\mid\mathcal A_n].
 \end{aligned}
\]
Exchanging labels $a,b$ leaves $\mathcal A_n$ unchanged.
Conditioning next on $\mathcal A_{n+1}$ and applying
\zeqref{eq:si-witness-diagonal} with $m=n+1$ yields
\begin{equation}\label{eq:si-weighted-exchange}
 \begin{aligned}
 (n+1)H
 &=(n+1)\mathbb E_{\rm ph}\!\left[
 \one_{\{q>\tau\}}
 \mathbb E_{\rm ph}[\psi(U_{1b})\mid\mathcal A_{n+1}]
 \,\middle|\,\mathcal A_n\right]\\
 &=\mathbb E_{\rm ph}
 [\one_{\{q>\tau\}}(D_n+\psi(q))\mid\mathcal A_n]\\
 &=D_np_n+
 \mathbb E_{\rm ph}[\psi(q)\one_{\{q>\tau\}}\mid\mathcal A_n].
 \end{aligned}
\end{equation}
The same assumption with $m=n$ gives
\[
 \mathbb E_{\rm ph}[\psi(q)\mid\mathcal A_n]=D_n/n.
\]
Since $B_{n+1}=B_n+\one_{\{q>\tau\}}$, we obtain
\[
 \begin{aligned}
 &\mathbb E_{\rm ph}
 [\psi(q)\Delta_{n+1}\mid\mathcal A_n]\\
 &\quad=(n+1)H-\frac{B_nD_n}{n}
 -\mathbb E_{\rm ph}
 [\psi(q)\one_{\{q>\tau\}}\mid\mathcal A_n]\\
 &\quad=D_n\left(p_n-\frac{B_n}{n}\right)
 =\frac{D_n}{n}\Delta_n.
 \end{aligned}
\]
This proves \zeqref{eq:si-weighted-defect}.

On $\{0<q<\tau\}$, the enlarged sample has $n+1\ge2$
old replicas, and label $a$ satisfies $0<U_{1a}<\tau$.
Lemma~\zcref[noname]{prop:outside-attachment} therefore
gives $\Delta_{n+1}=0$ on this event.
At $q=0$, $\psi(q)=0$. On $\{q\ge\tau\}$, we have
$\psi(q)\le\varepsilon$ and $|\Delta_{n+1}|\le n+1$,
the latter following from
$0\le(n+1)p_{n+1},B_{n+1}\le n+1$.
Thus \zeqref{eq:si-weighted-defect} and $D_n\ge1$ give
\[
 \left|p_n-\frac{B_n}{n}\right|
 \le\frac{1}{D_n}
 \mathbb E_{\rm ph}
 [\psi(q)|\Delta_{n+1}|\mid\mathcal A_n]
 \le(n+1)\varepsilon,
\]
which proves \zeqref{eq:si-defect-bound}.
\end{proof}

Fix a common continuity threshold $\tau$ with
$\rho((0,\tau))>0$. We construct weights satisfying
Lemma~\zcref[noname]{lem:si-weighted-witness} whose
supremum on $[\tau,K]$ tends to zero.

Choose $0<a_0\le s<\tau$ such that
$\rho([a_0,s])>0$. For all sufficiently large $y$, put
\[
 d_y=\int_{[0,K]}r\delta_y(r)\,\rho(dr).
\]
Since $\delta_y$ is decreasing and positive on $[0,K)$,
\begin{equation}\label{eq:si-witness-positivity}
 d_y\ge a_0\rho([a_0,s])\delta_y(s)>0.
\end{equation}
We may therefore define
\begin{equation}\label{eq:si-witness}
 \psi_y(r)=\frac{r\delta_y(r)}{d_y},
 \qquad 0\le r\le K.
\end{equation}
For each such $y$, this function is bounded, continuous,
and nonnegative, with $\psi_y(0)=0$ and
$\int_{[0,K]}\psi_y\,d\rho=1$.
Subtracting \zeqref{mix:GG} with $f=f_y/Z_y$ from
\zeqref{eq:si-first-moment}, then dividing by $d_y$,
proves \zeqref{eq:si-witness-diagonal} for every
$m\ge1$.

Moreover, \zeqref{source:eq:separation} gives
\begin{equation}\label{eq:si-witness-tail}
 \sup_{r\in[\tau,K]}\psi_y(r)
 \le
 \frac{K\delta_y(\tau)}
      {a_0\rho([a_0,s])\delta_y(s)}
 \longrightarrow0.
\end{equation}
Apply Lemma~\zcref[noname]{lem:si-weighted-witness}
for each $y$, then let $y\to\infty$ through integers.
The error bound \zeqref{eq:si-defect-bound} yields,
almost surely for every $n\ge1$,
\begin{equation}\label{eq:si-threshold-gg}
 n\Prob_{\rm ph}(U_{1,n+1}>\tau\mid\mathcal A_n)
 =\rho((\tau,K])
  +\sum_{j=2}^n\one_{\{U_{1j}>\tau\}}.
\end{equation}
The error estimate depends only on the supremum of
$\psi_y$ on $[\tau,K]$, so possible growth below
$\tau$ does not affect this limit.

If $\rho$ is supported on $\{0,K\}$, then $d_y=0$.
This case is treated in
\zcref{subsec:si-distribution-completion}.

\subsection{Completing the conditional overlap law}
\label{subsec:si-distribution-completion}

\begin{lemma}\label{lem:si-distribution-completion}
Fix $n\ge1$. Suppose \zeqref{mix:GG} and
\zeqref{eq:si-first-moment} hold with
$m\in\{1,n,n+1\}$ old replicas, using $f=f_y/Z_y$
in \zeqref{mix:GG} for all sufficiently large $y$.
Suppose also that \zeqref{eq:outside-attachment} holds
with $n+1$ old replicas at every common continuity
threshold. Then, almost surely,
\[
 \Law_{\rm ph}(U_{1,n+1}\mid\mathcal A_n)
 =\frac1n\rho+\frac1n\sum_{j=2}^n\delta_{U_{1j}}.
\]
\end{lemma}

\begin{proof}
Let $\Pi_n$ be a regular conditional law of
$U_{1,n+1}$ given $\mathcal A_n$, and define the
finite signed measure
\begin{equation}\label{eq:conditional-defect}
 \eta_n=n\Pi_n-\rho-\sum_{j=2}^n\delta_{U_{1j}}.
\end{equation}
Its total mass is zero, and
\zeqref{eq:si-first-moment} gives
\[
 \int_{[0,K]}r\,\eta_n(dr)=0.
\]
For $n=1$, $\Pi_1=\ell$ almost surely.

Since $\rho$ and $\ell$ have at most countably many
atoms, choose a deterministic countable dense set
$D\subset(0,K)$ of common continuity thresholds.
By \zeqref{eq:si-threshold-gg}, on one event of
probability one,
\[
 \eta_n((\tau,K])=0
 \qquad\text{for every }\tau\in D
 \text{ with }\rho((0,\tau))>0.
\]

If $\rho((0,\varepsilon))>0$ for every
$\varepsilon>0$, these identities hold for all
$\tau\in D$. The tails on this dense set determine
the restriction of $\eta_n$ to $(0,K]$, so that
restriction is zero. Its zero total mass then gives
$\eta_n(\{0\})=0$.

In the remaining case, put
\[
 a_*=\inf(\supp\rho\cap(0,K])>0.
\]
We first show that $\ell((0,a_*))=0$.
Subtracting the two identities with one old replica
gives
\begin{equation}\label{eq:si-deficit-marginal}
 \int_{[0,K]}r\delta_y(r)\,\ell(dr)
 =\int_{[0,K]}r\delta_y(r)\,\rho(dr).
\end{equation}
For $0<a_0\le s<t<a_*$, nonnegativity and
monotonicity imply
\[
 a_0\ell([a_0,s])\delta_y(s)
 \le\int_{[0,K]}r\delta_y(r)\,\rho(dr)
 \le K\delta_y(t).
\]
The last inequality uses
$\supp\rho\subset\{0\}\cup[a_*,K]$ and the vanishing
factor $r$ at zero. Divide by $\delta_y(s)>0$ and
apply \zeqref{source:eq:separation} to obtain
$\ell([a_0,s])=0$. A countable covering of
$(0,a_*)$ yields
\begin{equation}\label{eq:si-initial-gap}
 \ell((0,a_*))=0.
\end{equation}

By replica exchangeability,
\[
 \mathbb E_{\rm ph}\Pi_n((0,a_*))
 =\ell((0,a_*))=0.
\]
Thus $\Pi_n$ assigns no mass to this interval
almost surely, and none of the old overlaps lies
there almost surely. Consequently, $\eta_n$ is
supported on $\{0\}\cup[a_*,K]$.

If $a_*<K$, then $\rho((0,\tau))>0$ for every
$\tau>a_*$. The identities at the thresholds in
$D\cap(a_*,K)$ therefore show that $\eta_n$
vanishes on $(a_*,K]$. If $a_*=K$, this interval
is empty. In either case, $\eta_n$ is supported
on $\{0,a_*\}$. Its first moment gives
\[
 0=\int_{[0,K]}r\,\eta_n(dr)
   =a_*\eta_n(\{a_*\}),
\]
so $\eta_n(\{a_*\})=0$. Its zero total mass then
gives $\eta_n(\{0\})=0$.

The determining thresholds were chosen countably.
Thus, on one event of probability one, $\eta_n=0$
as a measure, which proves the conditional law
\zeqref{eq:si-shell-conditional-law} for every
bounded Borel test function.
\end{proof}

\begin{proof}[Proof of Proposition~{\zcref[noname]{prop:physical-shell-identification}}]
Lemma~\zcref[noname]{geo:sign} proves the sign relation
\zeqref{eq:si-shell-signs} at zero field, and
Lemma~\zcref[noname]{geo:um} proves ultrametricity of $U$.

Fix $n\ge1$. Lemma~\zcref[noname]{mix:identity}
gives \zeqref{mix:GG} with $f=f_y/Z_y$ for all
sufficiently large $y$. Taking the limit gives
\zeqref{eq:si-first-moment}. Both identities hold
with $1,n,n+1$ old replicas.
Lemma~\zcref[noname]{prop:outside-attachment}
gives \zeqref{eq:outside-attachment} with $n+1$
old replicas at every common continuity threshold.
Lemma~\zcref[noname]{lem:si-distribution-completion}
therefore proves \zeqref{eq:si-shell-conditional-law}.
A countable intersection over $n$ makes these
equalities of conditional measures simultaneous.

Taking $n=1$ gives
\[
 \kappa_\#\lambda_S=\rho=\kappa_\#\mu.
\]
Since $\kappa$ is strictly increasing, this implies
$\lambda_S=\mu$.
\end{proof}

\begin{proof}[Proof of Proposition~{\zcref[noname]{prop:marginal-sphere}} and
Theorem~{\zcref[noname]{thm:gg}}]
If $h=0$ and $\mu=\delta_0$,
Lemma~\zcref[noname]{lem:zero-measure} gives the zero
array, which satisfies both conclusions. Otherwise,
$Q>0$.

For each deterministic $S\in\{Q\}\cup\mathcal Z_+$
with $\Prob(E_S)>0$,
Proposition~\zcref[noname]{prop:shell-marking}
supplies the cavity formula and conditional
magnetization laws required by
Proposition~\zcref[noname]{prop:physical-shell-identification}.
Thus $\lambda_S=\mu$ for every such $S$.
Proposition~\zcref[noname]{prop:shell-reduction}
then shows that every physical directing measure
is supported on $\{v:\|v\|^2=Q\}$ and that
$S_{12}$ has law $\mu$. By
Lemma~\zcref[noname]{lem:norm-observable}, the
squared norms are determined by the overlap array,
so this support statement holds in every directing
representation. This proves
Proposition~\zcref[noname]{prop:marginal-sphere}.

In the cavity representation, both the predecessor
and successor measures are therefore supported on
$\{v:\|v\|^2=Q\}$. Hence $\Prob(E_Q)=1$ and
$\widehat G_Q=\widehat G$ almost surely.
We may thus apply
\zeqref{eq:si-shell-conditional-law} directly to
the successor array with $S=Q$.
Since
\[
 U_{ij}=\kappa(S_{ij}),\qquad
 \rho=\kappa_\#\mu,\qquad
 \mathcal A_n=\mathcal R_n,
\]
pushing forward that conditional law by
$\kappa^{-1}$ gives
\zeqref{eq:full-gg-conditional} for every $n\ge2$.
The conditioning sigma-field contains the signed
old overlaps, including at zero field.
Multiplying the conditional identity by a bounded
$\mathcal R_n$-measurable function and taking
expectations gives \zeqref{eq:full-gg}.
\end{proof}

\section{Reconstruction of the full overlap structure}\label{sec:rpc}
In this section we prove \zcref{thm:array} by reconstructing the full
overlap array from the marginal and identities obtained in
\zcref{sec:regression}. \zcref{subsec:rpc-reconstruction} recalls Ruelle
probability cascades and applies Panchenko's ultrametricity theorem and
the Baffioni--Rosati reconstruction theorem to identify the nonnegative
overlap law. At zero field, this argument is applied to squared overlaps.
\zcref{subsec:sign-reconstruction} then uses the sign relation in
Lemma~\zcref[noname]{lem:signs-all-zero-field} and spin-flip symmetry
to recover the signed array from its absolute values. Identifying
every subsequential limit gives convergence of the full sequence.

\subsection{Ruelle probability cascades and reconstruction}\label{subsec:rpc-reconstruction}

We first construct the Ruelle probability cascade associated with
a finitely supported probability measure on $[0,1]$. Write
\[
 \nu=\sum_{j=0}^r(m_j-m_{j-1})\delta_{q_j},
 \qquad
 0\le q_0<\cdots<q_r\le1,
 \qquad
 0=m_{-1}<m_0<\cdots<m_r=1.
\]
Thus $m_j=\nu([0,q_j])$.

Suppose first that $r\ge1$. Consider a rooted tree of depth $r$
with leaves indexed by $\mathbb N^r$. At each vertex of depth
$j<r$, independently generate a Poisson point process on
$(0,\infty)$ with intensity
\[
 m_jx^{-1-m_j}\,dx.
\]
Arrange its points in decreasing order and assign them to the
outgoing edges. For a leaf $\alpha=(\alpha_1,\ldots,\alpha_r)$,
write $u_{\alpha_1\cdots\alpha_\ell}$ for the point assigned
to its edge at depth $\ell$, and set
\[
 v_\alpha=\prod_{\ell=1}^r
 u_{\alpha_1\cdots\alpha_\ell},
 \qquad
 W_\alpha=
 \frac{v_\alpha}{\sum_{\beta\in\mathbb N^r}v_\beta}.
\]
The denominator is positive and finite almost surely, as the
total-mass calculation below shows. The $W_\alpha$ are the
\emph{cascade weights}.

For two leaves, let $\alpha\wedge\beta$ be the length of their
common initial segment:
\[
 \alpha\wedge\beta
 =\max\bigl\{0\le j\le r:
       \alpha_\ell=\beta_\ell\text{ for }1\le\ell\le j\bigr\}.
\]
In particular, $\alpha\wedge\beta=r$ when $\alpha=\beta$.
Conditional on the weights, sample leaves
$\alpha^1,\alpha^2,\ldots$ independently with probabilities
$(W_\alpha)_\alpha$, and define
\[
 \mathcal U_{ij}=q_{\alpha^i\wedge\alpha^j},
 \qquad i\ne j.
\]
We call this the finite-level RPC array associated with $\nu$.

We next derive its sampling rule conditional on all overlaps
among the first $n$ replicas. Put
\[
 \mathscr H_n
 =\sigma(\mathcal U_{ab}:1\le a<b\le n).
\]
This sigma-field records the nested partitions of the sampled
labels induced by their common ancestors. The conditional
probabilities below average over the cascade weights.

We use two facts about Poisson--Dirichlet distributions.
Let $0\le b<a<1$, and let $T$ be the sum of the points of a
Poisson process with intensity $a x^{-1-a}\,dx$.
Under the change of probability with density $T^b/\E T^b$,
the normalized points have law $\PD(a,-b)$
\cite[Proposition~14, Eq.~(44)]{PitmanYor1997}.
For independent sampling from these normalized masses, the
probability of a prescribed partition of $n$ labels into blocks
of sizes $n_1,\ldots,n_k$ is
\begin{equation}\label{eq:rpc-local-partition}
 p_{a,-b}(n_1,\ldots,n_k)
 =\frac{\prod_{\ell=1}^{k-1}(\ell a-b)}{(1-b)_{n-1}}
       \prod_{i=1}^k(1-a)_{n_i-1},
\end{equation}
where $(x)_j=x(x+1)\cdots(x+j-1)$ and $(x)_0=1$
\cite[Appendix~A.2, Proposition~51, Eq.~(178)]{PitmanYor1997}.
Taking ratios gives probabilities
\[
 \frac{n_i-a}{n-b},
 \qquad
 \frac{ka-b}{n-b}
\]
for the next label to join a specified existing block or start
a new block, respectively.

To identify the parameters at each vertex, consider an edge
with exponent $m_j$. Let $\mathcal T$ be the descendant tree
below that edge, and let $T(\mathcal T)$ be its total
unnormalized mass, excluding the edge point $u$.
Its contribution to the parent is $v=uT(\mathcal T)$.
Changing variables gives the marked intensity
\[
 m_j u^{-1-m_j}\,du\,\Prob(d\mathcal T)
 =
 m_j v^{-1-m_j}\,dv\,
 T(\mathcal T)^{m_j}\Prob(d\mathcal T).
\]
Consequently, the contributions form a Poisson process with
intensity
\[
 \E T^{m_j}\,m_jv^{-1-m_j}\,dv,
\]
and their descendant trees are independent marks with law
\[
 \frac{T(\mathcal T)^{m_j}}{\E T^{m_j}}
 \Prob(d\mathcal T).
\]
The marks are independent of the contributions. We may order
the contributions decreasingly, since relabeling children
preserves all overlaps. The constant multiplier of the
intensity disappears upon normalizing the contributions.
This is the marked Poisson calculation in the proof of
\cite[Lemma~2]{PT}.

The required fractional moments are finite. At the last level,
the descendant total is one. Proceeding upward, the total
below an edge of exponent $m_j$, for $j<r-1$, is a scaled
positive stable random variable of exponent $m_{j+1}$.
Its $m_j$-moment is finite because $m_j<m_{j+1}$
\cite[Proposition~10(ii), Eq.~(30)]{PitmanYor1997}.
This induction also shows that the total mass at the root is
positive and finite almost surely.

The same calculation gives a recursive description of the
normalized cascade. At depth $j$, the incoming edge induces
a tilt by the $m_{j-1}$-th power of the total mass; at the root
the exponent is $m_{-1}=0$. The normalized contributions of
the children therefore have law $\PD(m_j,-m_{j-1})$.
This tilt depends only on the contributions, so it preserves
their independence of the descendant marks.

Call a vertex represented if at least one of the first $n$
sampled leaves descends from it, and call the number of such
labels its size. The probability of a prescribed hierarchy
of these labels is the product, over its represented
non-leaf vertices, of the factors
\[
 p_{m_j,-m_{j-1}}(n_1,\ldots,n_k),
\]
where $j$ is the depth of the vertex and $n_1,\ldots,n_k$
are the sizes of its represented children.
There is no counting factor, since the partitions of the
labels are prescribed.

Suppose the next label reaches a represented parent at
depth $j$. Conditional on $\mathscr H_n$, its probability
of entering a specified represented child is
\[
 \frac{\text{child size}-m_j}
      {\text{parent size}-m_{j-1}}.
\]
Indeed, take the ratio of the probabilities of the extended
and original hierarchies. The factors outside the chosen
subtree cancel. Summing over all possible extensions within
the chosen child gives one, leaving the displayed ratio
from \zeqref{eq:rpc-local-partition}. Thus the rule remains
valid after conditioning on the finer partitions within
every represented subtree.

For $0\le j<r$, let $C$ be a block of the relation
\[
 a\sim_j b
 \quad\Longleftrightarrow\quad
 a=b\ \text{or}\ \mathcal U_{ab}>q_j
\]
on $\{1,\ldots,n\}$. Write
$C_0=\{1,\ldots,n\},C_1,\ldots,C_{j+1}=C$ for the label
sets along its ancestor chain. Multiplying the successive
conditional probabilities gives
\begin{equation}\label{eq:rpc-whole-hierarchy-rule}
 \begin{aligned}
 \Prob(n+1\text{ enters }C\mid\mathscr H_n)
 &=\prod_{\ell=0}^j
   \frac{|C_{\ell+1}|-m_\ell}
        {|C_\ell|-m_{\ell-1}}=\frac{|C|-m_j}{n}.
 \end{aligned}
\end{equation}

Taking $n=1$ gives
$\Prob(\mathcal U_{12}>q_j)=1-m_j$, and hence
\[
 \Prob(\mathcal U_{12}=q_j)=m_j-m_{j-1},
 \qquad 0\le j\le r.
\]
Thus $\mathcal U_{12}$ has law $\nu$
\cite[Theorem~3, Eq.~(1.19)]{PT}; see also
\cite{PanchenkoBook}.

More generally, if $C$ contains a fixed label $i\le n$, then
\[
 |C|
 =1+\sum_{\substack{1\le\ell\le n\\\ell\ne i}}
       \one_{\{\mathcal U_{i\ell}>q_j\}}.
\]
Equation~\zeqref{eq:rpc-whole-hierarchy-rule} therefore gives
the tail probabilities of the Ghirlanda--Guerra conditional
law. Since the overlap takes values in
$\{q_0,\ldots,q_r\}$, these tails determine
\[
 \Law(\mathcal U_{i,n+1}\mid\mathscr H_n)
 =\frac1n\nu
  +\frac1n\sum_{\substack{1\le\ell\le n\\\ell\ne i}}
       \delta_{\mathcal U_{i\ell}}.
\]

If $\nu=\delta_q$, we use the degenerate cascade with
$\mathcal U_{ij}=q$ for every $i\ne j$.
For a general probability measure $\nu$ on $[0,1]$, the
RPC array law is the weak limit of the finite-level cascade
laws associated with finitely supported measures converging
weakly to $\nu$. This limit exists and is independent of
the approximation; we denote it by $\mathcal L_\nu$
\cite{PanchenkoBook}.

\begin{proposition}\label{prop:rpc-reconstruction}
Let $(R_{ij})_{i,j\ge1}$ be a weakly exchangeable Gram array
with constant diagonal and off-diagonal marginal
$\nu\in\cP([0,1])$. If the array satisfies the full
Ghirlanda--Guerra identities, then $(R_{ij})_{i<j}$ has
law $\mathcal L_\nu$.
\end{proposition}

\begin{proof}
Put $q_*=\max\supp\nu$. Choose a directing representation
with random measure $G$ and, conditional on $G$, independent
samples $v^1,v^2,\ldots$, so that
\[
 R_{ij}=v^i\cdot v^j\quad(i\ne j).
\]
Write $q_i=\|v^i\|^2$.

Almost surely,
$G^{\otimes2}(v\cdot w>q_*)=0$. If a support point of $G$
had squared norm greater than $q_*$, continuity of the inner
product would give a neighborhood of positive $G$-mass
whose pairs all have overlap greater than $q_*$.
Thus every support point has squared norm at most $q_*$.
If $q_*=0$, then $G=\delta_0$ and the conclusion follows
from the definition of the degenerate cascade.

Suppose $q_*>0$. For $0\le t<q_*$, put
$a=\nu([0,t])<1$. Successive application of the
Ghirlanda--Guerra identities gives
\[
 \Prob(R_{12}\le t,\ldots,R_{1,M+1}\le t)
 =\prod_{n=1}^M
   \left(1-\frac{1-a}{n}\right)
 \longrightarrow0.
\]
Taking a countable sequence $t\uparrow q_*$ and using
exchangeability shows that, almost surely,
\[
 \sup_{j\ne i}R_{ij}=q_*
 \qquad\text{for every }i.
\]
On the other hand,
\[
 R_{ij}\le\sqrt{q_iq_j}\le\sqrt{q_iq_*},
 \qquad q_i\le q_*.
\]
It follows that $q_i=q_*$ almost surely. Hence $G$ is
supported on the sphere of squared radius $q_*$.

The Gram array of the vectors $v^i/\sqrt{q_*}$ has diagonal
one and satisfies the full Ghirlanda--Guerra identities.
Panchenko's theorem \cite[Theorem~1]{PanUM} therefore gives
ultrametricity, which is preserved upon scaling back:
\[
 R_{23}\ge\min(R_{12},R_{13})
 \qquad\text{almost surely}.
\]

Suppose first that $\nu$ has finite support.
The reconstruction theorem of Baffioni and Rosati
\cite[Theorem~6.3, Eqs.~(56)--(57)]{BaffioniRosati}
shows that ultrametricity and the full identities determine
every finite-dimensional overlap distribution from $\nu$.
The finite-level RPC associated with $\nu$ is ultrametric
by its tree construction and satisfies the full identities
by \zeqref{eq:rpc-whole-hierarchy-rule}. Its off-diagonal
law is therefore the unique law with these properties.

For general $\nu$, define
\[
 g_k(x)=2^{-k}\lfloor2^kx\rfloor,
 \qquad 0\le x\le1,
\]
and set
\[
 R^{(k)}_{ii}=1,\qquad
 R^{(k)}_{ij}=g_k(R_{ij})\quad(i\ne j).
\]
For each threshold $t$, ultrametricity makes
\[
 i\sim_t j
 \quad\Longleftrightarrow\quad
 i=j\ \text{or}\ R_{ij}\ge t
\]
an equivalence relation. Its indicator matrix is positive
semidefinite. Since
\[
 R^{(k)}_{ij}
 =2^{-k}\sum_{\ell=1}^{2^k}
   \one_{\{i\sim_{\ell2^{-k}}j\}},
\]
the transformed array is also Gram. It remains weakly
exchangeable and ultrametric. The full identities pass
to this array by applying the original identities to
bounded Borel test functions of the transformed entries.

Its off-diagonal marginal is $(g_k)_\#\nu$, so the
finite-level result identifies its law with that of
the corresponding RPC. Moreover,
\[
 \sup_{x\in[0,1]}|g_k(x)-x|\le2^{-k},
\]
and hence the transformed off-diagonal arrays converge
coordinatewise to the original array. Their laws depend
only on $\nu$, so this proves uniqueness of the
off-diagonal law for arbitrary $\nu$.

Finally, let $\nu_k$ be any finitely supported probability
measures converging weakly to $\nu$, and complete their RPC
arrays with diagonal one. Their laws are tight on the
compact space of off-diagonal arrays with entries in
$[0,1]$. Every subsequential limit is weakly exchangeable,
Gram, and ultrametric, with off-diagonal marginal $\nu$.
The Ghirlanda--Guerra identities for continuous test
functions pass to the limit. Uniqueness of finite Borel
measures then extends these identities to bounded Borel
test functions.

The uniqueness just proved identifies every subsequential
limit with the off-diagonal law of the original array.
Thus the finite-level RPC laws converge to this law,
independently of the approximation and without any
restriction on the atoms of $\nu$. By definition, this
law is $\mathcal L_\nu$.
\end{proof}

For completeness, the general RPC admits a directing measure
supported on the sphere of squared radius
$Q=\max\supp\nu$. To see this, write $(U_{ij})_{i<j}$ for
its overlap array and set $U_{ii}=Q$. The array is
ultrametric, and
\[
 U_{ij}=\int_0^Q\one_{\{U_{ij}>t\}}\,dt.
\]
For each $0\le t<Q$, the relation $U_{ij}>t$ is an
equivalence relation on the replica labels. Its indicator
matrix is positive semidefinite, so every finite matrix
$(U_{ij})$ is positive semidefinite as well. The directing
representation and the norm argument in the preceding proof
then give the claimed support restriction. When $Q=0$,
the directing measure is concentrated at the origin.

At nonzero field, the physical overlap array is nonnegative
and Gram, so Proposition~\zcref[noname]{prop:rpc-reconstruction}
applies directly. At zero field, define instead
\begin{equation}\label{eq:squared-gram}
 A_{ij}=R_{ij}^2.
\end{equation}
Every finite matrix $(A_{ij})$ is positive semidefinite by
the Schur product theorem. For distinct replica labels,
the directing representation also gives
\[
 A_{ij}
 =\langle v^i\otimes v^i,v^j\otimes v^j\rangle.
\]
Since the directing vectors have squared norm $Q$, their
tensor squares have squared norm $Q^2$.

The Ghirlanda--Guerra identities for absolute overlaps
imply the full identities for $A$, whose off-diagonal
marginal is $(q\mapsto q^2)_\#\mu$. Proposition~\zcref[noname]{prop:rpc-reconstruction}
therefore identifies its off-diagonal law with the RPC
law for this marginal. Taking square roots of the overlap
levels gives the RPC with marginal $\mu$, first for finite
cascades and then by weak approximation. Consequently,
$(|R_{ij}|)_{i<j}$ has law $\mathcal L_\mu$.
This argument does not require entrywise absolute values
to preserve positive semidefiniteness on arbitrary Gram
matrices.

\subsection{Reconstruction of the overlap signs}\label{subsec:sign-reconstruction}

At zero field, the signs of nonzero overlaps satisfy the
following relation, even when the Parisi measure has an atom
at zero.

\begin{lemma}\label{lem:signs-all-zero-field}
Every zero-field physical overlap array almost surely satisfies,
simultaneously for all distinct replica labels $i,j,k$,
\begin{equation}\label{eq:signs-all-zero-field}
 R_{ij}\ne0,\ R_{ik}\ne0
 \quad\Longrightarrow\quad
 \begin{gathered}
 R_{jk}\ne0,\\
 \operatorname{sgn}R_{jk}
 =\operatorname{sgn}R_{ij}\operatorname{sgn}R_{ik}.
 \end{gathered}
\end{equation}
\end{lemma}

\begin{proof}
If $Q=0$, all overlaps vanish and the implication is vacuous.
Suppose $Q>0$. The physical directing measure is supported
on the sphere of squared radius $Q$, so
\zeqref{eq:si-shell-signs} applies to the full physical array.
Since $\kappa$ is odd and strictly increasing,
\[
 \kappa(r)=0\quad\Longleftrightarrow\quad r=0,
 \qquad
 \operatorname{sgn}\kappa(r)=\operatorname{sgn}r.
\]
Thus \zeqref{eq:si-shell-signs} gives
\zeqref{eq:signs-all-zero-field} for each fixed triple.
Exchangeability and a countable intersection give the
assertion simultaneously for all distinct replica labels.
\end{proof}

\begin{lemma}\label{lem:signed-reconstruction}
Every zero-field physical overlap array has the same
off-diagonal law as
\[
 (\varepsilon_i\varepsilon_j\mathcal U_{ij})_{i<j},
\]
where $\mathcal U$ has law $\mathcal L_\mu$ and
$(\varepsilon_i)_{i\ge1}$ are independent fair signs,
independent of $\mathcal U$.
\end{lemma}

\begin{proof}
Put $U_{ij}=|R_{ij}|$ for $i<j$. The reconstruction applied
to $(R_{ij}^2)$ identifies the law of $U$ with
$\mathcal L_\mu$. We determine the conditional law of the
signed overlaps given $U$.

Fix $n$ and form the graph on $\{1,\ldots,n\}$ whose edges
are the pairs with $U_{ij}>0$.
Lemma~\zcref[noname]{lem:signs-all-zero-field} shows that
each connected component is complete and that the product
of the edge signs around every triangle is one.
In each component, choose its least label $a$, set $s_a=1$,
and put $s_i=\operatorname{sgn}R_{ai}$ for the remaining
labels. The triangle relation gives
\[
 R_{ij}=s_i s_jU_{ij}
 \qquad(1\le i<j\le n).
\]
For labels in different components, this equality holds
because both overlaps are zero.

At zero field, the even interactions make the Gibbs law
invariant under flipping all spins of any chosen replica.
Consequently, the finite-volume overlap law, and every
physical limit, is invariant under
\[
 R_{ij}\longmapsto\eta_i\eta_jR_{ij},
\]
where the $\eta_i$ are arbitrary deterministic signs equal
to one outside $\{1,\ldots,n\}$. These transformations leave
the entire absolute array $U$ unchanged. The conditional
law of the signed overlaps among the first $n$ replicas,
given $U$, is therefore invariant under this finite group.

The factorization above shows that the group acts
transitively on the possible sign patterns. Its invariant
law is thus uniform on that orbit. Independent fair vertex
signs produce exactly this law: each edge pattern has
$2^{c_n}$ representations, where $c_n$ is the number of
connected components, including isolated vertices.
The redundancy consists of a common flip of all vertex
signs within each component.

Thus, conditional on $U$, the overlaps among the first $n$
replicas have the law of
$(\varepsilon_i\varepsilon_jU_{ij})_{i<j\le n}$ with
independent fair signs. Since this holds for every $n$,
the finite-dimensional distributions give the asserted
infinite-array law.
\end{proof}

\begin{proof}[Proof of Theorem~{\zcref[noname]{thm:array}}]
By Proposition~\zcref[noname]{prop:marginal-sphere}, every
physical limit has $S_{12}\sim\mu$, and its directing
measure is supported on the sphere of squared radius $Q$.
Theorem~\zcref[noname]{thm:gg} gives the full
Ghirlanda--Guerra identities for $S$, conditional on the
signed overlaps.

At nonzero field, the limiting signed array is nonnegative,
so Proposition~\zcref[noname]{prop:rpc-reconstruction}
identifies its law with $\mathcal L_\mu$.
At zero field, reconstruction through the squared array
\zeqref{eq:squared-gram} identifies the absolute array,
and Lemma~\zcref[noname]{lem:signed-reconstruction}
determines the signed law. These conclusions also hold
when $Q=0$, since all off-diagonal overlaps then vanish.

Thus every subsequential limit of $\mathcal L_N$ equals
$\Law(\mathcal R^{\mu,h})$. Compactness of $\Omega$ gives
convergence of the full sequence. The microscopic diagonal
entries remain $R_{ii}=1$; the equality $\|v^i\|^2=Q$
concerns the directing vectors representing the
off-diagonal overlaps.

Finally, fix $n\ge2$, a bounded continuous function $F$ of
$(R_{ij})_{i<j\le n}$, and a continuous function
$\psi:[-1,1]\to\R$. Since the maps $R_{ij}\mapsto S_{ij}$
are continuous, full-array convergence gives convergence
of every term in the finite-volume version of
\zeqref{eq:full-gg}. The limiting terms satisfy that identity
with test function $\psi|_{[0,1]}$, because the limiting
$S$-overlaps are nonnegative. Hence the difference between
the two finite-volume sides tends to zero.
\end{proof}

\section{Quenched sampling laws and structural consequences}
\label{sec:quenched}
In this section we prove \zcref{thm:quenched} and
\zcref[range]{cor:geometry,cor:cavity-invariance} from the identified
overlap array. \zcref{subsec:quenched-measure} recovers the distribution
of the random quenched sampling measure by expressing its moments through
disjoint groups of replicas. \zcref{subsec:quenched-variance} derives the
one-third variance formula and characterizes when the quenched distribution
of $S_{12}$ converges to a deterministic measure.
\zcref{subsec:geometry-clusters} deduces the ultrametric geometry and
Poisson--Dirichlet cluster weights. Finally,
\zcref{subsec:cavity-invariance} combines the array identification with
the cavity representation from \zcref{sec:cavity} to prove invariance
under the one-site cavity map.

Throughout,
$\mu=\mu_{\xi,h}$ and $Q=\max\supp\mu$.

\subsection{The distribution of the quenched sampling measure}\label{subsec:quenched-measure}

The random measure $\mathscr Q_{\xi,h}^{\rm RPC}$ is the
conditional law of the overlap array given the cascade
environment. This law averages over independent replicas
sampled from the cascade weights and, at zero field,
independent fair replica signs. Its randomness therefore
comes entirely from the cascade environment.

\begin{proof}[Proof of {\zeqref{eq:quenched-convergence}}]
Let $F_1,\ldots,F_k$ be bounded continuous cylinder
functions on $\Omega$. Choose disjoint finite blocks of
replica labels $I_1,\ldots,I_k$, each large enough for the
corresponding function. Write $F_a(R|_{I_a})$ for $F_a$
evaluated after relabeling its replicas into $I_a$.
Conditional on the disorder $J$, the blocks are independent,
so
\begin{equation}\label{eq:quenched-block-moments}
 \E\prod_{a=1}^k\mathscr Q_N^J(F_a)
 =
 \E\left\langle
   \prod_{a=1}^k F_a(R|_{I_a})
 \right\rangle_{N,h}.
\end{equation}
The product on the right is a bounded continuous cylinder
function. Theorem~\zcref[noname]{thm:array} therefore gives
its limit under the RPC array law. Conditional on the
cascade environment, disjoint replica blocks are again
independent, including their signs at zero field. Hence
\[
 \E\prod_{a=1}^k\mathscr Q_N^J(F_a)
 \longrightarrow
 \E\prod_{a=1}^k\mathscr Q_{\xi,h}^{\rm RPC}(F_a).
\]

The space $\cP(\Omega)$ is compact and metrizable.
For each continuous cylinder function $F$, the evaluation
map $\nu\mapsto\nu(F)$
is continuous on $\cP(\Omega)$. These maps separate
probability measures, since continuous cylinder functions
are uniformly dense in $C(\Omega)$. Their finite products
and linear combinations thus form an algebra that contains
the constants and separates points of $\cP(\Omega)$.

By Stone--Weierstrass, this algebra is uniformly dense in
$C(\cP(\Omega))$. The mixed-moment convergence above
therefore extends by uniform approximation to every
continuous function on $\cP(\Omega)$, proving
\zeqref{eq:quenched-convergence}.
\end{proof}

\subsection{Non-self-averaging and the point-mass case}\label{subsec:quenched-variance}

\begin{proof}[Proof of the remaining assertions of Theorem~{\zcref[noname]{thm:quenched}}]
Recall that $P_N^J$ is the conditional law of $S_{12}$ given
the disorder $J$, as defined in
\zeqref{eq:quenched-array-definition}. At zero field,
$S_{12}=|R_{12}|$.

First work in the limiting array. For a bounded Borel
function $\psi:[-1,1]\to\R$, put
\[
 m=\int\psi\,d\mu,\qquad b=\int\psi^2\,d\mu.
\]
Apply Theorem~\zcref[noname]{thm:gg} with $n=2$ and
multiplier $\psi(S_{12})$ to obtain
\begin{equation}\label{eq:shared-pair-moment}
 \E[\psi(S_{12})\psi(S_{13})]
 =\tfrac12(m^2+b).
\end{equation}
Next apply the $n=3$ identity to the new overlap $S_{34}$,
with the same multiplier. Replica exchangeability and
\zeqref{eq:shared-pair-moment} give
\begin{equation}\label{eq:disjoint-pair-moment}
 \begin{aligned}
 3\E[\psi(S_{12})\psi(S_{34})]
 &=m^2+\E[\psi(S_{12})\psi(S_{13})]
       +\E[\psi(S_{12})\psi(S_{23})]\\
 &=2m^2+b.
 \end{aligned}
\end{equation}

At finite $N$, the pairs of replicas $(1,2)$ and $(3,4)$
are independent conditional on $J$. Hence
\[
 \E\bigl[\langle\psi(S_{12})\rangle_J^2\bigr]
 =
 \E\langle\psi(S_{12})\psi(S_{34})\rangle_J.
\]
Now suppose that the discontinuity set of $\psi$ has
$\mu$-measure zero. The single-overlap and product test functions 
are then almost surely continuous under the limiting array
law, since each displayed $S$-overlap has marginal $\mu$.
Theorem~\zcref[noname]{thm:array} therefore gives
\[
 \E\langle\psi(S_{12})\rangle_J\longrightarrow m,
 \qquad
 \E\bigl[\langle\psi(S_{12})\rangle_J^2\bigr]
 \longrightarrow\frac{2m^2+b}{3}.
\]
Subtracting the square of the mean proves
\zeqref{eq:one-third}. In particular, taking $\psi=\one_B$
for a $\mu$-continuity set $B$ gives limiting variance
$\mu(B)(1-\mu(B))/3$.

If $\mu=\delta_q$, marginal convergence yields
\[
 \E\int|s-q|\,P_N^J(ds)
 =\E\langle|S_{12}-q|\rangle_J
 \longrightarrow0.
\]
The integral bounds the bounded-Lipschitz distance between
$P_N^J$ and $\delta_q$. Markov's inequality therefore proves
weak convergence in probability to $\delta_q$.

If $\mu$ is not a point mass, take $\psi(s)=s$, for which
$\Var_\mu\psi>0$. Weak convergence in probability of $P_N^J$
to a deterministic probability measure would force
$P_N^J(\psi)$ to converge in probability to a constant.
Since these integrals are uniformly bounded, their variance
would tend to zero, contradicting \zeqref{eq:one-third}.
\end{proof}

\subsection{Geometry and threshold cluster weights}\label{subsec:geometry-clusters}

\begin{proof}[Proof of Corollary~{\zcref[noname]{cor:geometry}}]
Theorem~\zcref[noname]{thm:array} identifies the limiting
$S$-array with the RPC array, which is ultrametric.
This proves \zeqref{eq:geometry}. For each $\epsilon>0$,
the event
\[
 \{\min(S_{13},S_{23})-S_{12}\ge\epsilon\}
\]
is closed and has probability zero under the limiting law.
The closed-set bound for weak convergence therefore gives
\[
 \lim_{N\to\infty}\E\left\langle
 \one_{\{\min(S_{13},S_{23})-S_{12}\ge\epsilon\}}
 \right\rangle_{N,h}=0.
\]

At zero field, the directing vectors have squared norm $Q$,
and
\[
 d_\pm([v],[w])^2=2(Q-|v\cdot w|).
\]
Thus \zeqref{eq:geometry} is equivalent to
\[
 d_\pm([v^1],[v^2])
 \le
 \max\{d_\pm([v^1],[v^3]),d_\pm([v^2],[v^3])\}
\]
for independent directing samples, almost surely.
At nonzero field, the identity
$\|v-w\|^2=2(Q-v\cdot w)$ gives the same ultrametric
inequality for the ordinary Hilbert-space distance.
\end{proof}

\begin{proof}[Proof of Corollary~{\zcref[noname]{cor:pd}}]
Fix $0\le q<Q$ and put $a=\mu([0,q])$. Since $Q$ is
the largest point of $\supp\mu$, we have $0\le a<1$.
By ultrametricity, the relation
\[
 i\sim_q j
 \quad\Longleftrightarrow\quad
 i=j\ \text{or}\ S_{ij}>q
\]
is almost surely an equivalence relation.

We first relate its blocks to subsets of the directing
support. At zero field, use the image of the directing
measure under $v\mapsto[v]$, identifying $v$ with $-v$.
The overlap inequality holds for every triple in the
support: a violating triple would, by continuity, have
a neighborhood of positive product measure on which
the inequality still fails.

The equivalence classes on the support are the clusters
at threshold $q$. Each class contains a neighborhood of
each of its points, since the self-overlap is $Q>q$.
Every class therefore has positive directing mass.
There are at most countably many such classes, and
their masses sum to one. Conditional on the directing
measure, the strong law shows simultaneously for all
classes that their empirical frequencies converge to
their masses. In particular, every class is sampled.

Suppose the first $n$ replicas occupy $k$ blocks with
sizes $n_1,\ldots,n_k$. Choose the least label $i_j$ in
block $j$. The next replica joins this block exactly
when $S_{i_j,n+1}>q$. Applying
Theorem~\zcref[noname]{thm:gg} with
$\psi=\one_{(q,1]}$ gives
\begin{equation}\label{eq:pd-predictive}
 \Prob(\text{join block }j\mid\mathcal R_n)
 =\frac{1-a+n_j-1}{n}
 =\frac{n_j-a}{n}.
\end{equation}
The probability of starting a new block is consequently
\[
 1-\sum_{j=1}^k\frac{n_j-a}{n}=\frac{ka}{n}.
\]
These probabilities depend only on the block sizes, so
they also hold conditional on the sampled partition.

For $0<a<1$, the probability of a prescribed partition
of $\{1,\ldots,n\}$ into blocks of sizes $n_1,\ldots,n_k$
is
\begin{equation}\label{eq:pd-eppf}
 p(n_1,\ldots,n_k)
 =\frac{a^{k-1}(k-1)!}{(n-1)!}
       \prod_{j=1}^k(1-a)_{n_j-1}.
\end{equation}
Indeed, this expression equals one for the initial
one-label partition. Adding a label to block $j$
multiplies it by $(n_j-a)/n$, while creating a new block
multiplies it by $ka/n$, exactly as prescribed above.
These are the Ewens--Pitman partition probabilities
with parameters $(a,0)$. Their ranked block frequencies,
and hence the ranked directing cluster masses, have law
$\PD(a,0)$ \cite[Sections~1.3 and~2.3]{Petrov}.

If $a=0$, then $\Prob(S_{12}>q)=1$. Exchangeability and
a countable intersection show that all replica labels
belong to a single block, whose frequency is one.

The partitions at several thresholds are obtained from
the same RPC array. Their joint hierarchy therefore has
the law of the cascade restricted to those levels.
For a finite sample and finitely many thresholds $q_\ell$
with $\mu(\{q_\ell\})=0$, the comparison indicators
$\one_{\{S_{ij}>q_\ell\}}$ are almost surely continuous
under the limiting array law. Theorem~\zcref[noname]{thm:array}
therefore gives their joint convergence in distribution
to the indicators of the nested RPC partitions.
\end{proof}

\subsection{The fixed point of the physical cavity map}\label{subsec:cavity-invariance}

\begin{proof}[Proof of Corollary~{\zcref[noname]{cor:cavity-invariance}}]
Choose a physical predecessor $G_0$ and its successor
$\widehat G_0$ as in
Lemma~\zcref[noname]{cav:representation}, at the same
parameters $\xi,h$. By Theorem~\zcref[noname]{thm:array},
their annealed overlap-array laws both equal
$\Law(\mathcal R^{\mu,h})$. The disjoint-block argument
in \zeqref{eq:quenched-block-moments} therefore identifies
the distributions of their random conditional overlap-array
laws. Uniqueness of the directing representation
\cite[Theorem~1.2 and Proposition~1.3]{ArgChat} then gives
equality of their directing-measure laws up to isometry.
Here each measure is viewed on the closed linear span of
its support, and the isometries are linear. This common
law is also the law of the directing measure $G$ in the
statement.

The cavity representation obtains $\widehat G_0$ by
reweighting $G_0$ with
\[
 a(v)\cosh(h+Y_0(v)),
 \qquad
 a(v)=
 \exp\left\{\frac{\kappa(1)-\kappa(\|v\|^2)}2\right\}.
\]
Theorem~\zcref[noname]{thm:array} gives
$\|v\|^2=Q$ for $G_0$-almost every $v$.
Thus $a(v)$ is constant and cancels from the normalizing
factor, leaving
\[
 \widehat G_0=\mathcal C_{\xi,h}(G_0)
\]
with $\mathcal C_{\xi,h}$ defined by
\zeqref{eq:cavity-invariance}.

Linear isometries preserve the covariance
$\kappa(v\cdot w)$. Replacing a directing measure by an
isometric copy therefore gives an isometric copy of its
cavity transform in law. The conditional law of the
transformed isometry class depends only on the original
isometry class. Since $G$, $G_0$, and $\widehat G_0$
have the same law on these classes,
$\mathcal C_{\xi,h}(G)$ has the same law as $G$.

To express this transformation in the normalized covariance
convention of \cite[equations~(1.5)--(1.6)]{ArgChat}, put
\[
 L=\kappa(1),\qquad c(q)=\kappa(q)/L,\qquad
 \psi(x)=\log\cosh(h+\sqrt L\,x)-\log\cosh h.
\]
Then $c(1)=1$, $\psi(0)=0$, and
$\|\psi'\|_\infty\le\sqrt L$. The kernel
$c(v\cdot w)$ is positive semidefinite because $\kappa$
is a nonnegative linear combination of tensor inner
products.

Write $\ell_c=Y/\sqrt L$, and let $\E_z$ denote expectation
over an independent standard Gaussian variable. The weight
in that convention is
\[
 \E_z\exp\left\{
   \psi\left(\ell_c(v)+z\sqrt{1-c(\|v\|^2)}\right)
 \right\}
 =
 \frac{a(v)\cosh(h+Y(v))}{\cosh h}.
\]
Normalization cancels the constant $1/\cosh h$, giving
the physical cavity weight above. On the sphere of squared
radius $Q$, the factor $a(v)$ also cancels, yielding
exactly \zeqref{eq:cavity-invariance}.
\end{proof}

\section{Temperature chaos}
\label{sec:temperature-chaos}

We consider two zero-field Gibbs measures at different temperatures,
with the same realization of the interactions. Temperature chaos
means that the overlap between a replica from each measure converges
to zero. We prove this when the two Parisi measures have common
positive support points arbitrarily close to zero, and when either
Parisi measure is concentrated at zero. The support theorem for the
SK model then gives temperature chaos at every pair of distinct
nonnegative inverse temperatures. For a proof of temperature chaos in generic even $p$-spin models, see \cite{PanchenkoTemperature2016}. For predictions in the physics literature, see \cite{Kondor1989,FranzNeyNifle1995,RizzoCrisanti2003,Billoire2014}.

Assume \zeqref{eq:mixture}, and let $H_N$ be the Hamiltonian in
\zeqref{eq:hamiltonian}, with covariance $N\xi(R)$. Put
\[
 \kappa=\xi',\qquad L=\kappa(1).
\]
For $a\in\{1,2\}$ and $\beta_a\ge0$, define
\[
 G_{N,a}(\sigma)
 =\frac{\exp(\beta_a H_N(\sigma))}
        {\sum_{\tau\in\{-1,1\}^N}\exp(\beta_a H_N(\tau))}.
\]
Conditional on the common disorder, sample the two replica sequences
independently, with sequence $a$ sampled from $G_{N,a}$. We call $a$
the type of a replica. Throughout this section, $\langle\cdot\rangle$
denotes expectation over both sequences at fixed disorder, and
$\E\langle\cdot\rangle$ also averages over the disorder. For one
replica $\sigma^a$ of each type, write
\[
 R_N^{12}=\frac1N\sum_{i=1}^N\sigma_i^1\sigma_i^2.
\]
For $\beta_a>0$, let $\mu_a=\mu_{\beta_a^2\xi,0}$ be the
Parisi measure, and set
\[
 \alpha_a(q)=\mu_a([0,q]),\qquad Q_a=\max\supp\mu_a.
\]
The common-support hypothesis is
\begin{equation}\label{eq:tc-common-levels}
 0\in\overline{(\supp\mu_1\cap\supp\mu_2)\setminus\{0\}}.
\end{equation}
This condition requires a sequence of common positive support points
tending to zero; it does not require either support to be an interval.

\begin{theorem}\label{thm:tc-mixture}
Assume \zeqref{eq:mixture}, and let $\beta_1,\beta_2>0$ be distinct.
If \zeqref{eq:tc-common-levels} holds, then
\begin{equation}\label{eq:tc-chaos}
 \lim_{N\to\infty}\E\langle(R_N^{12})^2\rangle=0.
\end{equation}
The same conclusion holds without \zeqref{eq:tc-common-levels} if
$\mu_1=\delta_0$ or $\mu_2=\delta_0$.
\end{theorem}

\begin{corollary}\label{cor:tc-sk}
For the zero-field SK model with $\xi(q)=q^2/2$,
\zeqref{eq:tc-chaos} holds whenever $0\le\beta_1<\beta_2$.
\end{corollary}

We argue by contradiction. Suppose that a limiting overlap
$R^{12}$ between the two temperatures satisfies
$\Prob(|R^{12}|>s_0)>0$ for some $s_0>0$. Choose a common
support point $q\in\supp\mu_1\cap\supp\mu_2$ with
$0<q<s_0^2$. On the event $\{|R^{12}|>s_0\}$, we construct,
for every $n$ and every sufficiently small $\delta>0$, two
families of $n$ replicas, one at each temperature. After
multiplying the selected replicas by suitable signs, all
overlaps between distinct replicas in either family lie in
$[q-\delta,q+\delta]$, and every overlap between the two
families is at least $q-\delta$.

Let $B_{a,n}$ be the average of the signed spins at one fixed
site in the family at inverse temperature $\beta_a$, and let
$\mathscr R$ denote the full overlap array. The conditional
expectation of a product of two spins equals their overlap.
Thus the within-family bounds control
$\E[B_{a,n}^2\mid\mathscr R]$ from above, while the
between-family bound controls
$\E[B_{1,n}B_{2,n}\mid\mathscr R]$ from below. Subtracting gives
\[
 \E[(B_{1,n}-B_{2,n})^2\mid\mathscr R]
 \le \frac2n+4\delta.
\]
The conditional spin estimates also show that, as
$n\to\infty$ and $\delta\downarrow0$, the conditional moments
of $B_{a,n}$ converge to those of the scalar magnetization
$m_a(q)$ at inverse temperature $\beta_a$. The displayed
bound therefore forces $m_1(q)$ and $m_2(q)$ to have the same
moments. Since these variables take values in $[-1,1]$,
their distributions agree.

This equality determines the temperature. Indeed,
$\arctanh m_a(q)$ differs by a deterministically bounded
amount from a centered Gaussian variable of variance
$\beta_a^2\kappa(q)$. The bounded difference does not affect
the quadratic growth rate
\[
 \lim_{\lambda\to\infty}\frac1{\lambda^2}
 \log\E_{\rm sc}\cosh\bigl(\lambda\arctanh m_a(q)\bigr)
 =\frac{\beta_a^2\kappa(q)}2.
\]
Equal scalar laws give equal limits. Since $\kappa(q)>0$,
this implies $\beta_1=\beta_2$, a contradiction.

We first prove this last scalar assertion in
\zcref{subsec:tc-scalar-law}. In
\zcref{subsec:tc-joint-cavity}, we extend the conditional moment 
estimates to conditioning on overlaps from both temperatures and
derive the Gaussian integration-by-parts formula with both Gibbs
normalizing constants. The identities in
\zcref{subsec:tc-overlap-identities} then constrain overlaps between
the two types. In \zcref{subsec:tc-replica-selection}, these constraints
and the marginal Ghirlanda--Guerra identities produce the finite
families described above. Finally, \zcref{subsec:tc-completion}
compares their spin averages and proves
\zcref[nocomp]{thm:tc-mixture,cor:tc-sk}.

\subsection{Determining the temperature from a scalar law}
\label{subsec:tc-scalar-law}

Let $\Phi_a$ solve the Parisi PDE associated with $\mu_a$ and
covariance $\beta_a^2\xi$, with terminal value $\log(2\cosh x)$, and put
$u_a=\partial_x\Phi_a$. The scalar diffusion from
\zcref{sec:scalar}, started at zero, can be written as
\begin{equation}\label{eq:tc-scalar-diffusion}
 X^a_q=\beta_a B_{\kappa(q)}
       +\beta_a^2\int_0^q
          \alpha_a(s)u_a(s,X^a_s)\,\dd\kappa(s),
 \qquad 0\le q\le1,
\end{equation}
where $B$ is a standard Brownian motion. Write $\E_{\rm sc}$ for
expectation over this scalar diffusion, and define
\[
 m_a(q)=u_a(q,X^a_q),\qquad
 c_{a,k}(q)=\E_{\rm sc}m_a(q)^k\quad(k\ge1).
\]

\begin{lemma}\label{lem:tc-scalar-rigidity}
For fixed $\xi$ and $q\in(0,1]$, the law of $|m_a(q)|$ determines
$\beta_a$. The point $q$ need not belong to $\supp\mu_a$.
\end{lemma}

\begin{proof}
The drift accumulated in \zeqref{eq:tc-scalar-diffusion} through
time $q$ has absolute value at most $\beta_a^2\kappa(q)$.
The bound \zeqref{eq:sigmoid-shift}, applied with covariance
$\beta_a^2\xi$, gives
\[
 |\arctanh u_a(q,x)-x|
 \le2\beta_a^2\bigl(L-\kappa(q)\bigr).
\]
In particular, $|u_a(q,x)|<1$ for finite $x$. With
\[
 W_a=\arctanh m_a(q),\qquad
 C_a=\beta_a^2\bigl(2L-\kappa(q)\bigr),
\]
the two bounds imply the pathwise estimate
\[
 |W_a-\beta_a B_{\kappa(q)}|\le C_a.
\]
For real $x,d$, we have
$e^{-|d|}\cosh x\le\cosh(x+d)\le e^{|d|}\cosh x$.
Applying this inequality pathwise and then taking Gaussian
expectation yields
\[
 e^{-C_a|\lambda|+\lambda^2\beta_a^2\kappa(q)/2}
 \le\E_{\rm sc}\cosh(\lambda W_a)
 \le e^{C_a|\lambda|+\lambda^2\beta_a^2\kappa(q)/2}.
\]
The bounded error contributes only a term linear in $|\lambda|$
to the logarithm. Therefore
\begin{equation}\label{eq:tc-logit-scale}
 \lim_{\lambda\to\infty}\frac1{\lambda^2}
 \log\E_{\rm sc}\cosh\bigl(\lambda\arctanh m_a(q)\bigr)
 =\frac{\beta_a^2\kappa(q)}2.
\end{equation}
The expectation depends only on the law of $|m_a(q)|$.
Since $\kappa(q)>0$ and $\beta_a>0$, this limit determines
$\beta_a$.
\end{proof}

\subsection{Conditional spin laws for the two temperatures}
\label{subsec:tc-joint-cavity}

We need to retain overlaps involving both temperatures when applying
the conditional moment estimates. Index replicas by
$(a,i)\in\{1,2\}\times\mathbb N$; for a single label $j$, let
$a(j)$ denote its type. Fix a subsequential joint limit of the full
signed overlap array and the replica spins at one distinguished
site, denoted by $(R,\varepsilon)$. Such a limit exists by
compactness of the product space. Let $\mathscr R$ be the completed
sigma-field generated by every off-diagonal entry of $R$. Expectations
$\E$ for this limiting system include all its randomness.

Site permutation invariance relates the added-site spins to the
overlaps. For distinct replica labels $i,j$,
\begin{equation}\label{eq:tc-spin-site}
 \E[\varepsilon_i\varepsilon_j\mid\mathscr R]=R_{ij},
 \qquad \varepsilon_i^2=1.
\end{equation}
Indeed, against a bounded continuous function of finitely many
overlaps, the expectation of the spin product at the distinguished
site equals its average over all sites, which is the overlap.
Pass to the joint limit and then extend from continuous cylinder
functions by a monotone-class argument.

\begin{lemma}\label{lem:tc-mixed-source}
Fix $a\in\{1,2\}$, $k\ge1$, $q\in[0,1]$, and $e>0$.
There exists $\eta>0$ such that, for every deterministic tuple
of distinct type-$a$ labels $i_1,\ldots,i_k$ and signs
$s_1,\ldots,s_k\in\{-1,1\}$,
\begin{equation}\label{eq:tc-source-shell}
 \left|\E\left[\prod_{j=1}^k s_j\varepsilon_{i_j}
                     \,\middle|\,\mathscr R\right]
          -c_{a,k}(q)\right|\le e
\end{equation}
almost surely on the event
\[
 |s_js_\ell R_{i_ji_\ell}-q|<\eta
 \quad(1\le j<\ell\le k).
\]
The two-group conditional estimate
\zeqref{mark:source-conditional-bound} also holds given
$\mathscr R$, with both groups of type $a$ and scalar covariance
$\beta_a^2\xi$.
\end{lemma}

\begin{proof}
First take all signs equal to one. For the one-group construction
in \zcref{src:replica-groups}, use $k$ replicas, prescribe every
off-diagonal overlap to be $q$, and take
$P_0(\epsilon)=\prod_{j=1}^k\epsilon_j$.
Lemma~\zcref[noname]{src:replica-coefficient} identifies its scalar
coefficient as $c_{a,k}(q)$: conditional on the common scalar value
at $q$, the continuations have independent terminal spins with
mean $u_a(q,X_q^a)$. Apply
Lemma~\zcref[noname]{src:source-suppression} with covariance
$\beta_a^2\xi$ and accuracy $e$ to choose $\eta$.

Let $A_N$ be the empirical average of the spin monomial, let $T_N$ be a
nonnegative bounded continuous function of finitely many overlaps
of either type, and let $\chi_N\in[0,1]$ be a continuous cutoff
positive precisely on the prescribed open overlap constraint.
Site symmetry replaces the monomial at the distinguished site by
$A_N$. Splitting according to $|A_N-c_{a,k}(q)|\le e$ gives
\begin{equation}\label{eq:tc-mixed-suppression}
 \begin{aligned}
 &\left|\E\langle T_N\chi_N
       (P_0(\varepsilon)-c_{a,k}(q))\rangle\right|\le e\E\langle T_N\chi_N\rangle
       +2\|T_N\|_\infty C e^{-cN}.
 \end{aligned}
\end{equation}
For the second term, bound $T_N$ by its supremum and integrate out
every replica outside the selected tuple. Conditional on the disorder,
their Gibbs measures are probability measures. Bound the remaining
term using Lemma~\zcref[noname]{src:source-suppression}
at inverse temperature $\beta_a$. 

Pass to the joint limit. As in the proof of
\zeqref{mark:source-conditional-bound}, the inequalities against
all nonnegative continuous cylinder functions extend to all Borel
overlap events. Dividing on the sets where the cutoff is bounded
away from zero proves \zeqref{eq:tc-source-shell}.
The same argument with a two-group monomial proves the second
assertion. At zero field, flipping any selected replica preserves
the joint Gibbs law, even with the disorder fixed. This gives every
deterministic sign choice.

For each fixed $q$, take a common full-measure set over both types,
all degrees, deterministic tuples, sign choices, and dyadic
accuracies. There are only countably many choices. Consequently,
the tuple and signs in \zeqref{eq:tc-source-shell} may subsequently
be chosen measurably from $\mathscr R$, by partitioning according
to their countably many possible values.
\end{proof}

We next construct the cavity representation with a common Gaussian
field and two separate normalizing constants. In the statement below,
let $\nu_{a,q}$ denote the law of $m_a(q)$. For $0\le r\le Q_a$,
let $\nu_{a,r}^{Q_a}$ be the law of
\[
 \bigl(u_a(Q_a,X_{Q_a}^{a,1}),u_a(Q_a,X_{Q_a}^{a,2})\bigr),
\]
where the two scalar paths agree through time $r$ and then continue
independently conditional on their common value at $r$. For $r<0$,
reflect the second coordinate of $\nu_{a,|r|}^{Q_a}$.

\begin{lemma}\label{lem:tc-joint-cavity}
The chosen limit $(R,\varepsilon)$ has the following representation.
Let $G_1,G_2$ be random probability measures on a common separable
Hilbert space, with $G_a$ supported on $\{v:\|v\|^2=Q_a\}$.
Conditional on these measures, let $z$ be a centered Gaussian field
with covariance
\[
 \E[z(v)z(w)\mid G_1,G_2]=\kappa(v\cdot w).
\]
Define
\begin{equation}\label{eq:tc-cavity-weights}
 \begin{aligned}
 Y_a(v)&=\beta_a z(v),& M_a(v)&=\tanh Y_a(v),\\
 Z_a&=\int\cosh Y_a(v)\,G_a(dv),&
 \widehat G_a(dv)&=\frac{\cosh Y_a(v)}{Z_a}\,G_a(dv).
 \end{aligned}
\end{equation}
Conditional on $G_1,G_2,z$, sample all vectors independently, with
$v^i$ sampled from $\widehat G_{a(i)}$. Conditional on the measures,
the field, and all these vectors, the spins $\varepsilon_i$ are
independent signs with means $M_i=M_{a(i)}(v^i)$.
For $i\ne j$, their overlaps are $R_{ij}=v^i\cdot v^j$.
The measures $G_a$ and $\widehat G_a$ are mutually absolutely
continuous.

Conditional on $\mathscr R$, the law of $M_i$ is
$\nu_{a(i),Q_{a(i)}}$. If $i,j$ are distinct and of type $a$,
the conditional law of $(M_i,M_j)$ is $\nu_{a,R_{ij}}^{Q_a}$.
Applying $\arctanh$ gives the conditional laws of
$Y_i=Y_{a(i)}(v^i)$ and of a pair of fields of the same type.
Set $R_{ii}=Q_{a(i)}$ in formulas involving directing vectors.
Then, for all $i,j$, including the diagonal,
\begin{equation}\label{eq:tc-mark-site}
 \E[M_iM_j\mid\mathscr R]=R_{ij}.
\end{equation}

Let $\mathcal I$ be a finite set of old labels, let
$n_b=|\{j\in\mathcal I:a(j)=b\}|$, and let $F$ be a bounded
function of their overlaps. For $i\in\mathcal I$ and
$g\in C^1(\R^{\mathcal I})$ satisfying
\[
 |g(y)|+\sum_{j\in\mathcal I}|\partial_jg(y)|
 \le C\exp\left(c\sum_{j\in\mathcal I}|y_j|\right),
\]
we have
\begin{equation}\label{eq:tc-ibp}
 \begin{aligned}
 \E[FY_i g(\mathbf Y)]
 =\E\Bigl[F\Bigl(
 &\sum_{j\in\mathcal I}K_{ij}\partial_jg(\mathbf Y)
 +g(\mathbf Y)\sum_{j\in\mathcal I}K_{ij}M_j-g(\mathbf Y)\sum_{b=1}^2n_bK_{i,b+}M_{b+}
 \Bigr)\Bigr],
 \end{aligned}
\end{equation}
where
\[
 K_{ij}=\beta_{a(i)}\beta_{a(j)}\kappa(R_{ij}),
\]
and $b+$ denotes an additional independent type-$b$ sample in
its term. Coefficients of $g$ may depend measurably on the overlaps
among $\mathcal I$, provided the growth bound is uniform. These
coefficients are held fixed in the field derivatives.
\end{lemma}

\begin{proof}
We adapt the construction of Lemma~\zcref[noname]{cav:representation}
to both temperatures. Remove the distinguished site in each system
along the chosen joint subsequence. The joint overlap array of the
reduced systems is positive semidefinite and invariant under separate
finite permutations of the two replica sequences. The representation
theorem \cite[Theorem~2]{PanchenkoTemperature2016} gives two directing
measures on a common Hilbert space. This theorem uses these properties
of the array and does not require the genericity assumption used
elsewhere in that reference. The measures are supported on the unit
ball, since a positive-mass sufficiently small ball around a vector
of norm greater than one would produce off-diagonal overlaps greater
than one.

The reduced Hamiltonian has coefficients
$b_p(N/(N+1))^{(p-1)/2}$, and the unscaled one-site Gaussian field
has covariance $\kappa(NR/(N+1))$. The repeated-index estimate in
\zcref{app:cavity} bounds the variance of the omitted Gaussian
remainder by
\[
 v_N\le\sum_p b_p^2
       \min\left\{p,\frac{p(p-1)}{2(N+1)}\right\}
       =O_\xi(N^{-1}).
\]
At fixed disorder, the total variation distance between products of
probability measures is at most the sum of their marginal distances.
The estimate for removing an independent Gaussian remainder in that appendix therefore bounds
the disorder-averaged error for $n_1,n_2$ replicas by
\[
 \frac{\sqrt{v_N}}2\sum_{a=1}^2n_a\beta_a.
\]
It tends to zero for fixed replica counts.

Keep both cavity normalizing constants when passing to the limit.
Each is at least one. The auxiliary-replica approximation in
\zcref{app:cavity-joint}, using $J$ auxiliary replicas of each
type, approximates their inverse powers with error $O(J^{-1/2})$
for fixed $n_1,n_2$. To see that the same estimate controls the
product, use $x^{-n_a}\le1$ on $[1,\infty)$ and telescope the
two factors. Their individual errors are bounded by the Lipschitz
constant of $x\mapsto x^{-n_a}$ and the Gaussian second moments
of the auxiliary averages. We may thus take $N\to\infty$ at
fixed $J$, using continuity of Gaussian expectations with respect to the covariance matrix,
and then let $J\to\infty$.

The common field can be constructed from the map
\[
 v\longmapsto\bigoplus_{p\in2\mathbb N}
                  \sqrt p\,b_p\,v^{\otimes(p-1)},
\]
whose inner product at $v,w$ is $\kappa(v\cdot w)$.
Before reweighting, condition on the directing measures and reference
vectors sampled independently from them. The limiting unscaled field
at label $i$ is represented by
\[
 z(v^i)+\eta_i,\qquad
 \eta_i\sim N\bigl(0,L-\kappa(\|v^i\|^2)\bigr),
\]
where the $\eta_i$ are independent of $z$ and of one another
conditional on the reference vectors. They are independent by label,
even when two vectors coincide. For an added spin $e\in\{-1,1\}$
of type $a$, integration gives
\[
 \E_\eta e^{e\beta_a(z(v)+\eta)}
 =e^{\beta_a^2(L-\kappa(\|v\|^2))/2}e^{e\beta_a z(v)}.
\]
Summing over $e$ gives the cavity weight, while retaining $e$
gives a sign with conditional mean $\tanh(\beta_a z(v))$.
The limit construction includes these prescribed spins in the
numerator test functions, and hence realizes the chosen joint limit.

Each successor marginal is a limit of the original system at its
fixed temperature. Theorem~\zcref[noname]{thm:array} places its
directing measure on the sphere of squared radius $Q_a$. The
strictly positive cavity density transfers this support to $G_a$.
The radial factor in the last display is consequently constant on
each measure and cancels in its normalization. This proves
\zeqref{eq:tc-cavity-weights} and mutual absolute continuity.

To identify the conditional magnetization laws, use
Lemma~\zcref[noname]{lem:tc-mixed-source} in the proof of
Proposition~\zcref[noname]{mark:conditional}. We verify that
approximating a sampled vector by replicas in a small ball remains
valid when the conditioning includes both types. Fix a type $a$,
sample $v$ from $\widehat G_a$, and then sample $w$ from its
normalized restriction to a ball $B_v$ of positive mass. Write
$(G_a)_{B_v}$ for the normalized restriction of $G_a$ to this ball.
The density relative to $G_a(dv)(G_a)_{B_v}(dw)$ is
\[
 \frac{\cosh Y_a(v)\cosh Y_a(w)}
      {Z_a\,(G_a)_{B_v}(\cosh Y_a)}.
\]
Both denominators are at least one. Using the reference Gaussian
law with the directing measures and vectors fixed, put
$D=Y_a(w)-Y_a(v)$ and
\[
 \begin{aligned}
 c(z)&=\beta_a^2[\kappa(w\cdot z)-\kappa(v\cdot z)],\\
 \sigma^2(v,w)&=\beta_a^2[
   \kappa(\|w\|^2)+\kappa(\|v\|^2)-2\kappa(v\cdot w)].
 \end{aligned}
\]
If $\sigma^2(v,w)\le\chi$ throughout the ball,
Cauchy--Schwarz gives $|c(z)|\le\beta_a\sqrt{L\chi}$. 
Differentiating the logarithm of this density in the Gaussian
direction associated with $D$ gives 
\[
 c(v)M_a(v)+c(w)M_a(w)
 -\widehat G_a[cM_a]
 -\widehat G_a[cM_a\mid B_v],
\]
as in \zeqref{mark:ball-score}. Its absolute value is at most
$4\beta_a\sqrt{L\chi}$. Gaussian integration by parts therefore
bounds the expectation $A=\E_{\rm ball}D^2$ under this sampling law by
\[
 A\le\chi+4\beta_a\sqrt{L\chi}\sqrt A,
 \qquad A\le C_{\beta_a,L}\chi.
\]
This bound is independent of the ball mass. A bounded function of
the mixed overlaps multiplies the error by at most its supremum;
after that bound, all other normalized sampling measures integrate
to one. Since the covariance is continuous, $\chi\to0$ as the
ball radius tends to zero. The selection of nearby replicas and
the fixed-degree moment approximation in
\zcref{subsec:conditional-magnetizations} thus apply with
conditioning on $\mathscr R$. They give the stated one-point and
same-type pair laws. Since $u_a(Q_a,x)=\tanh x$ by
\zeqref{eq:scalar-terminal}, applying $\arctanh$ identifies the
field laws as well.

Conditional independence of the added spins and
\zeqref{eq:tc-spin-site} give \zeqref{eq:tc-mark-site} off the
diagonal. On the diagonal, the conditional scalar law and
$c_{a,2}(Q_a)=Q_a$ from \zeqref{eq:contact} give the same identity.

Finally, condition on $G_1,G_2$ and hold the reference vectors and
$F$ fixed during Gaussian differentiation. Their sampling density is
\[
 Z_1^{-n_1}Z_2^{-n_2}
 \prod_{j\in\mathcal I}\cosh Y_j.
\]
Differentiating $g$ gives the first sum in \zeqref{eq:tc-ibp}.
The numerator factors give one term $K_{ij}M_j$ for each old
label. Differentiating $Z_b^{-n_b}$ gives
\[
 -n_b\int\beta_{a(i)}\beta_b\kappa(v^i\cdot v)
                  M_b(v)\,\widehat G_b(dv),
\]
represented by the additional label $b+$. Finite Gaussian
projections justify the calculation as in
Lemma~\zcref[noname]{cav:gaussian}: variances are bounded,
normalizing constants are at least one, and Gaussian exponential
moments dominate every term. Approximation by simple functions of
the finite overlap array gives the assertion about coefficients
of $g$.
\end{proof}

The counts in \zeqref{eq:tc-ibp} include every label used by $F$
or by a coefficient of $g$. We apply this identity only with
finitely many such labels. The conditional magnetization laws,
in contrast, condition on the entire mixed array. Gaussian
differentiation is performed before this conditioning, under the
Gaussian field law conditional on $G_1,G_2$.

\subsection{Identities for overlaps with two reference replicas}
\label{subsec:tc-overlap-identities}

The marginal overlap laws supply two support properties used below.
The first ensures that replicas occur near every prescribed marginal
support level. The second supplies vectors with arbitrarily small
projections onto any fixed finite collection of vectors.

\begin{lemma}\label{lem:tc-richness}
Almost surely, for $\widehat G_a$-almost every $v$, every relatively
open interval $I\subseteq[0,1]$ with $\mu_a(I)>0$ satisfies
\[
 \widehat G_a\{w:|v\cdot w|\in I\}>0.
\]
If $0\in\supp\mu_a$ and $Q_a>0$, then almost surely, for every
$m\ge1$, every collection of Hilbert-space vectors
$z_1,\ldots,z_m$, and every $\epsilon>0$, there exists
$w\in\supp\widehat G_a$ such that
\[
 \max_{1\le i\le m}|z_i\cdot w|<\epsilon.
\]
\end{lemma}

\begin{proof}
Use the marginal identities \zeqref{eq:full-gg-conditional} and
the marginal ultrametricity from Theorem~\zcref[noname]{thm:array}.
Write $S_{ij}=|R_{ij}|$ for replicas of type $a$. If $A_n$ is
the event that $S_{12},\ldots,S_{1n}$ all lie outside $I$, then
\[
 \Prob(A_{n+1})
 =\left(1-\frac{\mu_a(I)}n\right)\Prob(A_n).
\]
For $\mu_a(I)>0$, the product of these factors tends to zero.
Conditional independent sampling shows that the set of $v$ for
which the displayed set has zero $\widehat G_a$-mass is null.
First intersect over a countable base of intervals, and then use
inclusion to obtain the assertion for every relatively open $I$.

For the second assertion, fix $0<\eta<Q_a$ and set
$d=\mu_a([0,\eta))>0$. Ultrametricity makes the relation
$S_{ij}\ge\eta$ an equivalence relation on distinct replicas,
with each label also equivalent to itself. Among $n$ old replicas,
let its equivalence classes be $C_1,\ldots,C_k$. By
\zeqref{eq:full-gg-conditional}, conditional on the unsigned
array of these type-$a$ replicas, the next replica joins a given
class $C$ with probability $(|C|-d)/n$. Consequently, it starts
a new class with probability $kd/n$. From any fixed configuration
with $k$ classes, the probability of no further class appearing
is bounded by
\[
 \prod_{n\ge n_0}\left(1-\frac{kd}{n}\right)=0.
\]
A countable union over possible last creation times proves that
infinitely many classes appear.

Apply this conclusion at a sequence of thresholds tending to zero.
For every sufficiently large integer $M$, we can choose support
points $w_1,\ldots,w_M$ whose pairwise absolute overlaps are less
than $1/M$. Their Gram matrix has diagonal $Q_a$ and absolute
off-diagonal row sums less than one, so its operator norm is at
most $Q_a+1$. Hence
\[
 \sum_{j=1}^M\sum_{i=1}^m(z_i\cdot w_j)^2
 \le(Q_a+1)\sum_{i=1}^m\|z_i\|^2.
\]
For at least one $j$, the inner sum is at most the right side
divided by $M$. Taking $M$ large gives the required $w_j$.
Once the support points exist, this deterministic argument applies
to every finite collection $z_1,\ldots,z_m$.

All probability-one assertions used here concern a single marginal
array and therefore also hold under the joint law of the two arrays.
In particular, we have only conditioned the marginal
Ghirlanda--Guerra identities on overlaps of that marginal.
\end{proof}

Fix two reference labels $1,2$ of type $a$, and set $Q=Q_a$.
We first work on $\{R_{12}\ge0\}$ and write $r=R_{12}$;
overlap multipliers in the next identity are taken to vanish
outside this event. Given a continuous function
$h:[-1,1]^2\to\R$, the pair
$(R_{1j},R_{2j})$ records the overlaps of label $j$ with these
references. Write
\[
 h_j=h(R_{1j},R_{2j}).
\]
The marginal law determines the following function of $r$:
\begin{equation}\label{eq:tc-row-calibration}
 c_h(r)=h(Q,r)+h(r,Q)
       -2\E_a[h(R_{13},R_{23})\mid R_{12}=r],
\end{equation}
where all three labels in $\E_a$ have type $a$. Conditional
expectations are understood up to null sets for the marginal law
of $R_{12}$. The first two terms are the values at the reference
labels; the last term is twice the expected value at an additional
replica of their type. The next lemma shows that this same function
appears when overlaps with both types are included.

\begin{lemma}\label{lem:tc-row}
Let $\mathcal I$ contain the reference labels and every label used
by a bounded test function $F$ of finitely many overlaps. With the counts $n_b$ and
additional labels $b+$ from Lemma~\zcref[noname]{lem:tc-joint-cavity},
\begin{equation}\label{eq:tc-row-identity}
 \sum_{b=1}^2n_b\beta_b\E[Fh_{b+}]
 =\sum_{j\in\mathcal I}\beta_{a(j)}\E[Fh_j]
      -\beta_a\E[Fc_h(r)]
\end{equation}
for each of the test functions
\begin{equation}\label{eq:tc-row-tests}
 \begin{gathered}
 h(x,y)=|u\kappa(x)+v\kappa(y)|,
       \qquad u,v\in\R,\\
 h(x,y)=(\kappa(x)-\kappa(y))x,
       \qquad h(x,y)=x\kappa(x).
 \end{gathered}
\end{equation}
The identity is preserved under finite linear combinations and
uniform limits on $[-1,1]^2$. On $r\ge\epsilon>0$, the
coefficients of these combinations may depend continuously on $r$,
and the limits may be taken uniformly in $r$.
\end{lemma}

\begin{proof}
For the first family, put
\[
 Z=uY_1+vY_2,\qquad
 L_r(\theta)=\E[e^{\theta Z}\mid\mathscr R].
\]
The conditional field law in
Lemma~\zcref[noname]{lem:tc-joint-cavity} makes $L_r(\theta)$ a
deterministic function of $r$. Define the tilted expectation
\[
 \E_\theta H=\E\left[H\frac{e^{\theta Z}}{L_r(\theta)}\right].
\]
This reweighting preserves the law of the entire overlap array,
since the density has conditional expectation one given
$\mathscr R$. Set
\[
 c_j=u\kappa(R_{1j})+v\kappa(R_{2j}),\qquad
 V(r)=\beta_a^2\bigl((u^2+v^2)\kappa(Q)+2uv\kappa(r)\bigr).
\]
Apply \zeqref{eq:tc-ibp} to each component of $Z$, with test function 
$e^{\theta Z}$ and multiplier $F/L_r(\theta)$. Combining the
two identities and dividing by $\beta_a$ gives
\begin{equation}\label{eq:tc-tilted-row}
 \begin{aligned}
 \sum_{b=1}^2n_b\beta_b\E_\theta[Fc_{b+}M_{b+}]
 &=\sum_{j\in\mathcal I}\beta_{a(j)}
              \E_\theta[Fc_jM_j]-\E[Fa_\theta(r)],\\
 a_\theta(r)
 &=\frac{L'_r(\theta)/L_r(\theta)-\theta V(r)}{\beta_a}.
 \end{aligned}
\end{equation}
Here $\theta V(r)$ is the derivative contribution from the
exponential test, and $L'_r/L_r$ is its conditional mean of $Z$.
The conditional pair law is symmetric under simultaneous reflection,
so Jensen's inequality gives $L_r(\theta)\ge1$. Thus the
overlap-dependent multiplier is bounded whenever $F$ is bounded.

We next show that the large positive tilt replaces $c_jM_j$ by
$|c_j|$. For any displayed label, including an additional label,
define
\[
 D_j=Y_j-\theta\beta_a\beta_{a(j)}c_j.
\]
Include $j$ among the old labels for this auxiliary calculation,
and apply \zeqref{eq:tc-ibp} with the test function 
$D_je^{\theta Z}/L_r(\theta)$. Differentiating the exponential
produces exactly the shift subtracted in $D_j$, leaving
\[
 \E_\theta D_j^2=K_{jj}+\E_\theta[D_jB_j],\qquad
 B_j=\sum_{\ell\in\mathcal I'}K_{j\ell}M_\ell
       -\sum_{b=1}^2n'_bK_{j,b+}M_{b+},
\]
where $\mathcal I'$ is the enlarged set of old labels and
$n'_b$ are its counts. Since $|M_\ell|\le1$ and
$|\kappa(R_{j\ell})|\le L$,
\[
 |B_j|\le2\beta_{a(j)}L\sum_{b=1}^2n'_b\beta_b.
\]
Cauchy--Schwarz therefore bounds $\E_\theta D_j^2$ uniformly
in $\theta$. For $|c_j|>\delta$, the shift in $Y_j$ has
magnitude at least $\theta\beta_a\beta_{a(j)}\delta$, so this
bound and $M_j=\tanh Y_j$ imply convergence to
$\operatorname{sgn}(c_j)$ in tilted probability. On
$|c_j|\le\delta$, the error $|c_jM_j-|c_j||$ is at most
$2\delta$. First let $\theta\to\infty$ and then
$\delta\downarrow0$ to obtain
\[
 \E_\theta|c_jM_j-|c_j||\longrightarrow0.
\]

To determine the remaining term in \zeqref{eq:tc-tilted-row},
apply that identity with only the two reference labels old and
with arbitrary bounded functions of $r$ as multipliers. Then
\[
 a_\theta(r)=\beta_a\E_\theta[
 c_1M_1+c_2M_2-2c_3M_3\mid R_{12}=r],
\]
where label $3$ has type $a$. The right side is bounded in
absolute value by $4\beta_a L(|u|+|v|)$. The preceding convergence
and preservation of the overlap law show that it converges in
$L^1$ of the marginal law of $r$ to $\beta_a c_h(r)$.
Consequently, we may take $\theta\to\infty$ in
\zeqref{eq:tc-tilted-row} with the bounded multiplier $F$.
This proves \zeqref{eq:tc-row-identity} for the absolute-value
test functions.

For $h(x,y)=(\kappa(x)-\kappa(y))x$, subtract the
integration-by-parts identities with left sides $Y_1M_1$ and
$Y_2M_1$. The conditional expectation of $(Y_1-Y_2)M_1$ is
a function of $r$, by the same-type pair law. The derivative term
is
\[
 \beta_a^2\bigl(\kappa(Q)-\kappa(r)\bigr)(1-Q),
\]
because \zeqref{eq:tc-mark-site} gives
$\E[M_1^2\mid\mathscr R]=Q$. The terms associated with an
old or additional label $j$ become
\[
 \beta_a\beta_{a(j)}
 \bigl(\kappa(R_{1j})-\kappa(R_{2j})\bigr)R_{1j},
\]
using \zeqref{eq:tc-mark-site} with conditioning that includes
the additional label when necessary. Thus the only term not already
of the form in \zeqref{eq:tc-row-identity} depends on $r$ alone.
Applying the same formula with just the reference pair old identifies
that term as $\beta_a c_h(r)$.

For the last test function, write $g(x)=x\kappa(x)$ and apply
\zeqref{eq:tc-ibp} with field test function $M_1$. Its conditional left
side is constant by the one-point law. With only label $1$ old,
\[
 \E[Y_1M_1]
 =\beta_a^2\left(
   \kappa(Q)(1-Q)+g(Q)-\int g\,d\mu_a\right).
\]
The integral uses the absolute-overlap law because $g$ is even.
For arbitrary old counts, the remainder is therefore
$\beta_a(g(Q)-\int g\,d\mu_a)$. With label $2$ included,
the marginal Ghirlanda--Guerra identity gives
\[
 2\E_a[g(R_{13})\mid R_{12}=r]
 =\int g\,d\mu_a+g(r).
\]
Hence the remainder is again $\beta_a c_h(r)$ for
$h(x,y)=g(x)$.

Linearity proves the assertion for finite combinations. Moreover,
\[
 |c_h(r)-c_{\widetilde h}(r)|
 \le4\|h-\widetilde h\|_\infty,
\]
so bounded convergence proves the uniform-limit assertion.
Approximating continuous coefficient functions on the compact
interval $[\epsilon,Q]$ gives the stated dependence on $r$.
The identities also apply on $\{r=0\}$: the scalar pair has
independent branches there, and each preceding calculation permits
the multiplier $\one_{\{r=0\}}$.
\end{proof}

The next consequence transfers an equality for marginal overlap
pairs to overlaps with a replica at the other temperature.

\begin{lemma}\label{lem:tc-zero-row}
Let $h$ be a test function allowed by Lemma~\zcref[noname]{lem:tc-row}.
Suppose that, for almost every sampled reference pair with $r\ge0$,
\[
 h(Q,r)=h(r,Q)=0,\qquad
 h(R_{13},R_{23})=0
\]
for almost every additional type-$a$ replica. If $h(0,0)=0$,
then, almost surely for that pair,
\[
 h(v^1\cdot w,v^2\cdot w)=0
 \quad\text{for every }w\in\supp\widehat G_{3-a}.
\]
The assertion also holds when all hypotheses and the conclusion
are restricted to a Borel set $D\subseteq[0,Q]$ of reference
overlaps $r$.
\end{lemma}

\begin{proof}
Put $b=3-a$. If $Q_b=0$, the support is the origin and the
claim follows from $h(0,0)=0$. Otherwise,
\zeqref{eq:zero-in-support} and
Lemma~\zcref[noname]{lem:tc-richness} give support vectors with
arbitrarily small projections onto both references.

Let $E=\{r\in D\}$, taking $D=[0,Q]$ when there is no further
restriction. The hypotheses give $c_h(r)=0$ on $E$. Include one
type-$b$ label $b,1$ among the old labels and take
$F=\one_Eh_{b,1}$ in \zeqref{eq:tc-row-identity}.
Every type-$a$ term vanishes, so, after dividing by $\beta_b>0$,
\[
 \E[\one_Eh_{b,1}h_{b,2}]=\E[\one_Eh_{b,1}^2].
\]
Exchangeability of the two type-$b$ labels gives
$\E[\one_E(h_{b,1}-h_{b,2})^2]=0$.
Conditional on the directing measures,
the field, and the reference pair, the labels $b,1$ and $b,2$
are independent samples. Thus their common function $h_{b,j}$
has zero conditional variance on $E$ and is constant almost
everywhere under $\widehat G_b$. Continuity extends that constant to its
support. Evaluating along vectors whose two projections tend to
zero identifies the constant as $h(0,0)=0$.
\end{proof}

\begin{lemma}\label{lem:tc-unequal-minimum}
Let $1,2$ be a same-type reference pair with $R_{12}=r>0$.
For a replica $j$ of the other type, put
$x=R_{1j}$ and $y=R_{2j}$. Almost surely, either $x=y=0$
or $xy>0$, and
\begin{equation}\label{eq:tc-unequal-minimum}
 |x|\ne|y|
 \quad\Longrightarrow\quad
 \min(|x|,|y|)=r.
\end{equation}
These assertions hold simultaneously for all sampled triples.
\end{lemma}

\begin{proof}
Set
\[
 X=\kappa(x),\qquad Y=\kappa(y),\qquad
 k=\kappa(r),\qquad K=\kappa(Q_a).
\]
For a third replica of type $a$, marginal ultrametricity and
the sign relation \zeqref{eq:signs-all-zero-field} imply that
either $x=y$, or $x,y$ have the same sign and
$\min(|x|,|y|)=r$. In particular, the following three functions
are zero for those overlap pairs and for the pairs $(Q_a,r)$
and $(r,Q_a)$:
\[
 \begin{gathered}
 |X|+|Y|-|X+Y|,\\
 |kX+KY|+|kX-KY|-2K|Y|,\\
 |kY+KX|+|kY-KX|-2K|X|.
 \end{gathered}
\]
Each displayed function is a finite linear combination of the
absolute-value test functions in \zeqref{eq:tc-row-tests}. Apply
Lemma~\zcref[noname]{lem:tc-zero-row} to each. The first excludes
opposite signs, while the other two, using
$|s+t|+|s-t|=2\max(|s|,|t|)$, give
\begin{equation}\label{eq:tc-cone}
 (X,Y)=(0,0)\quad\text{or}\quad
 \begin{cases}
 XY>0,\\
 k|X|\le K|Y|,\\
 k|Y|\le K|X|.
 \end{cases}
\end{equation}
The last two inequalities also exclude exactly one of $X,Y$
being zero.

We need one more test function to determine the smaller nonzero overlap.
Define
\[
 d=\max\left\{\min(|X|,|Y|),
                    \frac{k}{K}\max(|X|,|Y|)\right\},
 \qquad
 H_1(X,Y)=(X-Y)\kappa^{-1}(kX/d)
\]
away from $(0,0)$, and set $H_1(0,0)=0$. The inverse is taken
on $[-K,K]$, where $\kappa$ maps $[-Q_a,Q_a]$ bijectively
onto this interval. Its argument belongs to $[-K,K]$ because
$d\ge k|X|/K$. Moreover, $H_1$ is continuous, even under
$(X,Y)\mapsto(-X,-Y)$, positively homogeneous of degree one,
and bounded in absolute value by $Q_a(|X|+|Y|)$.

Every continuous even function on $\R^2$ that is positively
homogeneous of degree one can be approximated uniformly on bounded
sets by finite linear combinations of absolute linear forms.
To verify this, interpolate its values by homogeneous linear
functions on the sectors of a fine antipodally symmetric partition
of the plane. The interpolants converge uniformly on the unit
circle. Continuity across a boundary line makes the gradient jump
normal to that line; evenness relates the jumps on its two rays.
Subtracting a suitable multiple of the absolute value of the normal
linear form removes both jumps. After doing this for every line,
the remainder is globally linear and even, and hence zero.
For the family $H_1$ with $r\in[\epsilon,Q_a]$, a common
partition gives uniform approximation with coefficients continuous
in $r$. Thus $H_1(\kappa(x),\kappa(y))$ is an allowed test function.

Compare it with
\[
 H_2(x,y)=(\kappa(x)-\kappa(y))x.
\]
When $x=y$, both functions vanish. For every other marginal
overlap pair described above, $d=k$, so
$H_1(\kappa(x),\kappa(y))=H_2(x,y)$. This also holds at the
two reference labels. Lemma~\zcref[noname]{lem:tc-zero-row},
applied to their difference on $r\ge\epsilon$, therefore gives
the same equality for a replica of the other type. On
\zeqref{eq:tc-cone}, we have $d=\min(|X|,|Y|)$. If $x\ne y$,
cancel $X-Y$ in the equality to obtain
\[
 x=\kappa^{-1}\left(\frac{kX}{\min(|X|,|Y|)}\right).
\]
Applying $\kappa$ and canceling $X\ne0$ gives
$\min(|X|,|Y|)=k$. Strict monotonicity and oddness of
$\kappa$ prove \zeqref{eq:tc-unequal-minimum}.
Let $\epsilon\downarrow0$ through a countable sequence.

For negative reference overlaps, flip the second reference and its
spin. Every deterministic sign change preserves the zero-field law.
Intersecting the resulting full-measure sets over the countably
many triples permits this orientation to be chosen from the signed
array.
\end{proof}

\subsection{Selecting replicas near a common support point}
\label{subsec:tc-replica-selection}

We now construct the finite families needed to compare spin averages.
First, the marginal Ghirlanda--Guerra identities supply arbitrarily
many replicas whose mutual overlaps lie in a prescribed interval
of positive Parisi mass. The interval may contain an isolated
support point.

\begin{lemma}\label{lem:tc-near-branches}
Fix a type $a$, a sampled reference vector $v$, and deterministic
thresholds $0<\ell<u\le1$ with $\mu_a([\ell,u))>0$.
Almost surely there are infinitely many distinct labels of type
$a$ whose vectors, after multiplication by the signs of their
overlaps with $v$, satisfy
\begin{equation}\label{eq:tc-near-branches}
 \ell\le v\cdot v_i<u,\qquad
 \ell\le v_i\cdot v_j<u\quad(i\ne j).
\end{equation}
Distinct labels are allowed to sample the same directing vector.
\end{lemma}

\begin{proof}
For $t>0$, put $a_t=\mu_a([0,t))$ and declare two labels
equivalent at threshold $t$ if they coincide or their absolute
overlap is at least $t$. Marginal ultrametricity makes this an
equivalence relation. For $t>Q_a$, every class is a singleton
and $a_t=1$. In all cases, a represented class $C$ among $n$
old type-$a$ labels receives the next label with probability
$(|C|-a_t)/n$, conditional on the old unsigned marginal array.

Let $C$ be the class of the reference at threshold $\ell$,
and let $C_1,\ldots,C_k$ be its represented classes at threshold
$u$. The probability that the next replica enters $C$ but none
of $C_1,\ldots,C_k$ is
\[
 \frac{|C|-a_\ell}{n}
 -\sum_{j=1}^k\frac{|C_j|-a_u}{n}
 =\frac{ka_u-a_\ell}{n}.
\]
This is precisely the probability of creating a new threshold-$u$
class within $C$. Since
$a_u-a_\ell=\mu_a([\ell,u))>0$, an infinite-product argument
as in Lemma~\zcref[noname]{lem:tc-richness} excludes remaining
forever at any fixed $k$. Infinitely many such classes therefore
appear.

Discard the threshold-$u$ class containing the reference and take
one representative from each remaining class in $C$.
Their absolute overlaps with the reference and with one another
belong to $[\ell,u)$. Since $\ell>0$, the sign relation
\zeqref{eq:signs-all-zero-field} makes every one of these overlaps
positive after orientation by $v$. This proves
\zeqref{eq:tc-near-branches}. The construction uses a marginal
probability-one assertion, which remains valid under the joint law.
\end{proof}

\begin{lemma}\label{lem:tc-descendants}
Assume \zeqref{eq:tc-common-levels}. Let $v,w$ be a fixed pair
of sampled vectors of types $1,2$, respectively, and suppose that
\[
 A=\{|v\cdot w|>s_0\},\qquad \Prob(A)>0,
\]
for some deterministic $s_0>0$. There is a deterministic common
support point $q>0$ such that, almost surely on $A$, for every
$n\ge1$ and every sufficiently small $\delta>0$, one can choose
$n$ distinct labels of each type and a sign for each label,
measurably from $\mathscr R$, with the following properties:
all oriented overlaps between distinct labels within either type
belong to $(q-\delta,q+\delta)$, and every oriented overlap between
the two types is at least $q-\delta$.
\end{lemma}

\begin{proof}
Orient $w$ so that $s=v\cdot w>s_0$. Since overlaps have
absolute value at most one, we may take $s_0<1$.
By \zeqref{eq:tc-common-levels}, choose a deterministic
\[
 q\in\supp\mu_1\cap\supp\mu_2,
 \qquad 0<q<s_0^2.
\]
The inequality $q<s_0^2$ will ensure that only finitely many
selected vectors can retain overlap $s$ with the other reference.
For sufficiently small $\delta>0$, choose rational thresholds
\[
 0<q-\delta<\ell<q<u<q+\delta<s_0^2.
\]
Both measures assign positive mass to $[\ell,u)$ because $q$
is a support point in its interior.

Apply Lemma~\zcref[noname]{lem:tc-near-branches} around $v$.
This gives infinitely many oriented type-$1$ representatives $v_i$
with $t_i=v\cdot v_i\in[\ell,u)$ and mutual overlaps in that
interval. For the reference pair $(v,v_i)$, the two overlaps with
$w$ are $s$ and $v_i\cdot w$. The sign assertion and
\zeqref{eq:tc-unequal-minimum}, together with $s>u>t_i$, give
\[
 v_i\cdot w\in\{t_i,s\}.
\]
Only finitely many representatives can have overlap $s$ with $w$.
Indeed, if $M$ representatives do, their average $V$ satisfies
\[
 V\cdot w=s,\qquad
 \|V\|^2\le u+\frac{Q_1}{M}.
\]
Cauchy--Schwarz then gives $s^2\le Q_2(u+Q_1/M)$, and hence
\begin{equation}\label{eq:tc-high-branches}
 M\le\frac{Q_1Q_2}{s^2-Q_2u}
   \le\frac1{s_0^2-u}.
\end{equation}
The denominator is positive by our choice of $u$.
Discard these finitely many representatives and retain $n$ with
$v_i\cdot w=t_i$. Repeat the construction around the oriented
vector $w$ for type $2$, obtaining $w_j$ with
\[
 u_j=w\cdot w_j\in[\ell,u),\qquad v\cdot w_j=u_j.
\]

It remains to check overlaps between the two selected families.
Put $c=v_i\cdot w_j$. Applying the sign conclusion of
Lemma~\zcref[noname]{lem:tc-unequal-minimum} to the same-type
pair $(w,w_j)$ and the other-type vector $v_i$ gives $c>0$.
If $c<\min(t_i,u_j)$, its overlaps $t_i,c$ with this pair
are unequal, so \zeqref{eq:tc-unequal-minimum} gives
$\min(t_i,c)=u_j$, a contradiction. Therefore
\[
 v_i\cdot w_j\ge\min(t_i,u_j)\ge\ell>q-\delta.
\]

For measurability, establish the preceding marginal class counts
for every sampled reference and every rational threshold pair on
one common full-measure set, together with the overlap identities for deterministic triples and their versions after sign changes. Scan the type-$1$ labels in
order, keeping the first representative of each new threshold-$u$
class inside the reference's threshold-$\ell$ class and excluding
the reference's own threshold-$u$ class. Orient by the reference
and discard representatives whose overlap with $w$ lies outside
$[\ell,u)$. There are infinitely many candidates, and
\zeqref{eq:tc-high-branches} bounds the number discarded. Retain
the first $n$. Repeat for type $2$, with orientation relative to
the already oriented $w$. Every choice uses only the signed overlap
array. Rational threshold pairs and all finite selections form
countable families, so the same full-measure set suffices.
\end{proof}

\subsection{Completion of the proof}
\label{subsec:tc-completion}

The selected families have the same approximate overlap level
within each temperature and a lower bound at that level between
temperatures. By \zeqref{eq:tc-spin-site}, these are upper bounds
on the second moments of the separate spin averages and a lower
bound on their mixed second moment. Their difference is therefore
small in mean square. We now combine this estimate with the
conditional moment identification.

\begin{proof}[Proof of \zcref{thm:tc-mixture}]
First assume \zeqref{eq:tc-common-levels}. Suppose that a
subsequential cross overlap is nonzero with positive probability.
Choose an event $A$ and a deterministic common support point $q$
as in Lemma~\zcref[noname]{lem:tc-descendants}.
For its selections at width $\delta>0$, write $i_{a,j}$ for
the $j$th selected label within type $a$, and let $s_{a,j}$
be its sign. Extend these choices measurably off $A$, and define
\[
 B_{a,n}=\frac1n\sum_{j=1}^n
            s_{a,j}\varepsilon_{(a,i_{a,j})},
 \qquad |B_{a,n}|\le1.
\]

Fix a degree $k\ge1$ and accuracy $e>0$. Choose $\delta$
smaller than the two admissible widths in
Lemma~\zcref[noname]{lem:tc-mixed-source}. In the expansion of
$B_{a,n}^k$, every ordered tuple of distinct indices has
conditional expectation within $e$ of $c_{a,k}(q)$.
The fraction of tuples with a repeated index is at most
$k(k-1)/(2n)$, by a union bound over pairs of positions.
Each such term differs from $c_{a,k}(q)$ by at most two.
Consequently, on $A$,
\begin{equation}\label{eq:tc-average-moment}
 \left|\E[B_{a,n}^k\mid\mathscr R]-c_{a,k}(q)\right|
 \le e+\frac{k(k-1)}n.
\end{equation}
The argument partitioning over the possible selected labels and signs at the end of the proof of
Lemma~\zcref[noname]{lem:tc-mixed-source} permits the selected
labels and signs to depend on $\mathscr R$.

For the second moments, the $n$ diagonal terms in each average
contribute $1/n$ because $\varepsilon_i^2=1$.
Each within-type off-diagonal term is at most $q+\delta$,
and each between-type term is at least $q-\delta$.
Thus \zeqref{eq:tc-spin-site} gives, on $A$,
\begin{equation}\label{eq:tc-average-distance}
 \begin{aligned}
 \E[(B_{1,n}-B_{2,n})^2\mid\mathscr R]
 &\le\frac2n+2\left(1-\frac1n\right)(q+\delta)
            -2(q-\delta)\\
 &\le\frac2n+4\delta.
 \end{aligned}
\end{equation}
Since $|x^k-y^k|\le k|x-y|$ on $[-1,1]$, conditional
Cauchy--Schwarz and \zeqref{eq:tc-average-moment} imply
\[
 |c_{1,k}(q)-c_{2,k}(q)|
 \le2e+\frac{2k(k-1)}n+k\sqrt{2/n+4\delta}.
\]
The left side is deterministic, and the inequality holds on the
event $A$ of positive probability. All thermodynamic limits have
already been taken. For fixed $k,e$, choose $\delta<e^2$ as
well as below the required overlap widths, then let $n\to\infty$.
Finally let $e\downarrow0$. We obtain
$c_{1,k}(q)=c_{2,k}(q)$ for every $k$.

In particular, all moments of $m_1(q)^2$ and $m_2(q)^2$ agree.
Polynomial approximation on $[0,1]$ identifies their laws and
therefore the laws of $|m_1(q)|$ and $|m_2(q)|$.
Lemma~\zcref[noname]{lem:tc-scalar-rigidity} forces
$\beta_1=\beta_2$, a contradiction. The unbounded test function involving
$\arctanh$ is used only after these scalar laws have been
identified; their magnetizations have absolute value strictly
less than one.

Every subsequential cross overlap consequently vanishes almost
surely. Compactness of the array space and boundedness of the
squared overlap give \zeqref{eq:tc-chaos} along the full sequence.

Finally, suppose $\mu_1=\delta_0$. At fixed disorder define the
spin correlation matrices
\[
 C_a(i,j)=\langle\sigma_i\sigma_j\rangle_{G_{N,a}}.
\]
Conditional independence of the two replicas gives
\[
 \langle(R_N^{12})^2\rangle=N^{-2}\tr(C_1C_2),
 \qquad
 \langle(R_N^{aa'})^2\rangle=N^{-2}\tr(C_a^2),
\]
where $a'$ denotes an independent replica at the same temperature
as $a$. Applying the Hilbert--Schmidt Cauchy--Schwarz inequality
to the matrices, and then Cauchy--Schwarz over the disorder, yields
\[
 \E\langle(R_N^{12})^2\rangle
 \le\left(
 \E\langle(R_N^{11'})^2\rangle
 \E\langle(R_N^{22'})^2\rangle\right)^{1/2}
 \longrightarrow0.
\]
The first factor tends to zero by
Theorem~\zcref[noname]{thm:array}, and the second is at most one.
Interchanging the types proves the case $\mu_2=\delta_0$.
\end{proof}

\begin{proof}[Proof of \zcref{cor:tc-sk}]
For $\beta_1>1$, the support theorem
\cite[Theorem~1.1]{LopattoFRSB} gives
$\supp\mu_a=[0,Q_a]$ with $Q_a>0$ for both temperatures.
Thus \zeqref{eq:tc-common-levels} holds, and
Theorem~\zcref[noname]{thm:tc-mixture} applies. The normalization
in the cited support theorem is the present covariance
$N\beta_a^2R^2/2$; placing $\log2$ in the PDE terminal value
leaves its minimizer unchanged.

For completeness, the scalar identities show that
$\mu_\beta=\delta_0$ when $0<\beta\le1$. Let
$u_\beta,V_\beta,X^\beta$ denote the scalar objects for
$\beta^2q^2/2$ at zero field, and put
$\Gamma_\beta(q)=\E_{\rm sc}u_\beta(q,X_q^\beta)^2$.
Symmetry gives $\Gamma_\beta(0)=0$. By
\zeqref{eq:gamma-derivative} and Lemma~\zcref[noname]{lem:slack},
\[
 \Gamma'_\beta(q)
 =\beta^2\E_{\rm sc}V_\beta(q,X_q^\beta)^2<\beta^2
 \qquad(0<q<1).
\]
To verify strictness, $0<V_\beta\le1-u_\beta^2$, while
$u_\beta(q,\cdot)$ is odd and strictly increasing. The Gaussian
part of $X_q^\beta$ has variance $\beta^2q>0$, and its drift
has absolute value at most $\beta^2q$. Hence
$\Prob(X_q^\beta>1)>0$, and $u_\beta(q,X_q^\beta)>0$
on that event. Thus $V_\beta^2<1$ with positive probability.
Integrating gives $\Gamma_\beta(q)<\beta^2q\le q$ for every
$q>0$. The contact identity \zeqref{eq:contact} excludes all
positive support points, proving $\mu_\beta=\delta_0$.
The final assertion of Theorem~\zcref[noname]{thm:tc-mixture}
now applies whenever $0<\beta_1\le1$.

If $\beta_1=0$, the first replica is uniform on
$\{-1,1\}^N$, so its spin correlation matrix is the identity.
The diagonal entries of the second correlation matrix are one.
Consequently, $\E\langle(R_N^{12})^2\rangle=1/N$, which
completes the proof.
\end{proof}

\appendix
\section{Continuity of the Gaussian Hamiltonian}\label{app:source}

\subsection{A bound on the modulus of continuity}
We bound the expected change in $H_N$ when each coordinate
is moved by at most $\delta$. We first treat finite mixtures.
The constants below depend only on $\xi'(1)+\xi''(1)$ and
are independent of $N$ and $\delta$.

The covariance is
\[
 \E[H_N(m)H_N(n)]=N\xi(N^{-1}m\cdot n),
 \qquad m,n\in[-1,1]^N.
\]
Differentiating in the direction $v$ at both arguments
and then setting $n=m$ gives, with $q=N^{-1}\|m\|^2$,
\[
 \operatorname{Var}(\nabla H_N(m)\cdot v)
 =\xi'(q)\|v\|^2+\frac{\xi''(q)}N(m\cdot v)^2
 \le[\xi'(1)+\xi''(1)]\|v\|^2.
\]
Integrating the directional derivative along the segment
from $n$ to $m$ and applying Cauchy--Schwarz therefore gives
\begin{equation}\label{eq:gaussian-metric}
 \E|H_N(m)-H_N(n)|^2\le C_\xi\|m-n\|^2.
\end{equation}

Fix $0<\delta\le1$, and define the Gaussian process
\[
 T_\delta=
 \{(m,n)\in[-1,1]^N\times[-1,1]^N:
          \|m-n\|_\infty\le\delta\},
 \qquad
 X_{m,n}=H_N(m)-H_N(n).
\]
By \zeqref{eq:gaussian-metric},
\[
 \sup_{(m,n)\in T_\delta}\|X_{m,n}\|_{L^2}
 \le C_\xi\delta\sqrt N.
\]
The distance defined by its increments satisfies
\[
 \begin{split}
 d((m,n),(m',n'))
 &:=\|X_{m,n}-X_{m',n'}\|_{L^2}\le C_\xi\bigl(\|m-m'\|+\|n-n'\|\bigr).
 \end{split}
\]
In particular, after enlarging $C_\xi$, the diameter of
$T_\delta$ in this distance is at most
$C_\xi\delta\sqrt N$.

Let $\mathcal N(T_\delta,d,\varepsilon)$ be the minimum
number of $d$-balls of radius $\varepsilon$ needed to cover
$T_\delta$. A Euclidean grid in each cube, with one point
of $T_\delta$ chosen in every product cell that meets
$T_\delta$, gives
\[
 \log\mathcal N(T_\delta,d,\varepsilon)
 \le 2N\log\left(1+\frac{C_\xi\sqrt N}{\varepsilon}\right).
\]
Since $X_{m,m}=0$ and $X_{n,m}=-X_{m,n}$, the Gaussian
entropy integral \cite[Corollary~5.25]{VanHandel2016}
bounds the supremum of the absolute increments. Taking
$C_\xi\ge2e$ sufficiently large gives
\begin{align}
 \frac1N\E\sup_{(m,n)\in T_\delta}|H_N(m)-H_N(n)|
 &\le\frac{C_\xi}{N}
   \int_0^{C_\xi\delta\sqrt N}
   \sqrt{N\log\left(1+
          \frac{C_\xi\sqrt N}{\varepsilon}\right)}
   \,d\varepsilon
   \notag\\
 &\le C_\xi^2\delta
          \sqrt{\log(C_\xi/\delta)}.
 \label{eq:rounding}
\end{align}
For the last inequality, substitute
$\varepsilon=C_\xi\delta\sqrt N\,t$ and use
\[
 \int_0^1\sqrt{\log(1+1/(\delta t))}\,dt
 \le\sqrt{\log(2/\delta)+1}
 \le\sqrt{\log(C_\xi/\delta)}.
\]

We next construct a continuous version for an infinite
mixture. Let $H_N^{(\le K)}$ denote the sum of the interaction
terms of orders $p\le K$. For $K>M$, put
\[
 T_N^{M,K}=H_N^{(\le K)}-H_N^{(\le M)},\qquad
 a_M=\sum_{p>M}p^2b_p^2.
\]
The directional-derivative calculation above gives
\[
 \operatorname{Var}(\nabla T_N^{M,K}(m)\cdot v)
 \le a_M\|v\|^2.
\]
If $a_M>0$, the normalized process
$T_N^{M,K}/\sqrt{a_M}$ therefore satisfies the preceding
estimates with universal constants. Since
$T_N^{M,K}(0)=0$, applying \zeqref{eq:rounding} with
$\delta=1$ and comparing every $m\in[-1,1]^N$ with the
origin yields
\[
 \E\|T_N^{M,K}\|_{C([-1,1]^N)}
 \le C N\sqrt{a_M},
\]
where $C$ is independent of $M,K,N$. If $a_M=0$, the tail
vanishes identically.

Exponential summability gives $a_M\to0$. Thus, for each
fixed $N$, the partial sums are Cauchy in
$L^1(C([-1,1]^N))$. Their limit is a random continuous
function. At each fixed $m$, the partial sums also converge
in $L^2$ to the Gaussian series defining $H_N(m)$, so this
continuous limit is a version of $H_N$.

The constants in \zeqref{eq:rounding} are uniform over the
partial sums, since their covariance derivatives at one
are bounded by those of the full mixture. Replacing one
continuous function by another changes its modulus of
continuity by at most twice their supremum distance.
The $L^1(C([-1,1]^N))$ convergence therefore passes
\zeqref{eq:rounding} to the full process.

\section{Parameter continuity and convergence of overlaps and added spins}\label{app:cavity}

\subsection{Differentiability in an active interaction coefficient}
Panchenko's differentiability theorem identifies active moments for
finite mixtures \cite{PanDiff}. The following scalar argument allows the
remaining interaction terms to form an infinite mixture satisfying
\zeqref{eq:mixture}.

\begin{lemma}\label{lem:active-derivative}
Fix an even integer $p\ge2$, $a_0>0$, and $h\in\R$.
Let $\xi_{\rm rest}$ be a fixed even mixture satisfying
the summability condition in \zeqref{eq:mixture}, and set
\[
 \xi_a(q)=\xi_{\rm rest}(q)+a^2q^p,\qquad a>0.
\]
Let $P(a)$ be the limiting pressure at field $h$, and let
$\mu_a$ be the Parisi minimizer for $\xi_a$ and $h$.
Then $P$ is differentiable at $a_0$, with
\begin{equation}\label{eq:coefficient-derivative}
 P'(a_0)=a_0\left(1-\int q^p\,\mu_{a_0}(dq)\right).
\end{equation}
\end{lemma}

\begin{proof}
We first differentiate the Parisi functional with its trial
measure fixed. We then evaluate this derivative at the
minimizer and justify differentiation of the minimum.
Throughout the proof, $h$ remains fixed.

For a probability measure $\nu$ on $[0,1]$, put
$\alpha(q)=\nu([0,q])$ and $v_a(q)=\xi_a''(q)$.
Let $\Phi=\Phi_{a,\nu}$ and $X=X^{a,\nu}$ be the scalar
Parisi potential and diffusion for $\xi_a,\nu,h$, and write
\[
 u=\Phi_x,\qquad V=\Phi_{xx},\qquad
 \Gamma(q)=\E[u(q,X_q)^2].
\]
Expectations here are over the scalar diffusion. The
Parisi functional is
\[
 \mathcal P(a,\nu)
 =\Phi_{a,\nu}(0,h)
  -\frac12\int_0^1qv_a(q)\alpha(q)\,dq.
\]

Suppose first that $\nu$ is finitely supported, so that
$\alpha$ is a step function. Differentiating the scalar
recursion in $a$ shows that $w=\partial_a\Phi$ satisfies
\[
 w_q+\tfrac12v_aw_{xx}+v_a\alpha u\,w_x
 =-\tfrac12\dot v_a(V+\alpha u^2),\qquad w(1,x)=0,
 \qquad \dot v_a(q)=2ap(p-1)q^{p-2}.
\]
The forcing term is bounded. Feynman--Kac evaluates
$w(0,h)$, and differentiation of the integral in
$\mathcal P(a,\nu)$ gives
\begin{equation}\label{eq:fixed-measure-derivative}
 \partial_a\mathcal P(a,\nu)
 =\frac12\int_0^1\dot v_a(q)
   [\E V(q,X_q)+\alpha(q)(\Gamma(q)-q)]\,dq.
\end{equation}
In particular, the term $-q\alpha(q)$ comes from
differentiating that integral.

For an arbitrary trial measure, denote the right side
of \zeqref{eq:fixed-measure-derivative} by $D(a,\nu)$.
We next prove that $D$ is jointly continuous and equals
$\partial_a\mathcal P(a,\nu)$ for every $\nu$.

Fix a compact interval $I\subset(0,\infty)$ containing
$a_0$ in its interior. On this interval, $v_a$ and
$\dot v_a$ are uniformly bounded. The scalar estimates
\[
 |u|\le1,\qquad 0<V\le1,\qquad
 \sup_{q,x}|\Phi_{xxx}(q,x)|\le C
\]
hold uniformly over $a\in I$ and all trial measures
\cite{ACunique}.

Let $a_n\to a$ in $I$ and $\nu_n\Rightarrow\nu$. Use
subscripts $n$ for the associated scalar quantities, and
write $v_n=v_{a_n}$. Then
\[
 \alpha_n\to\alpha\quad\text{in }L^1([0,1]),
 \qquad v_n\to v_a\quad\text{uniformly on }[0,1].
\]
Subtracting the two Parisi PDEs gives
\[
 \begin{split}
 \bigl(\partial_q+\tfrac12v_n\partial_{xx}
       +\tfrac12v_n\alpha_n(u_n+u)\partial_x\bigr)
       (\Phi_n-\Phi)
 ={}&-\tfrac12(v_n-v_a)V\\
    &-\tfrac12(v_n\alpha_n-v_a\alpha)u^2,
 \end{split}
\]
with zero terminal difference. The drift is bounded,
and the spatial supremum of the forcing term is at most
\[
 C\bigl(|v_n(q)-v_a(q)|
          +|\alpha_n(q)-\alpha(q)|\bigr).
\]
Feynman--Kac therefore gives
\[
 \|\Phi_n-\Phi\|_\infty
 \le C\bigl(\|v_n-v_a\|_{L^1([0,1])}
             +\|\alpha_n-\alpha\|_{L^1([0,1])}\bigr)
 \longrightarrow0.
\]
Here and below, the supremum norms of the potentials and
their spatial derivatives are over $[0,1]\times\R$.

The uniform derivative bounds turn this convergence into
uniform convergence of $u_n$ and $V_n$. More explicitly,
forward and centered differences with step $\rho>0$ give
\[
 \begin{split}
 \|u_n-u\|_\infty
 &\le 2\rho^{-1}\|\Phi_n-\Phi\|_\infty+C\rho,\\
 \|V_n-V\|_\infty
 &\le 4\rho^{-2}\|\Phi_n-\Phi\|_\infty+C\rho.
 \end{split}
\]
First let $n\to\infty$ and then let $\rho\downarrow0$.

Couple $X^n$ and $X$ using the same Brownian motion and
the same initial value $h$. Their drifts have a common
spatial Lipschitz constant, and
\[
 \int_0^1\sup_x
 |v_n(q)\alpha_n(q)u_n(q,x)
       -v_a(q)\alpha(q)u(q,x)|\,dq\longrightarrow0.
\]
Their diffusion coefficients satisfy
\[
 \int_0^1|\sqrt{v_n(q)}-\sqrt{v_a(q)}|^2\,dq
 \le\int_0^1|v_n(q)-v_a(q)|\,dq\longrightarrow0.
\]
Doob's inequality and Gronwall's inequality consequently give
\[
 \E\sup_{0\le q\le1}|X_q^n-X_q|^2\longrightarrow0.
\]
Together with the uniform convergence of $u_n,V_n$ and
the $L^1$ convergence of $\alpha_n$, this proves joint
continuity of $D(a,\nu)$. The potential estimate and the
formula for $\mathcal P$ also prove joint continuity of
$\mathcal P(a,\nu)$.

To establish the derivative formula for a general $\nu$,
choose finitely supported measures $\nu_j\Rightarrow\nu$.
For $a_1,a_2\in I$, the formula already proved gives
\[
 \mathcal P(a_2,\nu_j)-\mathcal P(a_1,\nu_j)
 =\int_{a_1}^{a_2}D(s,\nu_j)\,ds.
\]
The preceding continuity estimates and uniform bounds
permit passage to the limit in this identity. Thus
\[
 \mathcal P(a_2,\nu)-\mathcal P(a_1,\nu)
 =\int_{a_1}^{a_2}D(s,\nu)\,ds.
\]
Since $D(\cdot,\nu)$ is continuous, this proves
\zeqref{eq:fixed-measure-derivative} for every trial measure.

We now fix $a\in I$ and evaluate $D(a,\mu_a)$. All scalar
quantities in the following calculation use $\nu=\mu_a$.
Define
\[
 F(q)=\E V(q,X_q)+\alpha(q)(\Gamma(q)-q).
\]
It\^o's formula for $u(q,X_q)$ and $V(q,X_q)$ gives
\[
 \Gamma'(q)=v_a(q)\E[V(q,X_q)^2],\qquad
 d\,\E V(q,X_q)=-\alpha(q)\Gamma'(q)\,dq.
\]
The Stieltjes product rule therefore yields on $(0,1)$
\begin{equation}\label{eq:derivative-stieltjes}
 dF(q)
 =-\alpha(q)\,dq+(\Gamma(q)-q)\,\mu_a(dq)
 =-\alpha(q)\,dq.
\end{equation}
The last equality uses the contact identity
$\Gamma(q)=q$ on $\supp\mu_a$.

At time one, $V(1,x)=1-u(1,x)^2$ and $\alpha(1)=1$,
so $F(1)=0$. Moreover, $\Gamma(1)<1$ and the contact
identity imply $1\notin\supp\mu_a$. Thus $F$ has no jump
at one, and integrating \zeqref{eq:derivative-stieltjes}
gives
\[
 F(q)=\int_q^1\alpha(s)\,ds
 \qquad\text{for almost every }q\in[0,1].
\]
Substituting this expression and changing the order of
integration gives
\begin{align}
 D(a,\mu_a)
 &=ap(p-1)\int_0^1q^{p-2}
                  \int_q^1\alpha(s)\,ds\,dq\notag\\
 &=ap\int_0^1s^{p-1}\alpha(s)\,ds
   =a\left(1-\int t^p\,\mu_a(dt)\right).
 \label{eq:derivative-minimizer}
\end{align}
For the last equality, write
$\alpha(s)=\int\one_{\{t\le s\}}\,\mu_a(dt)$ and use
$\int_t^1ps^{p-1}\,ds=1-t^p$. This calculation includes
any atom at zero. The measure $\mu_a$ is held fixed
throughout this evaluation of the partial derivative.

Finally, we pass from the Parisi functional to
its minimum over trial measures. Weak compactness of the trial measures and
joint continuity of $\mathcal P$ imply that every
subsequential limit of $\mu_{a_n}$, for $a_n\to a_0$,
minimizes $\mathcal P(a_0,\cdot)$. Uniqueness of the
minimizer therefore gives
\[
 \mu_{a_n}\Rightarrow\mu_{a_0}.
\]
For either sign of $t$, with $a_0+t\in I$, use
$\mu_{a_0+t}$ as a trial measure at $a_0$ and
$\mu_{a_0}$ as a trial measure at $a_0+t$. Minimality
and the fundamental theorem of calculus with the trial measure held fixed give
\[
 \int_{a_0}^{a_0+t}D(s,\mu_{a_0+t})\,ds
 \le P(a_0+t)-P(a_0)
 \le\int_{a_0}^{a_0+t}D(s,\mu_{a_0})\,ds.
\]
Joint continuity of $D$ and convergence of the minimizers
show that both bounding integrals equal
\[
 tD(a_0,\mu_{a_0})+o(|t|).
\]
Hence $P'(a_0)=D(a_0,\mu_{a_0})$, and
\zeqref{eq:derivative-minimizer} proves
\zeqref{eq:coefficient-derivative}.
\end{proof}

\begin{corollary}\label{cor:active-moments}
Let $\mu$ be the Parisi minimizer for the fixed mixture
and field. For every even $p\ge2$ with $b_p>0$,
\[
 \lim_{N\to\infty}\E\langle R_{12}^p\rangle_{N,h}
 =\int q^p\,\mu(dq).
\]
Moreover,
\[
 \lim_{N\to\infty}
 \E\langle R_{12}\kappa(R_{12})\rangle_{N,h}
 =\int q\kappa(q)\,\mu(dq).
\]
\end{corollary}

\begin{proof}
Fix an active order $p$ and put $a_0=b_p$. Write $p_N(a)$
and $P(a)$ for the averaged finite-volume pressure and its
limit when $b_p$ is replaced by $a$, with all other
coefficients and $h$ fixed. Convexity gives, for $0<t<a_0$,
\[
 \frac{p_N(a_0)-p_N(a_0-t)}{t}
 \le p_N'(a_0)
 \le\frac{p_N(a_0+t)-p_N(a_0)}{t}.
\]
First let $N\to\infty$, using the Parisi formula, and then
let $t\downarrow0$. Differentiability of $P$ from
Lemma~\zcref[noname]{lem:active-derivative} implies
$p_N'(a_0)\to P'(a_0)$. Gaussian integration by parts
therefore gives
\[
 a_0\bigl(1-\E\langle R_{12}^p\rangle_{N,h}\bigr)
 =p_N'(a_0)
 \longrightarrow
 a_0\left(1-\int q^p\,\mu(dq)\right).
\]
Since $a_0>0$, this proves the first assertion.

For the second assertion, the series
\[
 r\kappa(r)=\sum_{k\in2\mathbb N}k b_k^2r^k
\]
converges uniformly on $[-1,1]$, because
$\sum_k k b_k^2<\infty$. Apply the first assertion to each
finite partial sum. The uniform bound on the remaining
tail then permits the cutoff to tend to infinity.
\end{proof}

\subsection{Coefficient changes under one-site deletion}
We prove the relative-entropy estimate used in the proof of
Lemma~\zcref[noname]{cav:representation}. Recall that the reduced
$N$-site Hamiltonian has coefficients
\begin{equation}\label{eq:core-coefficients}
 b_{p,N}=b_p\left(\frac N{N+1}\right)^{(p-1)/2}.
\end{equation}
Write $\mathbf b=(b_p)_p$ and $\mathbf b_N=(b_{p,N})_p$,
and couple the Gibbs measures $G_{\mathbf b}$ and
$G_{\mathbf b_N}$ using the same Gaussian variables and field.
We display the coefficient vector as an argument in the averaged
pressure $p_N(\mathbf a)$ and write $\partial_p$ for differentiation
with respect to $a_p$. All sums over interaction orders below
run over even integers at least two.

For probability measures on the spin configurations, define
\[
 D(P\Vert Q)=\sum_\sigma P(\sigma)
                 \log\frac{P(\sigma)}{Q(\sigma)}.
\]
Let $H_{N,p}$ be the order-$p$ Hamiltonian with coefficient one,
so that $H_{N,\mathbf a}=\sum_p a_pH_{N,p}$. Put
\[
 \Delta H_N=H_{N,\mathbf b_N}-H_{N,\mathbf b}
           =\sum_p(b_{p,N}-b_p)H_{N,p}.
\]
Expanding the two relative entropies cancels their
log-partition terms and gives
\[
 D(G_{\mathbf b}\Vert G_{\mathbf b_N})
 +D(G_{\mathbf b_N}\Vert G_{\mathbf b})
 =G_{\mathbf b_N}(\Delta H_N)-G_{\mathbf b}(\Delta H_N).
\]
Since $\E G_{\mathbf a}(H_{N,p})=N\partial_p p_N(\mathbf a)$,
averaging yields
\begin{equation}\label{eq:core-entropy}
 \begin{aligned}
 &\E\bigl[D(G_{\mathbf b}\Vert G_{\mathbf b_N})
          +D(G_{\mathbf b_N}\Vert G_{\mathbf b})\bigr]=N\sum_p(b_{p,N}-b_p)
       \bigl[\partial_p p_N(\mathbf b_N)
                    -\partial_p p_N(\mathbf b)\bigr].
 \end{aligned}
\end{equation}
This calculation also applies to infinite mixtures. Indeed, for
each fixed $N$, the Gaussian maximal inequality gives
$\E\max_\sigma|H_{N,p}(\sigma)|\le CN$, uniformly in $p$,
because there are $2^N$ configurations and each field value has
variance $N$. Moreover, $\sum_p b_p<\infty$ by Cauchy--Schwarz
and \zeqref{eq:mixture}. Thus
\[
 \sum_p|b_{p,N}-b_p|\,
       \E\max_\sigma|H_{N,p}(\sigma)|
 \le CN\sum_p b_p<\infty,
\]
which justifies termwise averaging.

We next show that, for every fixed active $p$, the two pressure
derivatives in \zeqref{eq:core-entropy} have the same limit.
Gaussian integration by parts gives, for nonnegative coefficients,
\[
 \partial_p p_N(\mathbf a)
 =a_p\bigl(1-\E\langle R_{12}^p\rangle_{\mathbf a}\bigr),
 \qquad 0\le\partial_p p_N(\mathbf a)\le a_p.
\]
At the fixed vector $\mathbf b$, convergence to
$\partial_pP(\mathbf b)$ follows from
Corollary~\zcref[noname]{cor:active-moments} and
Lemma~\zcref[noname]{lem:active-derivative}.
It remains to prove convergence at the moving vector $\mathbf b_N$.

Let $e_p$ be the coefficient vector with entry one at order $p$
and zero elsewhere. Fix $|t|<b_p/2$. For sufficiently large $N$,
both $\mathbf b_N+t e_p$ and $\mathbf b+t e_p$ have nonnegative
coefficients. Gaussian interpolation between their covariance
functions gives
\[
 \begin{aligned}
 &|p_N(\mathbf b_N+t e_p)-p_N(\mathbf b+t e_p)|\\
 &\qquad\le\frac12\sum_j|b_{j,N}^2-b_j^2|
               +|t|\,|b_{p,N}-b_p|
 \longrightarrow0.
 \end{aligned}
\]
The term involving $t$ comes from expanding the difference
$(b_{p,N}+t)^2-(b_p+t)^2$. The sum tends to zero by dominated
convergence, since $0\le b_{j,N}\le b_j$ and $\sum_jb_j^2<\infty$.
The Parisi formula at the fixed vector $\mathbf b+t e_p$
therefore implies
\[
 p_N(\mathbf b_N+t e_p)\longrightarrow P(\mathbf b+t e_p).
\]

For $0<t<b_p/2$, convexity in the $p$-th coefficient gives
\[
 \frac{p_N(\mathbf b_N)-p_N(\mathbf b_N-t e_p)}{t}
 \le\partial_p p_N(\mathbf b_N)
 \le\frac{p_N(\mathbf b_N+t e_p)-p_N(\mathbf b_N)}{t}.
\]
First let $N\to\infty$ with $t$ fixed. The bounding quotients
converge to the backward and forward difference quotients of
$P$ at $\mathbf b$. Then let $t\downarrow0$ and use
Lemma~\zcref[noname]{lem:active-derivative} to obtain
\[
 \partial_p p_N(\mathbf b_N)\longrightarrow\partial_pP(\mathbf b).
\]

Finally, the inequalities $1-e^{-x}\le x$ and
$\log(1+1/N)\le1/N$ give
\[
 N|b_{p,N}-b_p|\le\frac{p-1}{2}b_p.
\]
The derivative bound above and $b_{p,N}\le b_p$ consequently imply
\[
 N|b_{p,N}-b_p|\,
 \bigl|\partial_p p_N(\mathbf b_N)-\partial_p p_N(\mathbf b)\bigr|
 \le(p-1)b_p^2.
\]
This bound is summable by \zeqref{eq:mixture}. For each active
$p$, the summand tends to zero because the derivative difference
tends to zero and $N|b_{p,N}-b_p|$ is bounded. For an inactive
$p$, it vanishes identically. Dominated convergence in
\zeqref{eq:core-entropy} therefore proves
\[
 \E\bigl[D(G_{\mathbf b}\Vert G_{\mathbf b_N})
          +D(G_{\mathbf b_N}\Vert G_{\mathbf b})\bigr]
 \longrightarrow0.
\]

\subsection{Repeated indices in one-site deletion}
For $\sigma\in\{-1,1\}^N$ and $\epsilon\in\{-1,1\}$,
separate the interaction terms according to the number of
occurrences of the index $N+1$. Write
\[
 H_{N+1}(\sigma,\epsilon)
 =H_N^{\rm red}(\sigma)+E_N(\sigma)
   +\epsilon z_N(\sigma)+\epsilon O_N(\sigma).
\]
Here $H_N^{\rm red}$ contains the terms with no occurrence,
$z_N$ contains those with exactly one, $E_N$ contains those
with an even number of occurrences at least two, and
$O_N$ contains those with an odd number at least three.
These four centered Gaussian processes are independent
because they use disjoint sets of Gaussian coefficients.
The reduced Hamiltonian $H_N^{\rm red}$ has the coefficients
in \zeqref{eq:core-coefficients}.

Put $a_N=N/(N+1)$, $d_N=1/(N+1)$, and
$R=N^{-1}\sigma\cdot\tau$. The covariances of
$H_N^{\rm red}+E_N$ and $z_N+O_N$ are, respectively,
\begin{align}
 C_N^e(R)
 &=\frac{N+1}{2}
   [\xi(a_NR+d_N)+\xi(a_NR-d_N)],\notag\\
 C_N^o(R)
 &=\frac{N+1}{2}
   [\xi(a_NR+d_N)-\xi(a_NR-d_N)].
 \label{eq:parity-covariances}
\end{align}
Indeed, adding the two expansions retains even powers
of $d_N$, while subtracting them retains odd powers.
The terms of degrees zero and one in $d_N$ give
\[
 \E[H_N^{\rm red}(\sigma)H_N^{\rm red}(\tau)]
 =(N+1)\xi(a_NR),\qquad
 \E[z_N(\sigma)z_N(\tau)]=\kappa(a_NR).
\]

We bound the variances of the remaining terms by Taylor
expansion at $a_N$. The nonnegative mixture coefficients
bound the second and third derivatives of $\xi$ on
$[0,1]$ by their values at one. Hence
\begin{equation}\label{eq:parity-errors}
 v_N^e:=\operatorname{Var}(E_N(\sigma))
 =C_N^e(1)-(N+1)\xi(a_N)
 \le\frac{\xi''(1)}{2(N+1)},
\end{equation}
and
\[
 v_N^{o,\ge3}:=\operatorname{Var}(O_N(\sigma))
 =C_N^o(1)-\kappa(a_N)
 \le\frac{\xi'''(1)}{6(N+1)^2}.
\]
These diagonal variances do not depend on $\sigma$.

To translate these variance bounds into comparisons of
Gibbs measures, let $G$ be a random Gibbs measure on a
finite configuration space and let $W$ be an independent
centered Gaussian process on that space, with
$\sup_x\operatorname{Var}(W(x))\le v$. Define
\[
 G^W(x)=\frac{e^{W(x)}G(x)}{G e^W}.
\]
The relative entropy satisfies
\[
 D(G\Vert G^W)=\log G e^W-GW.
\]
Conditional on $G$, the second term has mean zero.
Jensen's inequality and Gaussian exponential moments give
\[
 \E_W D(G\Vert G^W)
 =\E_W\log G e^W
 \le\log\sum_xG(x)e^{\operatorname{Var}(W(x))/2}
 \le\frac v2.
\]
Averaging over $G$ and applying Pinsker's inequality yields
\[
 \E\|G-G^W\|_{\rm TV}
 \le\left(\frac12\E D(G\Vert G^W)\right)^{1/2}
 \le\frac{\sqrt v}{2}.
\]

Apply this estimate on the full configuration space
$(\sigma,\epsilon)$. The remainder
$E_N(\sigma)+\epsilon O_N(\sigma)$ is independent of
$H_N^{\rm red}+\epsilon z_N$ and has variance
$v_N^e+v_N^{o,\ge3}$. Thus, with the external field unchanged,
removing this remainder changes the Gibbs measure by at most
$\frac12\sqrt{v_N^e+v_N^{o,\ge3}}=O_\xi(N^{-1/2})$
in expected total variation.

\subsection{Joint convergence of overlaps and added spins}
\label{app:cavity-joint}
We prove the joint convergence used in the proof of
Lemma~\zcref[noname]{cav:representation}. We approximate the
normalizing constant using finitely many auxiliary replicas,
pass to the overlap limit, and then remove the approximation.

Write $G_{N,h}^{\rm red}=G_{\mathbf b_N}$ and $L=\kappa(1)$.
Conditional on the reduced disorder, the field $z_N$ is centered
Gaussian with covariance
\[
 \E[z_N(\sigma)z_N(\tau)\mid G_{N,h}^{\rm red}]
 =\kappa\left(\frac{N}{N+1}R(\sigma,\tau)\right).
\]
In particular, its diagonal variances are at most $L$.
Fix $n\ge1$, signs $e_1,\ldots,e_n\in\{-1,1\}$, and a
bounded continuous function $F$ of the overlaps among these
$n$ replicas. Every replica label used by $F$ is included in $n$.
We first prescribe all $n$ added-site signs and sum over any
unspecified signs at the end.

Conditional on the reduced disorder and $z_N$, sample
$\sigma^1,\ldots,\sigma^n$ independently from $G_{N,h}^{\rm red}$.
Put
\[
 \mathcal Z_N=G_{N,h}^{\rm red}[\cosh(h+z_N)]\ge1,
 \qquad
 A_N=2^{-n}\exp\left\{
       \sum_{i=1}^n e_i(h+z_N(\sigma^i))\right\}.
\]
The finite-system sampling formula in the proof of
Lemma~\zcref[noname]{cav:representation} gives the expectation
with these prescribed signs as
\[
 I_N=\E[F(R)A_N\mathcal Z_N^{-n}].
\]
Here the repeated-index remainder has been removed, as in that
proof. Throughout this subsection, $\E$ averages over the reduced
disorder, the Gaussian field, and replicas sampled from the reduced
Gibbs measure before cavity reweighting.

To approximate $\mathcal Z_N$, sample $m$ additional replicas
$\tau^1,\ldots,\tau^m$ from $G_{N,h}^{\rm red}$, conditionally
independently of one another and of the displayed replicas given
the reduced disorder and $z_N$. Define
\[
 \widehat{\mathcal Z}_{N,m}
 =\frac1m\sum_{j=1}^m\cosh(h+z_N(\tau^j)),
 \qquad
 I_{N,m}=\E[F(R)A_N\widehat{\mathcal Z}_{N,m}^{-n}].
\]
The conditional sample mean has mean $\mathcal Z_N$, so its
conditional variance gives
\[
 \E\left[
   |\widehat{\mathcal Z}_{N,m}-\mathcal Z_N|^2
   \,\middle|\,G_{N,h}^{\rm red},z_N\right]
 \le\frac1mG_{N,h}^{\rm red}[\cosh^2(h+z_N)].
\]
For a centered Gaussian variable $Z$ of variance at most $L$,
\[
 \E\cosh^2(h+Z)
 =\frac{1+e^{2\Var(Z)}\cosh(2h)}2
 \le\frac{1+e^{2L}\cosh(2h)}2.
\]
Averaging first over the field and then over the reduced disorder
therefore yields
\[
 \E|\widehat{\mathcal Z}_{N,m}-\mathcal Z_N|^2
 \le\frac{C_{h,L}}m.
\]

We also need a moment bound for the numerator. Conditional on the
reduced disorder and the displayed configurations, the signed sum
$\sum_{i=1}^n e_i z_N(\sigma^i)$ is centered Gaussian with variance
at most $n^2L$. Hence
\[
 \E A_N^2\le4^{-n}\exp\{2n|h|+2n^2L\}.
\]
Both normalizing constants are at least one, and
\[
 |x^{-n}-y^{-n}|\le n|x-y|
 \qquad (x,y\ge1).
\]
Cauchy--Schwarz and the preceding moment bounds therefore give
\[
 \begin{aligned}
 |I_{N,m}-I_N|
 &\le n\|F\|_\infty
       \E\bigl[A_N
          |\widehat{\mathcal Z}_{N,m}-\mathcal Z_N|\bigr]\\
 &\le n\|F\|_\infty
       \bigl(\E A_N^2\bigr)^{1/2}
       \bigl(\E|\widehat{\mathcal Z}_{N,m}
                    -\mathcal Z_N|^2\bigr)^{1/2}\\
 &\le C_{h,L,n}\|F\|_\infty m^{-1/2}.
 \end{aligned}
\]
The constant is independent of $N$ and of the reduced Gibbs measure.

Fix $m$ and condition on all $n+m$ reduced replicas. Their field
values form a Gaussian vector whose covariance depends only on
their overlap matrix, including the diagonal entries $R_{ii}=1$.
Integrating this vector therefore expresses $I_{N,m}$ as an
expectation of a function of finitely many overlaps. We verify the
continuity needed to pass this expression to the overlap limit.

Let $K_\ell\to K$ be covariance matrices of size $n+m$ with
diagonal entries at most $L$, and let $g$ be a standard Gaussian
vector of that dimension. Couple the associated Gaussian vectors by
\[
 Z^{(\ell)}=K_\ell^{1/2}g,
 \qquad Z=K^{1/2}g.
\]
Continuity of the positive semidefinite square root gives
$Z^{(\ell)}\to Z$ almost surely, including when $K$ is singular.
The integrand defining $I_{N,m}$ is continuous in these field
coordinates. Its denominator is at least one, and the numerator
has the uniform second-moment bound proved above. The Gaussian
integrals consequently converge by uniform integrability.

The covariance matrices with diagonal entries at most $L$ form
a compact set, so this continuity is uniform. Since
\[
 \kappa\left(\frac{N}{N+1}r\right)\longrightarrow\kappa(r)
 \quad\text{uniformly for }r\in[-1,1],
\]
the functions obtained by integrating over the Gaussian field converge uniformly to bounded
continuous functions of the finite overlap matrix. Along a
subsequence on which the reduced overlap arrays converge, weak
convergence therefore gives
\[
 I_{N,m}\longrightarrow I_m
 \qquad\text{for each fixed }m.
\]

It remains to identify the limit as $m\to\infty$. Use the Gaussian
representation described in the proof of
Lemma~\zcref[noname]{cav:representation}. Conditional on a predecessor
directing measure $G$, sample the Gaussian field $Y$. Conditional on
$G,Y$, sample the reference vectors independently from $G$.
Each limiting field coordinate is represented as $Y(v)+\eta$, where
\[
 \E[Y(v)Y(w)\mid G]=\kappa(v\cdot w),
 \qquad
 \eta\sim N\bigl(0,L-\kappa(\|v\|^2)\bigr).
\]
Conditional on $G,Y$ and the vectors, the additional noises are
independent for distinct replica labels. Let $w^j$ and $\eta_j$
be the vectors and noises for the auxiliary replicas, and put
\[
 W_j=\cosh(h+Y(w^j)+\eta_j).
\]
Conditional on $G,Y$, the variables $W_j$ are independent and
identically distributed. The Gaussian average computed in the
proof of Lemma~\zcref[noname]{cav:representation} gives their
common conditional mean
\[
 \E[W_j\mid G,Y]
 =G[a\cosh(h+Y)]=Z_h.
\]
Moreover, $\E W_j^2\le C_{h,L}$ by the Gaussian moment bound above,
so this conditional mean is finite almost surely. The conditional
strong law therefore gives
\[
 \frac1m\sum_{j=1}^m W_j\longrightarrow Z_h
 \quad\text{almost surely}.
\]

For the displayed replicas, write
\[
 A=2^{-n}\exp\left\{
   \sum_{i=1}^n e_i(h+Y(v^i)+\eta_i)\right\}.
\]
The representation of $I_m$ uses this numerator and the auxiliary
sample mean in its denominator. Since the sample mean is at least
one and $\E A^2\le4^{-n}e^{2n|h|+2n^2L}$, dominated convergence
gives
\[
 I_m\longrightarrow
 I=\E\bigl[F((v^i\cdot v^j)_{i<j})A Z_h^{-n}\bigr].
\]
Integrating the independent noises of the displayed replicas
is the calculation in the proof of
Lemma~\zcref[noname]{cav:representation} that yields
\zeqref{eq:cavity-tilt} and \zeqref{eq:cavity-mark}.

Finally, the uniform approximation bound implies
\[
 \limsup_{N\to\infty}|I_N-I|
 \le C_{h,L,n}\|F\|_\infty m^{-1/2}+|I_m-I|.
\]
Letting $m\to\infty$ proves the required convergence. Thus the
thermodynamic limit is taken with a fixed number of auxiliary
replicas, and that number is sent to infinity afterward.
Summing over unspecified signs proves joint convergence with
any fixed collection of added spins. Indicators of events depending on finitely many overlaps may also be used when their boundaries have zero probability
under the limiting successor overlap law.

\bibliographystyle{amsplain}
\begingroup
\raggedright
\bibliography{manuscript}
\endgroup
\end{document}